\documentclass[oneside]{amsart}
\usepackage[T1]{fontenc}
\usepackage[latin9]{inputenc}
\usepackage{geometry}
\usepackage{mathrsfs}
\usepackage{mathtools}
\usepackage{bm}
\usepackage{amsbsy}
\usepackage{amstext}
\usepackage{amsthm}
\usepackage{amssymb}
\usepackage[all]{xy}

\makeatletter
\numberwithin{equation}{section}
\numberwithin{figure}{section}
\theoremstyle{plain}
\newtheorem*{thm*}{\protect\theoremname}
\theoremstyle{plain}
\newtheorem{thm}{\protect\theoremname}[section]
\theoremstyle{remark}
\newtheorem{rem}[thm]{\protect\remarkname}
\theoremstyle{definition}
\newtheorem{defn}[thm]{\protect\definitionname}
\theoremstyle{plain}
\newtheorem{lem}[thm]{\protect\lemmaname}
\theoremstyle{definition}
\newtheorem{example}[thm]{\protect\examplename}
\theoremstyle{plain}
\newtheorem{prop}[thm]{\protect\propositionname}
\theoremstyle{plain}
\newtheorem{cor}[thm]{\protect\corollaryname}

\usepackage{blindtext}

\makeatletter
\renewcommand\part{%
   \if@noskipsec \leavevmode \fi
   \par
   \addvspace{4ex}%
   \@afterindentfalse
   \secdef\@part\@spart}

\def\@part[#1]#2{%
    \ifnum \c@secnumdepth >\m@ne
      \refstepcounter{part}%
      \addcontentsline{toc}{part}{\thepart\hspace{1em}#1}%
    \else
      \addcontentsline{toc}{part}{#1}%
    \fi
    {\parindent \z@ \raggedright
     \interlinepenalty \@M
     \normalfont
     \ifnum \c@secnumdepth >\m@ne
       \Large\bfseries \partname\nobreakspace\thepart
       \par\nobreak
     \fi
     \huge \bfseries #2%
     \par}%
    \nobreak
    \vskip 3ex
    \@afterheading}
\def\@spart#1{%
    {\parindent \z@ \raggedright
     \interlinepenalty \@M
     \normalfont
     \huge \bfseries #1\par}%
     \nobreak
     \vskip 3ex
     \@afterheading}
\makeatother

\usepackage{shuffle}
\usepackage[dvipdfmx]{graphicx}
\usepackage{pgfplots}
\usepackage{tikz}

\usetikzlibrary{decorations.markings}
\tikzset{->-/.style={decoration={
  markings,
  mark=at position #1 with {\arrow{>}}},postaction={decorate}}}

\makeatother

\providecommand{\corollaryname}{Corollary}
\providecommand{\definitionname}{Definition}
\providecommand{\examplename}{Example}
\providecommand{\lemmaname}{Lemma}
\providecommand{\propositionname}{Proposition}
\providecommand{\remarkname}{Remark}
\providecommand{\theoremname}{Theorem}

\begin{document}
\address[Minoru Hirose]{Institute for Advanced Research, Nagoya University,  Furo-cho, Chikusa-ku, Nagoya, 464-8602, Japan}
\email{minoru.hirose@math.nagoya-u.ac.jp}
\address[Nobuo Sato]{Department of Mathematics, National Taiwan University, No. 1, Sec. 4, Roosevelt Rd., Taipei 10617, Taiwan (R.O.C.)}
\email{nbsato@ntu.edu.tw}
\title[{The motivic Galois group of mixed Tate motives over $\mathbb{Z}[1/2]$}]{The motivic Galois group of mixed Tate motives over $\mathbb{Z}[1/2]$
and its action on the fundamental group of $\mathbb{P}^{1}\setminus\{0,\pm1,\infty\}$}
\author{Minoru Hirose and Nobuo Sato}
\date{\today}
\begin{abstract}
In this paper we introduce confluence relations for motivic Euler
sums (also called alternating multiple zeta values) and show that
all linear relations among motivic Euler sums are exhausted by the
confluence relations. This determines all automorphisms of the de
Rham fundamental groupoid of $\mathbb{P}^{1}\setminus\{0,\pm1,\infty\}$
coming from the action of the motivic Galois group of mixed Tate motives
over $\mathbb{Z}[1/2]$. Moreover, we also discuss other applications
of the confluence relations such as an explicit $\mathbb{Q}$-linear
expansion of a given motivic Euler sum by their basis and $2$-adic
integrality of the coefficients in the expansion.
\end{abstract}

\maketitle
\global\long\def\pio#1#2{{_{#1}\Pi_{#2}}}%
\global\long\def\dpath#1#2#3{{_{#1}#2_{#3}}}%
\global\long\def\depth{{\rm dep}}%
\global\long\def\word{{\rm w}}%
\global\long\def\modzeta{\tilde{\zeta}}%
\global\long\def\modword{\tilde{{\rm w}}}%
\global\long\def\dist{{\rm dist}}%
\global\long\def\Reg{\mathop{{\rm Reg}}}%
\global\long\def\reg{\mathop{{\rm reg}}}%
\global\long\def\barreg{\mathop{{\rm \overline{reg}}}}%
\global\long\def\shreg{{\rm reg}_{\shuffle}}%
\global\long\def\iconf{\mathcal{I}_{\mathrm{CF}}}%
\global\long\def\restr#1#2{{\left.#1\right|_{#2}}}%
\global\long\def\indset{\mathfrak{I}}%
\global\long\def\hconf{\mathcal{H}_{{\rm CF}}}%
\global\long\def\vspan{{\rm span}}%
\global\long\def\modx{\tilde{X}}%
\global\long\def\formalit#1#2#3#4{\mathbb{I}_{#1}(#2;#3;#4)}%
\global\long\def\evallim#1{{\rm ev}_{#1}^{\mathfrak{m}}}%
\global\long\def\compmap{{\rm per}}%
\global\long\def\history#1#2{\mathcal{F}(#1,#2)}%
\global\long\def\jhistory#1#2{\mathcal{G}(#1,#2)}%
\global\long\def\jexthistory#1#2#3#4{\mathcal{G}_{#1}^{#2}(#3,#4)}%
\global\long\def\monohistory#1{\mathcal{F}(#1)}%
\global\long\def\subweak{\prec}%
\global\long\def\subweakneq{\precneqq}%
\global\long\def\substrong{\lhd}%

\tableofcontents{}

\section{Introduction}

\subsection{Basic idea of confluence relations}

In a previous article of the authors \cite{HS_confluence}, we introduced
a class of rational linear relations among multiple zeta values which
we call `confluence relations', and discussed their significance.
The confluence relations are later proven by Furusho \cite{Furusho_confluence}
to be equivalent to Drinfeld's pentagon equation of the KZ-associator.
The purpose of this article is to generalize and develop the theory
of confluence relations for Euler sums and show that they shed new
light on understanding the motivic Galois actions of a class of mixed
Tate motives over $\mathbb{Z}[1/2]$ on the fundamental group of the
$\{0,1,-1,\infty\}$-punctured projective line. 

To give a little more details, we consider a complex function of $z$
defined by
\[
I(p_{0}(z);p_{1}(z),\dots,p_{k}(z);p_{k+1}(z))
\]
where $I$ is the iterated integral symbol and $p_{0},\dots,p_{k+1}$
are polynomials in $z$. By making use of Goncharov's differential
formula
\begin{align*}
 & \frac{d}{dz}I(p_{0}(z);p_{1}(z),\dots,p_{k}(z);p_{k+1}(z))\\
 & =\sum_{\varepsilon\in\{1,-1\}}\varepsilon\sum_{\substack{1\leq i\leq k\\
p_{i}\neq p_{i+\varepsilon}
}
}\frac{d\log(p_{i}(z)-p_{i+\varepsilon}(z))}{dz}I(p_{0}(z);p_{1}(z),\dots,p_{i-1}(z),p_{i+1}(z),\dots,p_{k}(z);p_{k+1}(z)),
\end{align*}
for a fixed complex number $x$, we can expand $I(p_{0}(z);p_{1}(z),\dots,p_{k}(z);p_{k+1}(z))$
as a $\mathbb{C}$-linear sum of the functions of the form
\[
I(x;a_{1},\dots,a_{l};z)\ \ \ (a_{1},\dots,a_{l}\in\mathbb{C})
\]
in a canonical way. Then, by passing the ``limit'' $z$ (in some generalized
sense) to some point $y\in\mathbb{C}$ of this expansion, we obtain
relations among iterated integrals which we refer to as the `confluence
relations' (of general type). In a nutshell, confluence relations
are the relations of special values of hyperlogarithms that arise
from the system of differential equations satisfied by hyperlogarithms.
The aforementioned article of the authors \cite{HS_confluence} studies
the simplest case of the confluence relations when $x=0$, $y=1$,
$p_{1},\dots,p_{k}\in\{0,1,z\}$ and $p_{0}=0$, $p_{k+1}=1$. In
this article, we focus on the confluence relations for which $x=0$,
$y=1$, $p_{1},\dots,p_{k}\in\{0,-1,z,-z^{2}\}$ and $p_{0}=0$, $p_{k+1}=z$.
We will prove that these relations actually give all the motivic relations
among Euler sums, which also has a significant implication in the
Grothendieck-Teichm{\"u}ller theory. In the rest of this introduction,
we will discuss our main results and their significance in three folds. 

\subsection{Confluence relations for real-valued Euler sums and their consequences}

To begin with, we define the set of indices $\indset$ by
\[
\indset\coloneqq\{(k_{1},\dots,k_{d})\in(\mathbb{Z}\setminus\{0\})^{d}\mid d=0\text{ or }k_{d}\neq1\},
\]
and an Euler sum $\zeta(\Bbbk)$ together with its modified version
$\modzeta(\Bbbk)$ by
\[
\zeta(\Bbbk)\coloneqq\sum_{0<m_{1}<\cdots<m_{d}}\prod_{j=1}^{d}\frac{{\rm sgn}(k_{j})^{m_{j}}}{m_{j}^{\left|k_{j}\right|}}
\]
and $\modzeta(\Bbbk)\coloneqq(-2)^{d}\zeta(\Bbbk)$ for $\Bbbk=(k_{1},\dots,k_{d})\in\indset$.
We call ${\rm wt}(\Bbbk)=\left|k_{1}\right|+\cdots+\left|k_{d}\right|$
the weight and $\depth(\Bbbk)=d$ the depth of $\Bbbk$. Also, we
put
\[
\indset(k,d)\coloneqq\{\Bbbk\in\indset\mid{\rm wt}(\Bbbk)=k,\depth(\Bbbk)\leq d\}
\]
and
\[
\indset^{\mathrm{D}}(k,d)\coloneqq\{(k_{1},\dots,k_{r})\in\indset(k,d)\mid k_{1},\dots,k_{r-1}>0,k_{r}<0,k_{2}\equiv\cdots\equiv k_{r}\equiv1\pmod{2}\}.
\]
Fix $k,d\geq0$. In Part \ref{part:Confluence_Euler_sum} of this
article, we will show by an elementary argument using the confluence
relations for Euler sums that
\begin{thm*}[see Theorem \ref{thm:main_explicit}]
\label{thm:main_explicit_intro}For any $\Bbbk\in\indset(k,d)\setminus\indset^{\mathrm{D}}(k,d)$,
we have
\[
\modzeta(\Bbbk)=\sum_{\Bbbk'\in\indset^{\mathrm{D}}(k,d)}\alpha_{\Bbbk,\Bbbk'}\modzeta(\Bbbk')
\]
where $\alpha_{\Bbbk,\Bbbk'}$ are certain explicitly given rational
numbers with odd denominators.
\end{thm*}
The above theorem is significant for the following three reasons:
First, the theorem gives an upper bound 
\[
\dim\left\langle \zeta(\Bbbk)\mid\Bbbk\in\indset(k,d)\right\rangle _{\mathbb{Q}}\leq\#\indset^{\mathrm{D}}(k,d),
\]
which was already shown by Deligne \cite{Deli_es} (see also \cite{Glanois_fundamental_groupoid}
where the author considers a closely related alternative for $\{\zeta(\Bbbk)\left|\Bbbk\in\indset^{\mathrm{D}}(k,d)\right.\}$)
using the motivic framework, but up to now no elementary proof ($=$
does not depend on the theory of mixed Tate motives, Borel's calculation
of the $K$-groups modulo torsion, etc.) has been given. Thus, our
result gives the first elementary proof of the above dimension upper
bound. Second, by the method in \cite{Deli_es}, the existence of
the rational numbers $\alpha_{\Bbbk,\Bbbk'}$ can be proved, while
the above theorem tells us more precisely that the denominators of
$\alpha_{\Bbbk,\Bbbk'}$ are odd. Third, the coefficients $\alpha_{\Bbbk,\Bbbk'}$
are given in a completely explicit way, which enables us to obtain
lots of new information about $\alpha_{\Bbbk,\Bbbk'}$ such as an
upper bound of their denominators. It should be emphasized that such
precise results (oddness of the denominators of the coefficients,
an explicit formula for the coefficients and a calculable bound of
the denominator of the coefficients) seem to be unreachable by simply
extending the method of \cite{Deli_es}.

\subsection{Confluence relations for motivic Euler sums and their consequences}

In Part \ref{part:Motivicity} of this article, we prove that the
confluence relations are also satisfied by motivic Euler sums. Here
we state three remarkable applications of the confluence relations
for motivic Euler sums. For $\Bbbk\in\indset$, let $\zeta^{\mathfrak{m}}(\Bbbk)$
denote the motivic Euler sum and $\modzeta^{\mathfrak{m}}(\Bbbk)\coloneqq(-2)^{\depth(\Bbbk)}\zeta^{\mathfrak{m}}(\Bbbk)$
its modified version. Our first application is the following theorem.
\begin{thm*}[see Theorem \ref{thm:main_general_motivic_zeta}]
For any $\Bbbk\in\indset(k,d)\setminus\indset^{\mathrm{D}}(k,d)$,
we have
\[
\modzeta^{\mathfrak{m}}(\Bbbk)=\sum_{\Bbbk'\in\indset^{\mathrm{D}}(k,d)}\alpha_{\Bbbk,\Bbbk'}\modzeta^{\mathfrak{m}}(\Bbbk')
\]
where $\alpha_{\Bbbk,\Bbbk'}$ are the same rational numbers as in
Theorem \ref{thm:main_explicit_intro}.
\end{thm*}
Our second application is to fill the ``missing final piece'' of Brown's
decomposition algorithm for motivic multiple zeta values (which also
works for motivic Euler sums). In \cite{Bro_dec}, Brown gives an
algorithm to decompose a motivic multiple zeta value in terms of a
given basis. In \cite[p54]{Bro_dec}, he remarked that
\begin{quote}
``This is only an algorithm in the true sense of the word in so far
as it is possible to compute the coefficients $c_{\xi}$ and this
is the only transcendental input.''
\end{quote}
The determination of $c_{\xi}$ is an inevitable part of Brown's algorithm.
By computing $c_{\xi}$ as a real number to many digits, one can guess
the numerator and denominator of $c_{\xi}$, but since a guess is
just a guess, its correctness can not be mathematically verified,
no matter how many digits one calculates. Our theorem gives the first
successful method to determine the numerator and denominator of $c_{\xi}$
definitely. Suppose that we know the existence of $c_{\xi}$ such
that $\sum_{\Bbbk}q_{\Bbbk}\zeta^{\mathfrak{m}}(\Bbbk)=c_{\xi}\zeta^{\mathfrak{m}}(k)$
for some $\mathbb{Q}$-linear combination $\sum_{\Bbbk}q_{\Bbbk}\zeta^{\mathfrak{m}}(\Bbbk)$
of motivic Euler sums. Then, we find 
\[
\sum_{\Bbbk}q_{\Bbbk}\zeta^{\mathfrak{m}}(\Bbbk)=\sum_{\Bbbk}(-2)^{-\depth(\Bbbk)}q_{\Bbbk}\alpha_{\Bbbk,(-k)}\modzeta^{\mathfrak{m}}(-k)=\sum_{\Bbbk}(-2)^{1-\depth(\Bbbk)}(2^{1-k}-1)q_{\Bbbk}\alpha_{\Bbbk,(-k)}\zeta^{\mathfrak{m}}(k),
\]
and thus we get the following formula for $c_{\xi}$.
\begin{thm*}
If $\sum_{\Bbbk}q_{\Bbbk}\zeta^{\mathfrak{m}}(\Bbbk)=c\cdot\zeta^{\mathfrak{m}}(k)$
for some $c\in\mathbb{Q}$, then
\[
c=\sum_{\Bbbk}(-2)^{1-\depth(\Bbbk)}(2^{1-k}-1)q_{\Bbbk}\alpha_{\Bbbk,(-k)}.
\]
\end{thm*}
Notice that, as a consequence of this theorem (also as a consequence
of the previous theorem), we obtain a purely algebraic algorithm of
determining the expansion of multiple zeta values by a given basis. 

Finally, as our third application, we shall show 
\begin{thm*}
The set of relations among motivic Euler sums coincides with the set
of confluence relations.
\end{thm*}
As we will see in the next section, this theorem has remarkable implications
in the Grothendieck-Teichm{\"u}ller theory. 

\subsection{Actions of the motivic Galois groups on the fundamental groups}

The study of the actions of (various types of) Galois groups on (various
types of) fundamental groups of algebraic varieties is one of the
most important research topics in the Grothendieck--Teichm{\"u}ller
theory. For example, it is well-known that the absolute Galois group
acts on the geometric fundamental torsor of $\mathbb{P}^{1}\setminus\{0,1,\infty\}$:
\[
\phi_{{\rm abs}}:{\rm Gal}(\overline{\mathbb{Q}}/\mathbb{Q})\to{\rm Aut}(\pi_{1}^{{\rm geom}}(\mathbb{P}^{1}\setminus\{0,1,\infty\},\overrightarrow{1}_{0},-\overrightarrow{1}_{1})).
\]
A celebrated theorem of Belyi \cite{Belyi} shows that $\phi_{{\rm abs}}$
is injective, which says that its kernel is trivial. On the other
hand, no concrete description (e.g. via generators, defining equations
etc.) of the image $\phi_{{\rm abs}}({\rm Gal}(\overline{\mathbb{Q}}/\mathbb{Q}))$
is known. It has been shown by Drinfeld \cite{Dr_quasi} that the
image $\phi_{{\rm abs}}({\rm Gal}(\overline{\mathbb{Q}}/\mathbb{Q}))$
is contained in the profinite Grothendieck-Teichm{\"u}ller group
\[
\widehat{\mathrm{GT}}\subset{\rm Aut}(\pi_{1}^{{\rm geom}}(\mathbb{P}^{1}\setminus\{0,1,\infty\},\overrightarrow{1}_{0},-\overrightarrow{1}_{1}))
\]
which is closely related to Grothendieck's idea in \cite{Gro}. However,
it is not known whether or not $\phi_{{\rm abs}}({\rm Gal}(\overline{\mathbb{Q}}/\mathbb{Q}))$
coincides with $\widehat{\mathrm{GT}}$ at the moment. As a motivic
analog of this, it is known that the motivic Galois group $\mathcal{G}_{\mathcal{MT}(\mathbb{Z})}$
of mixed Tate motives over $\mathbb{Z}$ acts on the de Rham fundamental
torsor of $\mathbb{P}^{1}\setminus\{0,1,\infty\}$:
\[
\phi_{1}:\mathcal{G}_{\mathcal{MT}(\mathbb{Z})}\to{\rm Aut}(\pi_{1}^{{\rm dR}}(\mathbb{P}^{1}\setminus\{0,1,\infty\},\overrightarrow{1}_{0},-\overrightarrow{1}_{1}))
\]
and a celebrated theorem of Brown \cite{Bro_mix} states that $\phi_{1}$
is injective. However, in this case again, no concrete description
(e.g. via generators, defining equations etc.) of the image $\phi_{1}(\mathcal{G}_{\mathcal{MT}(\mathbb{Z})})$
is known. The graded version of the Grothendieck--Teichm{\"u}ller
group
\[
{\rm GRT}\subset{\rm Aut}(\pi_{1}^{{\rm dR}}(\mathbb{P}^{1}\setminus\{0,1,\infty\},\overrightarrow{1}_{0},-\overrightarrow{1}_{1}))
\]
which was also proposed by Drinfeld \cite{Dr_quasi} is one of the
most famous conjectural combinatorial description of the image $\phi_{1}(\mathcal{G}_{\mathcal{MT}(\mathbb{Z})})$,
and it is known that
\[
\phi_{1}(\mathcal{G}_{\mathcal{MT}(\mathbb{Z})})\subset{\rm GRT}.
\]

Our main result gives a combinatorial description of $\phi_{1}(\mathcal{G}_{\mathcal{MT}(\mathbb{Z})})$.
More strongly, we shall give an exact combinatorial description of
the image of
\[
\phi_{2}:\mathcal{G}_{\mathcal{MT}(\mathbb{Z}[1/2])}\to{\rm Aut}(\pi_{1}^{{\rm dR}}(\mathbb{P}^{1}\setminus\{0,\pm1,\infty\},\overrightarrow{1}_{0},-\overrightarrow{1}_{1}))
\]
where $\mathcal{G}_{\mathcal{MT}(\mathbb{Z}[1/2])}$ is the motivic
Galois group of mixed Tate motives over $\mathbb{Z}[1/2]$\footnote{The injectivity of $\phi_{2}$ was proved by Deligne \cite{Deli_es}.},
a larger class of mixed Tate motives over the rationals admitting
ramification at $2$. Let us explain about our description of $\phi_{2}(\mathcal{G}_{\mathcal{MT}(\mathbb{Z}[1/2])})$
here. Since it is more natural to consider the Galois action on the
entire fundamental groupoids rather than the individual fundamental
torsors, we will describe the former in what follows. 

Let $N$ be either $1$ or $2$. Then the fundamental groupoids are
defined as the union $\bigsqcup_{a,b\in B^{(N)}}\pio ab^{(N)}$ of
\begin{align*}
\pio ab^{(N)} & \coloneqq\pi_{1}^{{\rm dR}}(\mathbb{P}^{1}\setminus\left(B^{(N)}\cup\{\infty\}\right),a',b')\simeq{\rm Spec}(\mathfrak{H}^{(N)})\quad\left(a,b\in B^{(N)}\right)
\end{align*}
where $B^{(1)}=\{0,1\}$, $B^{(2)}=\{0,1,-1\}$ and $0'$, $1'$,
$-1$' denote tangential base points $\overrightarrow{1}_{0}$, $-\overrightarrow{1}_{1}$
and $\overrightarrow{1}_{-1}$, respectively, and $\mathfrak{H}^{(1)}\coloneqq(\mathbb{Q}\left\langle e_{0},e_{1}\right\rangle ,\shuffle)$,
(resp. $\mathfrak{H}^{(2)}\coloneqq(\mathbb{Q}\left\langle e_{0},e_{1},e_{-1}\right\rangle ,\shuffle)$)
is the commutative ring consisting of non-commutative polynomials
in two (resp. three) indeterminates equipped with the shuffle product.
Then $\mathcal{G}_{\mathcal{MT}(\mathbb{Z}[1/N])}$ acts on the fundamental
groupoid $\bigsqcup_{a,b\in B^{(N)}}\pio ab^{(N)}$, and we denote
this action by 
\[
\widehat{\phi}_{N}:\mathcal{G}_{\mathcal{MT}(\mathbb{Z}[1/N])}\to\Gamma^{(N)}\coloneqq\mathrm{Aut}(\bigsqcup_{a,b\in B^{(N)}}\pio ab^{(N)}).
\]
The set of $\mathbb{Q}$-rational points of $\pio ab^{(N)}$ can be
canonically identified with the set of group-like elements in $R^{(N)}$
where $R^{(1)}\coloneqq\mathbb{Q}\left\langle \left\langle e^{0},e^{1}\right\rangle \right\rangle $
and $R^{(2)}\coloneqq\mathbb{Q}\left\langle \left\langle e^{0},e^{1},e^{-1}\right\rangle \right\rangle $.
For a group-like element $p\in R^{(N)}$, we write $\dpath apb$ for
the corresponding element of $\pio ab^{(N)}$. Then an automorphism
$\sigma\in\Gamma^{(N)}$ (with respect to the groupoid scheme structure
$[g;a,b][h;b,c]=[gh;a,c]$) is completely determined by its action
to the elements
\begin{align*}
t_{0},\ t_{1}\text{ and }\dpath 011 & \quad\text{ for }N=1\\
t_{0},\ t_{1},\ t_{-1},\ \dpath 011,\text{ and }\dpath 01{-1} & \quad\text{ for }N=2
\end{align*}
where $t_{a}\coloneqq\dpath a{\exp(e^{a})}a$. Then the image of $\mathcal{G}_{\mathcal{MT}(\mathbb{Z}[1/N])}$
in $\Gamma^{(N)}$ under $\widehat{\phi}_{N}$ is equal to
\[
\left\{ \sigma_{m,p}^{(N)}\in\Gamma^{(N)}\left|(m,p)\in\mathbb{Q}^{\times}\times{\rm Spec}\left(\mathfrak{H}^{(N)}/(e_{1}e_{0},\ker(L^{\mathfrak{m}})\cap\mathfrak{H}^{(N)})\right)\right.\right\} 
\]
with 
\[
\sigma_{m,p}^{(1)}(t_{a})=t_{a}^{m}\;\left(a\in\{0,1\}\right),\;\sigma_{m,p}^{(1)}(\dpath 011)=\dpath 0p1
\]
and
\[
\sigma_{m,p}^{(2)}(t_{a})=t_{a}^{m}\;\left(a\in\{0,\pm1\}\right),\;\sigma_{m,p}^{(2)}(\dpath 011)=\dpath 0p1,\sigma_{m,p}^{(2)}(\dpath 01{-1})=\dpath 0{\tau(p)}{-1},
\]
where $t_{a}^{m}=\dpath a{\exp(me^{a})}a$, $\tau$ is an automorphism
of $\mathfrak{H}^{(2)}$ defined by $\tau(e_{a})=e_{-a}$ for $a\in\{0,\pm1\}$,
and $L^{\mathfrak{m}}$ is a linear map from $\mathfrak{H}^{(2)}$
to the set $\mathcal{H}_{2}$ of motivic Euler sums defined by
\[
L^{\mathfrak{m}}(e_{a_{1}}\cdots e_{a_{k}})=I^{\mathfrak{m}}(\overrightarrow{1}_{0};a_{1},\dots,a_{k};-\overrightarrow{1}_{1}).
\]
Thus $\widehat{\phi}_{N}(\mathcal{G}_{\mathcal{MT}(\mathbb{Z}[1/N])})$
is completely determined by the information of $\ker(L^{\mathfrak{m}})$.
In this article, we introduce a submodule $\widehat{\mathcal{I}}_{\mathrm{CF}}\subset\mathfrak{H}^{(2)}$
with explicit generators, namely, ``(the non-admissible extension
of) the confluence relations of level two'', and show the following
theorem.
\begin{thm*}[See Theorem \ref{thm:main_conf-2}]
\label{thm:main_conf}We have
\[
\widehat{\mathcal{I}}_{\mathrm{CF}}=\ker(L^{\mathfrak{m}}).
\]
\end{thm*}
As corollaries of this theorem, we obtain the following descriptions
of the images of the motivic Galois groups $\mathcal{G}_{\mathcal{MT}(\mathbb{Z}[1/2])}$
and $\mathcal{G}_{\mathcal{MT}(\mathbb{Z})}$.
\begin{thm}
\label{thm:main_galois}The image $\widehat{\phi}_{2}(\mathcal{G}_{\mathcal{MT}(\mathbb{Z}[1/2])})$
is given by
\begin{align*}
\widehat{\phi}_{2}(\mathcal{G}_{\mathcal{MT}(\mathbb{Z}[1/2])}) & =\left\{ \sigma_{m,p}^{(2)}\in\Gamma^{(2)}\left|(m,p)\in\mathbb{Q}^{\times}\times{\rm Spec}\left(\mathfrak{H}^{(2)}/(e_{1}e_{0},\widehat{\mathcal{I}}_{\mathrm{CF}}\right)\right.\right\} .
\end{align*}
\end{thm}

\begin{thm}
\label{thm:main_galois-1}The image $\widehat{\phi}_{1}(\mathcal{G}_{\mathcal{MT}(\mathbb{Z})})$
is given by
\[
\widehat{\phi}_{1}(\mathcal{G}_{\mathcal{MT}(\mathbb{Z})})=\left\{ \sigma_{m,p}^{(1)}\in\Gamma^{(1)}\left|(m,p)\in\mathbb{Q}^{\times}\times{\rm Spec}\left(\mathfrak{H}^{(1)}/(e_{1}e_{0},\widehat{\mathcal{I}}_{\mathrm{CF}}\cap\mathfrak{H}^{(1)}\right)\right.\right\} .
\]
\end{thm}

\begin{rem}
The above theorems also determine the images of $\phi_{1}$ and $\phi_{2}$
since 
\[
\phi_{N}(\mathcal{G}_{\mathcal{MT}(\mathbb{Z}[1/N])})=\left\{ \sigma\mid_{\pio 01}\:\left|\sigma\in\widehat{\phi}_{N}(\mathcal{G}_{\mathcal{MT}(\mathbb{Z}[1/N])})\right.\right\} 
\]
where $\mid_{\pio 01}$ means the restriction to $\pio 01$.
\end{rem}

\subsection{Structure of the article}

This article is organized into two parts of different flavors: Part
1 focuses on (real-valued) Euler sums and their confluence relations,
while Part 2 (except for its last section) is devoted to proving the
motivicity of the confluence relations. As the contents are mostly
independent, for convenience, we will use slightly different conventions
for the two parts. Since Part 1 is dedicated to the study of Euler
sums, it deals only with very special cases of hyperlogarithms and
their limits that are relevant to Euler sums. Part 2, on the other
hand, studies the motivic counterpart of hyperlogarithms in a great
generality. This is because restricting ourselves to Euler sum case
does not simplify the argument greatly, and also because we consider
that proving such a fundamental result in a great generality would
be useful from the perspective of future applications. We will construct
the motivic counterpart of ``limits of hyperlogarithms'' thereby defining
the confluence relations in ultimate generality, and as an upshot
prove their motivicity. Using this motivicity theorem, we will derive
the aforementioned consequences in the Grothendieck-Teichm{\"u}ller
theory from the results of Part 1 at the end. The proof of the motivicity
of (generalized) confluence relation consists of loads of auxiliary
lemmas and propositions with loads of notations, which may be painful
to read. Therefore, except for those who are interested in the technical
details of the proof of the motivicity of the confluence relations
are advised to ignore Part 2. 

Part 1 consists of two sections, Sections \ref{sec:Definition-of-confluence}
and \ref{sec:Main-theorem}. In Section \ref{sec:Definition-of-confluence},
we give an algebraic set-up analogous to that by Hoffman \cite{Hoffman},
and make use of the differential formula for hyperlogarithms together
with a suitable limiting process called \emph{regularized limit} to
define a class of ($\mathbb{Q}$-linear) relations satisfied by Euler
sums denote by $\iconf$, which we call \emph{confluence relations
of level two}. Then, in Section \ref{sec:Main-theorem}, we prove
Theorems \ref{thm:main_main} and \ref{thm:main_explicit} as goals
of Part 1.

Part 2 consists of four sections, Sections \ref{sec:Part2_preliminaries},
\ref{sec:Evaluation-map}, \ref{sec:Motivicity-of-confluence}, and
\ref{sec:Proof-of-Theorem}. Section \ref{sec:Part2_preliminaries}
is devoted to preliminaries to the following sections. In Section
\ref{sec:Evaluation-map}, we define the motivic counterpart of ``regularized
limits of hyperlogarithms'' and prove basic (and desirable) properties
satisfied by those ``limits''. Then in Section \ref{sec:Motivicity-of-confluence},
we generalize the notion of the confluence relation to a very general
setting and prove the motivicity of the generalized confluence relations
using the properties shown in Section \ref{sec:Evaluation-map}. Finally,
in Section \ref{sec:Proof-of-Theorem}, we shall prove the aforementioned
main theorems.

\pagebreak{}

\part{\label{part:Confluence_Euler_sum}Confluence relations for Euler
sums}

In this part, we define confluence relations for Euler sums and study
their consequences. 

\section{Definition of confluence relations\label{sec:Definition-of-confluence}}

\subsection{Iterated integrals}

First of all we give an algebraic and analytic set-up for iterated
integrals.
\begin{defn}
For a subset $S\subset\mathbb{C}$, we denote by $\mathcal{V}(S)$
the $\mathbb{Z}$-module generated by the formal symbols $\{e_{a}\mid a\in S\}$.
\end{defn}

\begin{defn}
For $u\in\mathcal{V}(S)$, we define a rational $1$-form $\omega_{u}$
on $\mathbb{P}^{1}(\mathbb{C})$ by the linearity on $u$ and
\[
\omega_{e_{a}}(t)=\frac{dt}{t-a}.
\]
\end{defn}

\begin{defn}
For $a\in\mathbb{P}^{1}(\mathbb{C})=\mathbb{C}\cup\{\infty\}$, we
define $\mathcal{V}^{(a)}(S)\subset\mathcal{V}(S)$ by
\[
\mathcal{V}^{(a)}(S)=\begin{cases}
\bigoplus_{t\in S\setminus\{a\}}\mathbb{Z}\cdot e_{t} & a\in\mathbb{C}\\
\sum_{t,u\in S}\mathbb{Z}\cdot(e_{t}-e_{u}) & a=\infty.
\end{cases}
\]
\end{defn}

Note that $\omega_{u}$ is regular at $a\in\mathbb{P}^{1}(\mathbb{C})$
if $u\in\mathcal{V}^{(a)}(S)$.

\begin{defn}
For a subset $S\subset\mathbb{C}$, we denote by $\mathcal{A}(S)$
the free non-commutative polynomial ring over $\mathbb{Z}$ generated
by the formal symbols $\{e_{a}\mid a\in S\}$. In other words,
\[
\mathcal{A}(S)=\bigoplus_{k=0}^{\infty}\mathcal{V}(S)^{\otimes k}.
\]
\end{defn}

\begin{defn}
For a subset $S\subset\mathbb{C}$ and $x,y\in\mathbb{P}^{1}(\mathbb{C})$,
we set
\[
\mathcal{A}^{(x,y)}(S)=\mathbb{Z}\oplus(\mathcal{V}^{(x)}(S)\cap\mathcal{V}^{(y)}(S))\oplus\mathcal{V}^{(x)}(S)\mathcal{A}(S)\mathcal{V}^{(y)}(S).
\]
\end{defn}

\begin{defn}
Let $x$ and $y$ be usual or tangential base points of $\mathbb{P}^{1}(\mathbb{C})$.
We denote by $\pi_{S}^{x,y}$ the set of homotopy class of paths $\gamma:[0,1]\to\mathbb{P}^{1}(\mathbb{C})$
from $x$ to $y$ such that $\gamma((0,1))\subset\mathbb{C}\setminus S$.
\end{defn}

\begin{defn}[Iterated integral]
For $x,y\in\mathbb{P}^{1}(\mathbb{C})$, $\gamma\in\pi_{S}^{x,y}$
and $u\in\mathcal{A}^{(x,y)}(S)$, define $I_{\gamma}(x;u;y)$ by
the linearity for $u$ and
\[
I_{\gamma}(x;v_{1}\cdots v_{k};y)=\int_{0<t_{1}<\cdots<t_{k}<1}\omega_{v_{1}}(\gamma(t_{1}))\cdots\omega_{v_{k}}(\gamma(t_{k}))
\]
for $v_{1}\in\mathcal{V}^{(x)}(S)$, $v_{2},\dots,v_{k-1}\in\mathcal{V}(S)$,
and $v_{k}\in\mathcal{V}^{(y)}(S)$. Furthermore, we extend the definition
of $I_{\gamma}(x;u;y)$ for tangential base points $x$ and $y$ on
$\mathbb{P}^{1}(\mathbb{C})$ and $u\in\mathcal{A}(S)$ in the usual
way (see, for example, \cite{Deli_tan}).
\end{defn}

We simply write $I_{\gamma}(-)$ as $I(-)$ if the choice of $\gamma$
is obvious from the context.

\subsection{Iterated integral expression of Euler sums and the distribution relation}
\begin{defn}
For $x>0$, define a linear map
\[
L_{x}:\mathcal{A}^{(0,x)}(\mathbb{R}\setminus(0,x))\longrightarrow\mathbb{R}
\]
by
\begin{align*}
L_{x}(e_{a_{1}}\cdots e_{a_{k}}) & \coloneqq I(0;e_{a_{1}}\cdots e_{a_{k}};x)\\
 & =\int_{0<t_{1}<\cdots<t_{k}<x}\frac{dt_{1}}{t_{1}-a_{1}}\cdots\frac{dt_{k}}{t_{k}-a_{k}}.
\end{align*}
\end{defn}

\begin{defn}
For $S\subset\{0,1,-1\}$, $k\in\mathbb{Z}_{\geq0}$ and $d\in\mathbb{Z}_{\geq0}\cup\{\infty\}$,
we put $\mathcal{A}^{0}(S)\coloneqq\mathcal{A}^{(0,1)}(S)$ and
\begin{align*}
\mathcal{A}_{k,d}(S) & \coloneqq\bigoplus_{\substack{a_{1},\dots,a_{k}\in S\\
\#\{i\mid a_{i}\neq0\}\leq d
}
}\mathbb{Z}e_{a_{1}}\cdots e_{a_{k}}\subset\mathcal{A}(S)\\
\mathcal{A}_{k,d}^{0}(S) & \coloneqq\bigoplus_{\substack{a_{1},\dots,a_{k}\in S\\
\#\{i\mid a_{i}\neq0\}\leq d\\
a_{1}\neq0,a_{k}\neq1
}
}\mathbb{Z}e_{a_{1}}\cdots e_{a_{k}}\subset\mathcal{A}^{0}(S).
\end{align*}
\end{defn}

\begin{defn}
\label{def:reg_shuffle-1}For $S\subset\{0,1,-1\}$, define $\shreg:\mathcal{A}(S)\to\mathcal{A}^{0}(S)$
as the unique linear map such that
\[
\shreg(e_{0}^{n}\shuffle e_{1}^{m}\shuffle u)=\begin{cases}
u & n=m=0\\
0 & \text{otherwise}
\end{cases}
\]
for all $u\in\mathcal{A}^{0}(S)$.
\end{defn}

\begin{defn}
\label{def:index_word}For $\Bbbk=(k_{1},\dots,k_{d})\in\indset$,
we define $\word(\Bbbk)\in\mathcal{A}^{0}(\{0,1,-1\})$ and $\modword(\Bbbk)\in\mathcal{A}^{0}(\{0,1,-1\})$
by
\[
\word(\Bbbk)\coloneqq e_{\epsilon_{1}}e_{0}^{\left|k_{1}\right|-1}\cdots e_{\epsilon_{d}}e_{0}^{\left|k_{d}\right|-1}
\]
and $\modword(\Bbbk)\coloneqq2^{d}\word(\Bbbk)$ where $\epsilon_{i}={\rm sgn}(k_{i})\epsilon_{i+1}$
and $\epsilon_{d+1}=1$.
\end{defn}

Now we can express an Euler sum by $L_{1}$:
\[
\zeta(\Bbbk)=(-1)^{d}L_{1}(\word(\Bbbk)),\ \modzeta(\Bbbk)=L_{1}(\modword(\Bbbk)).
\]

\begin{defn}
We define a $\dist:\mathcal{A}_{k,d}(\{0,1\})\to\mathcal{A}_{k,d}(\{0,1,-1\})$
as the ring homomorphism such that

\[
\dist(e_{0})=2e_{0},\ \dist(e_{1})=e_{1}+e_{-1}.
\]
\end{defn}

\begin{defn}
For $k\in\mathbb{Z}_{\geq0}$ and $d\in\mathbb{Z}_{\geq0}\cup\{\infty\}$,
define the subspace $\tilde{\mathcal{A}}_{k,d}^{0}(\{0,1,-1\})\subset\mathcal{A}_{k,d}^{0}(\{0,1,-1\})$
by

\[
\tilde{\mathcal{A}}_{k,d}^{0}(\{0,1,-1\})=\begin{cases}
\mathcal{A}_{k,d}^{0}(\{0,1,-1\})=\mathbb{Z} & k=0\\
2\mathcal{A}_{k,d}^{0}(\{0,1,-1\}) & k>0.
\end{cases}
\]
\end{defn}

By definition, for $u\in\mathcal{A}_{k,d}^{0}(\{0,1\})$, we have
\begin{equation}
\dist(u)\in\tilde{\mathcal{A}}_{k,d}^{0}(\{0,1,-1\}).\label{eq:dist_in_tilde}
\end{equation}
The following equation is well-known.
\begin{lem}[Distribution relation]
\label{lem:distr}For $u\in\mathcal{A}^{0}(\{0,1\})$, 
\[
L_{1}(u)=L_{1}(\dist(u)).
\]
\end{lem}

\begin{proof}
It follows from the change of variables $t=s^{2}$ since $\frac{dt}{t}=\frac{2ds}{s}$
and $\frac{dt}{t-1}=\frac{ds}{s-1}+\frac{ds}{s+1}$.
\end{proof}

\subsection{Evaluation at $z=0$}

Let $z$ be a real variable such that $0<z<1$. We put
\[
\mathcal{B}_{k,d}\coloneqq\bigoplus_{\substack{a_{1},\dots,a_{k}\in\{0,-1,z,-z^{2}\}\\
a_{1}\neq0,a_{k}\neq z\\
\#\{i\mid a_{i}\neq0\}\leq d
}
}\mathbb{Z}e_{a_{1}}\cdots e_{a_{k}}
\]
and
\[
\mathcal{B}\coloneqq\bigoplus_{k=0}^{\infty}\mathcal{B}_{k,\infty}.
\]
Then for $u\in\mathcal{B}$, we can regard $L_{z}(u)$ as a function
on $0<z<1$. The limit $\lim_{z\to0}L_{z}(u)$ does not necessarily
exist, but the following \emph{regularized limit} exists.
\begin{defn}[Regularized limit]
Fix $a\in\mathbb{R}$. Suppose that $f(z)$ behaves as
\[
f(a\pm\epsilon)=c_{0}+c_{1}\log(\epsilon)+\cdots+c_{M}\log^{M}(\epsilon)+O(\epsilon\log^{M}\epsilon)
\]
around $a$ for some $M>0$ and $c_{0},\dots,c_{M}\in\mathbb{C}$.
Then we define the regularized limit $\Reg_{z\to a\pm0}f(z)$ to be
the constant term, i.e.,
\[
\Reg_{z\to a\pm0}f(z)\coloneqq c_{0}.
\]
\end{defn}

Notice that $\Reg_{z\to a\pm0}$ extends the notion of the limit in
the usual sense, since $\Reg_{z\to a\pm0}f(z)=\lim_{z\to a\pm0}f(z)$
if $\lim_{z\to a\pm0}f(z)$ exists. In this subsection, we shall define
$\reg_{z\to0}:\mathcal{B}_{k,d}\to\mathcal{A}_{k,d}^{0}(\{0,1,-1\})$
which naturally satisfies
\[
\Reg_{z\to+0}L_{z}(u)=L_{1}(\reg_{z\to0}u)
\]
for $u\in\mathcal{B}_{k,d}$.

Let $u=e_{a_{1}}\cdots e_{a_{k}}\in\mathcal{B}$. Note that $\lim_{z\to0}L_{z}(u)=0$
if $a_{j}=-1$ for some $j$. Thus the computation of $\Reg_{z\to+0}L_{z}(u)$
for $u\in\mathcal{B}$ is reduced to the case
\[
u\in\mathcal{B}'\coloneqq\mathcal{B}\cap\mathbb{Z}\left\langle e_{0},e_{z},e_{-z^{2}}\right\rangle .
\]
Put
\begin{align*}
\mathcal{B}'' & \coloneqq\mathcal{B}'\cap\left(\mathbb{Z}\oplus e_{z}\mathbb{Z}\left\langle e_{0},e_{z},e_{-z^{2}}\right\rangle \right),\\
\mathcal{B}''' & \coloneqq\mathcal{B}'\cap\mathbb{Z}\left\langle e_{0},e_{-z^{2}}\right\rangle .
\end{align*}
Since we have
\[
\mathcal{B}''\otimes\mathcal{B}'''\simeq\mathcal{B}'\ ;\ u_{1}\otimes u_{2}\mapsto u_{1}\shuffle u_{2},
\]
the computation of $\Reg_{z\to+0}L_{z}(u)$ for $u\in\mathcal{B}'$
is further reduced to the cases $u\in\mathcal{B}''$ and $u\in\mathcal{B}'''$.
For $u=e_{a_{1}}\cdots e_{a_{k}}\in\mathcal{B}''$, we have
\begin{align*}
\lim_{z\to0}L_{z}(u) & =\lim_{z\to0}L_{1}(e_{\frac{a_{1}}{z}}\cdots e_{\frac{a_{k}}{z}})\\
 & =L_{1}(\barreg_{z\to0}(u))
\end{align*}
where $\mathop{{\rm \overline{reg}}}_{z\to0}$ is the ring homomorphism
determined by $\mathop{{\rm \overline{reg}}}_{z\to0}(e_{0})=\mathop{{\rm \overline{reg}}}_{z\to0}(e_{-z^{2}})=e_{0}$
and $\mathop{{\rm \overline{reg}}}_{z\to0}(e_{z})=e_{1}$. Next, for
$u=e_{a_{1}}\cdots e_{a_{k}}\in\mathcal{B}'''$, put $w=w(u)=e_{b_{1}}\cdots e_{b_{k}}\in\mathcal{C}\coloneqq\mathbb{Z}\oplus e_{1}\mathbb{Z}\left\langle e_{0},e_{1}\right\rangle $
where $b_{j}=-\frac{a_{j}}{z^{2}}$ for $1\leq j\leq k$. Then
\[
L_{z}(u)=I_{\gamma_{z}}(0;w;-\frac{1}{z})
\]
where $\gamma_{z}$ is the straight path from $0$ to $-1/z$. Furthermore,
by the change of variables $t\mapsto t/(t-1)$, we have
\[
I_{\gamma_{z}}(0;w;-\frac{1}{z})=I_{\gamma_{z}'}(0;\varrho(w);\frac{1}{1+z})
\]
where $\varrho$ is the automorphism of $\mathcal{A}(\{0,1\})$ defined
by $\varrho(e_{0})=e_{0}-e_{1}$ and $\varrho(e_{1})=-e_{1}$, and
$\gamma'_{z}$ is the straight path from $0$ to $1/(1+z)$. Thus
we have
\[
\Reg_{z\to+0}I_{\gamma_{z}}(0;w;-\frac{1}{z})=L_{1}(\shreg(\varrho(w))).
\]
This already gives an expression of $\Reg_{z\to+0}L_{z}(u)$ for $u\in\mathcal{B}'''$
in terms of $L_{1}$. However, we further have to rewrite $L_{1}(\shreg(\varrho(w)))$
for depth compatibility since $\varrho(\mathcal{A}_{k,d}(\{0,1\}))\not\subset\mathcal{A}_{k,d}(\{0,1\})$
in general. Define an anti-automorphism $\varsigma$ of $\mathcal{A}(\{0,1\})$
by $\varsigma(e_{0})=-e_{1}$ and $\varsigma(e_{1})=-e_{0}$.
\begin{defn}
Define the linear map $\wp:\mathcal{C}\to\mathcal{A}^{0}(\{0,1,-1\})$
as follows. For $d\geq0$ and $\Bbbk=(k_{1},\dots,k_{d})\in\mathbb{Z}_{\geq1}^{d}$,
define $\theta(\Bbbk)$ recursively by

\[
\theta(k_{1},\dots,k_{d})=\begin{cases}
-\sum_{j=1}^{d}e_{1}e_{0}^{k_{j}-1}\cdots e_{1}e_{0}^{k_{1}-1}\shuffle\theta(k_{j+1},\dots,k_{d}) & d>0\\
1 & d=0,
\end{cases}
\]
and $\theta_{1}(\Bbbk)$ by
\[
\theta_{1}(\Bbbk)=\theta'(\Bbbk)+2\sum_{\substack{\Bbbk=(\Bbbk',\{1\}^{2m})\\
m\geq1
}
}e_{-1}e_{0}^{2m-1}\shuffle\theta'(\Bbbk')
\]
with $\theta'=\dist\circ\shreg\circ\varsigma\circ\theta$. Put
\[
\word^{\star}(\emptyset)=1,
\]
\[
\word^{\star}(k_{1},\dots,k_{d})=-e_{1}e_{0}^{k_{1}-1}(e_{0}-e_{1})e_{0}^{k_{2}-1}\cdots(e_{0}-e_{1})e_{0}^{k_{d}-1}.
\]
Note that $\{\varrho(\word^{\star}(\Bbbk))\}$ is a basis of $\mathcal{C}$.
Thus, we finally define $\wp$ by
\[
\wp(\varrho(\word^{\star}(k_{1},\dots,k_{d})))=\theta_{1}(k_{1},\dots,k_{d}).
\]
\end{defn}

\begin{lem}
\label{lem:L1_shreg_varrho_is_L1_wp}For $w\in\mathcal{C}$, we have
\[
L_{1}(\shreg(\varrho(w)))=L_{1}(\wp(w)).
\]
\end{lem}

\begin{proof}
We prove the lemma using a result in Kaneko-Yamamoto's article \cite{KY}.
To avoid confusion we will use $E_{b}=(-1)^{b}e_{b}$ instead of $e_{b}$
for $b\in\{0,\pm1\}$ throughout this proof, since $E_{b}$ is the
convention used in their article (and denoted as $e_{b}$ there).
We denote by $P$ the set of isomorphism classes of \emph{2-posets},
where a 2-poset means a pair $(X,\delta_{X})$ where $X$ is a finite
partially ordered set and $\delta_{X}:X\to\{0,1\}$ is a map. We depict
a $2$-poset as a Hasse diagram in which an element $x$ with $\delta(x)=0$
(resp. $\delta(x)=1$) is represented by $\circ$ (resp. $\bullet$).
Define $W:P\to\mathcal{A}(\{0,1\})$ by
\[
W((X,\delta))=\sum_{\substack{f:\{1,\dots,n\}\to X\\
{\rm bijection}
}
}\begin{cases}
E_{\delta(f(1))}\cdots E_{\delta(f(n))} & \text{if \ensuremath{f^{-1}(x)<f^{-1}(y)} for all \ensuremath{x,y\in X} such that \ensuremath{x<y}}\\
0 & {\rm otherwise}
\end{cases}
\]
where $n=\#X$. For $\Bbbk=(k_{1},\dots,k_{d})\in\mathbb{Z}_{\geq1}^{d}$,
we write
\[
\bullet\!\!-\!\!\boxed{\Bbbk}
\]
for the diagram

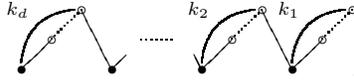
\[
\begin{xy}{(0,1)\ar@/^{2mm}/@{-}^{k_{d}}(7,8)},{(24,1)\ar@/^{2mm}/@{-}^{k_{2}}(31,8)},{(36,1)\ar@/^{2mm}/@{-}^{k_{1}}(43,8)},{(0,0)\ar@{{*}-}(4,4)},{(4,4)\ar@{{o}.}(8,8)},{(8,8)\ar@{{o}-}(12,0)},{(12,0)\ar@{{*}-}(14,2)},{(16,4)\ar@{.}(20,4)},{(23,2)\ar@{-}(24,0)},{(24,0)\ar@{{*}-}(28,4)},{(28,4)\ar@{{o}.}(32,8)},{(32,8)\ar@{{o}-}(36,0)},{(36,0)\ar@{{*}-}(40,4)},{(40,4)\ar@{{o}.{o}}(44,8)}\end{xy}.
\]
Then $W(\bullet\!\!-\!\!\boxed{\Bbbk})=\theta(k_{1},\dots,k_{d})$
by induction on $d$ from the case ${\bf k}=(k_{d},\dots,k_{1})$
and ${\bf l}=(0)$ of the second formula of \cite[Lemma 5.2]{KY}.
(Strictly speaking, the case ${\bf l}=(0)$ is not included in \cite[Lemma 5.2]{KY},
but the proof works also for this case). Let $R$ be a commutative
$\mathbb{Q}$-algebra and $Z:\mathcal{A}^{0}(\{0,1\})\to R$ a linear
map satisfying the regularized double shuffle relation (see \cite{KY}).
We extend $Z$ to a map $\mathbb{Z}+E_{1}\mathbb{Z}\left\langle E_{0},E_{1}\right\rangle \to R[T]$
in two ways, namely $Z_{\shuffle}^{(T)}$ and $Z_{*}^{(T)}$, by the
properties
\[
Z_{\shuffle}^{(T)}(E_{1})=Z_{*}^{(T)}(E_{1})=T,
\]
\[
Z_{\shuffle}^{(T)}\mid_{\mathcal{A}^{0}(\{0,1\})}=Z_{*}^{(T)}\mid_{\mathcal{A}^{0}(\{0,1\})}=Z
\]
and
\[
Z_{\shuffle}^{(T)}(u\shuffle v)=Z_{\shuffle}^{(T)}(u)Z_{\shuffle}^{(T)}(v),\ Z_{*}^{(T)}(u*v)=Z_{*}^{(T)}(u)Z_{*}^{(T)}(v),
\]
where $*$ is the harmonic product (see \cite{KY} for its definition).
Furthermore, define $\mathbb{Q}$-linear maps $\rho_{Z},\rho_{Z}^{\star}:R[T]\to R[T]$
by the equalities
\[
\rho_{Z}(e^{T\xi})=\Gamma_{Z}(\xi)e^{T\xi},\ \rho_{Z}^{\star}(e^{T\xi})=\Gamma_{Z}(-\xi)^{-1}e^{T\xi}
\]
in $R[T][[\xi]]$ ($\rho_{Z},\rho_{Z}^{\star}$ act coefficient-wisely)
where
\[
\Gamma_{Z}(\xi)=\exp\left(\sum_{n=2}^{\infty}\frac{Z(E_{1}E_{0}^{n-1})}{n}(-\xi)^{n}\right)\in R[[\xi]].
\]
Then the regularization theorem \cite{IKZ} says that
\[
Z_{\shuffle}^{(T)}(u)=\rho_{Z}(Z_{*}^{(T)}(u))
\]
for $u\in\mathbb{Z}+E_{1}\mathbb{Z}\left\langle E_{0},E_{1}\right\rangle $,
and the star-regularization theorem (\cite[Theorem 4.6]{KY}) says
that
\[
Z_{\shuffle}^{(T)}(W(\bullet\!\!-\!\!\boxed{\Bbbk}))=\rho_{Z}^{\star}(Z_{*}^{(T)}(\word^{\star}(\Bbbk)))
\]
for $\Bbbk=(k_{1},\dots,k_{d})\in\mathbb{Z}_{\geq1}^{d}$ since
\[
\word^{\star}(\Bbbk)=E_{1}E_{0}^{k_{1}-1}(E_{1}+E_{0})E_{0}^{k_{2}-1}\cdots(E_{1}+E_{0})E_{0}^{k_{d}-1}.
\]
If we write $\Bbbk=(\Bbbk',\{1\}^{m})$ with an admissible index $\Bbbk'$
and $m\geq0$, then

\begin{align*}
W(\bullet\!\!-\!\!\boxed{\Bbbk}) & =W\Bigg(\begin{xy}{(1,8)\ar@/^{1mm}/@{-}^{m}(4,5)},{(0,8)\ar@{{*}.}(4,4)},{(4,4)\ar@{{*}-}(8,0)},{(8,0)\ar@{{*}-}(12,0)},(14,0)*{\boxed{\Bbbk'}}\end{xy}\Bigg)\\
 & =\sum_{j=0}^{m}E_{1}^{j}\shuffle\shreg\circ W\Bigg(\begin{xy}{(1,8)\ar@/^{1mm}/@{-}^{m-j}(4,5)},{(0,8)\ar@{{*}.}(4,4)},{(4,4)\ar@{{*}-}(8,0)},{(8,0)\ar@{{*}-}(12,0)},(14,0)*{\boxed{\Bbbk'}}\end{xy}\Bigg)\\
 & =\sum_{j=0}^{m}E_{1}^{j}\shuffle\shreg\circ W\Bigg(\bullet\!\!-\!\!\boxed{(\Bbbk',\{1\}^{m-j})}\Bigg).
\end{align*}
Thus
\[
W(\bullet\!\!-\!\!\boxed{\Bbbk})=\sum_{\Bbbk=(\Bbbk'',\{1\}^{j})}E_{1}^{j}\shuffle\shreg\circ W\Bigg(\bullet\!\!-\!\boxed{\Bbbk''}\Bigg).
\]
Thus, by regularization theorem, star-regularization theorem, $Z_{\shuffle}^{(T)}\circ\shreg=Z_{\shuffle}^{(0)}$
and $Z_{\shuffle}^{(0)}=Z\circ\shreg$, we have
\begin{align*}
Z_{\shuffle}^{(T)}(\word^{\star}(\Bbbk)) & =\rho_{Z}\circ Z_{*}^{(T)}(\word^{\star}(\Bbbk))\\
 & =\rho_{Z}\circ(\rho_{Z}^{\star})^{-1}\circ Z_{\shuffle}^{(T)}(W(\bullet\!\!-\!\!\boxed{\Bbbk})).\\
 & =\sum_{\Bbbk=(\Bbbk'',\{1\}^{j})}\rho_{Z}\circ(\rho_{Z}^{\star})^{-1}\left(\frac{T^{j}}{j!}Z_{\shuffle}^{(0)}(W(\bullet\!\!-\!\boxed{\Bbbk''}))\right)\\
 & =\sum_{\Bbbk=(\Bbbk'',\{1\}^{j})}\rho_{Z}\circ(\rho_{Z}^{\star})^{-1}\left(\frac{T^{j}}{j!}\right)\times Z\circ\shreg(W(\bullet\!\!-\!\!\boxed{\Bbbk''})).
\end{align*}
Let us calculate $\rho_{Z}\circ(\rho_{Z}^{\star})^{-1}\left(\frac{T^{j}}{j!}\right)$.
By definition
\begin{align*}
\rho_{Z}\circ(\rho_{Z}^{\star})^{-1}(e^{T\xi}) & =\Gamma_{Z}(\xi)\Gamma_{Z}(-\xi)e^{T\xi}\\
 & =\exp\left(\sum_{n=2}^{\infty}\frac{Z(E_{1}E_{0}^{n-1})}{n}(\xi^{n}+(-\xi)^{n})\right)e^{T\xi}\\
 & =\exp\left(\sum_{m=1}^{\infty}\frac{Z(E_{1}E_{0}^{2m-1})}{m}\xi^{2m}\right)e^{T\xi},
\end{align*}
and 
\[
\exp\left(\sum_{m=1}^{\infty}\frac{Z(E_{1}E_{0}^{2m-1})}{m}\xi^{2m}\right)
\]
is equal to
\[
\sum_{m=0}^{\infty}Z(\word^{\star}(\{2\}^{m}))\xi^{2m}
\]
by harmonic relation. Thus
\[
\rho_{Z}\circ(\rho_{Z}^{\star})^{-1}(e^{T\xi})=\sum_{m=0}^{\infty}Z(\word^{\star}(\{2\}^{m}))\xi^{2m}e^{T\xi}
\]
From the coefficient of $\xi^{n}$, we get
\[
\rho_{Z}\circ(\rho_{Z}^{\star})^{-1}\left(\frac{T^{n}}{n!}\right)=\sum_{m=0}^{\left\lfloor n/2\right\rfloor }Z(\word^{\star}(\{2\}^{m}))\xi^{2m}\frac{T^{n-2m}}{(n-2m)!},
\]
and thus the constant term of $\rho_{Z}\circ(\rho_{Z}^{\star})^{-1}\left(\frac{T^{n}}{n!}\right)$
is given by
\[
\begin{cases}
Z(\word^{\star}(\{2\}^{n/2})) & n:{\rm even}\\
0 & n:{\rm odd}.
\end{cases}
\]
Furthermore, we have
\[
Z(\word^{\star}(\{2\}^{n/2}))=2(1-2^{1-n})Z(E_{1}E_{0}^{n-1})
\]
for even $n>0$ by \cite[Theorem 3.13 (2)]{Li-Qin_DS}. Therefore,
\[
Z_{\shuffle}^{(0)}(\word^{\star}(\Bbbk))=\sum_{\Bbbk=(\Bbbk'',\{1\}^{2m})}Z\circ\shreg(W(\bullet\!\!-\!\!\boxed{\Bbbk''}))\times\begin{cases}
1 & m=0\\
2(1-2^{1-2m})Z(E_{1}E_{0}^{2m-1}) & m>0.
\end{cases}
\]
Now, let us prove the lemma. If we put $w=\varrho(\word^{\star}(\Bbbk))$
then
\begin{align*}
L_{1}(\shreg(\varrho(w))) & =L_{1}(\shreg(\word^{\star}(\Bbbk)))\\
 & =\sum_{\Bbbk=(\Bbbk'',\{1\}^{2m})}L_{1}(\shreg(W(\bullet\!\!-\!\!\boxed{\Bbbk''})))\times\begin{cases}
1 & m=0\\
2(1-2^{1-2m})L_{1}(E_{1}E_{0}^{2m-1}) & m>0.
\end{cases}\\
 & =\sum_{\Bbbk=(\Bbbk'',\{1\}^{2m})}L_{1}(\shreg(\theta(\Bbbk'')))\times\begin{cases}
1 & m=0\\
-2L_{1}(E_{-1}E_{0}^{2m-1}) & m>0
\end{cases}\\
 & =L_{1}\circ\theta_{1}(\Bbbk)\\
 & =L_{1}(\wp(\varrho(\word^{\star}(\Bbbk))))\\
 & =L_{1}(\wp(w)).
\end{align*}
Here, the first equality is trivial, the second equality is because
$L_{1}$ satisfies the double shuffle relation, the third equality
is by the distribution relation (Lemma \ref{lem:distr}), the fourth
equality is by the duality relation of multiple zeta values and Lemma
\ref{lem:distr}, and the last two equalities are by definition. Since
any element of $\mathcal{C}$ can be written as a linear sum of $\varrho(\word^{\star}(\Bbbk))$,
we complete the proof.
\end{proof}
\begin{lem}
\label{lem:depth_preserve_wp}$\wp(\mathcal{A}_{k,d}(\{0,1\}))\subset\tilde{\mathcal{A}}_{k,d}(\{0,1,-1\})$
\end{lem}

\begin{proof}
Let $k>0$. It is enough to show that $\wp(u)\in\mathcal{A}_{k,d}(\{0,1,-1\})$
for $u=e_{1}e_{0}^{l_{1}-1}\cdots e_{1}e_{0}^{l_{d}-1}$ with $l_{1}+\cdots+l_{d}=k$.
Then
\begin{align*}
\varrho(u) & =(-e_{1})(e_{0}-e_{1})^{l_{1}-1}(-e_{1})(e_{0}-e_{1})^{l_{2}-1}\cdots(-e_{1})(e_{0}-e_{1})^{l_{d}-1}\\
 & =-e_{1}(e_{0}-e_{1})^{l_{1}-1}((e_{0}-e_{1})-e_{0})(e_{0}-e_{1})^{l_{2}-1}\cdots((e_{0}-e_{1})-e_{0})(e_{0}-e_{1})^{l_{d}-1}
\end{align*}
can be written as a $\mathbb{Q}$-linear sum of
\[
\word^{\star}(k_{1},\dots,k_{r})
\]
with $r\geq k-d+1$. In other words, $u$ can be written as a $\mathbb{Q}$-linear
sum of
\[
\varrho(\word^{\star}(k_{1},\dots,k_{r}))
\]
with $r\geq k-d+1$, and thus $\wp(u)$ can be written as a $\mathbb{Q}$-linear
sum of
\[
\wp(\varrho(\word^{\star}(k_{1},\dots,k_{r})))=\theta_{1}(k_{1},\dots,k_{r})
\]
with $r\geq k-d+1$. For any $k_{1},\dots,k_{r}$ with $k_{1}+\cdots+k_{r}=k$,
if
\[
(k_{1},\dots,k_{r})
\]
has the form
\[
(\Bbbk',\{1\}^{2m})\ \ \ (m\geq0),
\]
then
\[
\varsigma(\theta(\Bbbk'))\in\mathcal{A}_{k-2m,k-r}(\{0,1\}),
\]
and especially if $r\geq k-d+1$ then
\[
\varsigma(\theta(\Bbbk'))\in\mathcal{A}_{k-2m,d-1}(\{0,1\}).
\]

By (\ref{eq:dist_in_tilde}) and the facts that neither $\dist$ nor
$\shreg$ change the depth and that taking shuffle product with $e_{-1}e_{0}^{2m-1}$
increases depth by one, we can now conclude that
\[
\wp(u)\in\tilde{\mathcal{A}}_{k,d}(\{0,1,-1\}).\qedhere
\]
\end{proof}
Combining the arguments above, we can construct a map $\reg_{z\to0}:\mathcal{B}\to\mathcal{A}^{0}(\{0,1,-1\})$
with desirable properties as follows.
\begin{defn}
\label{def:reg_limit_Euler}Define $\reg_{z\to0}:\mathcal{B}\to\tilde{\mathcal{A}}^{0}(\{0,1,-1\})$
as the composite map
\[
\mathcal{B}\xrightarrow{e_{-1}\mapsto0}\mathcal{B}'\xrightarrow[\simeq]{u\shuffle v\mapsto u\otimes v}\mathcal{B}''\otimes\mathcal{B}'''\xrightarrow{\mathop{{\rm \overline{reg}}}_{z\to0}\otimes(e_{a}\mapsto e_{-a/z^{2}})}\mathcal{A}^{0}(\{0,1\})\otimes\mathcal{C}\xrightarrow{u\otimes v\mapsto\dist(u)\shuffle\wp(v)}\tilde{\mathcal{A}}^{0}(\{0,1,-1\}).
\]
Here, $e_{-1}\mapsto0$ in the first arrow denotes the ring homomorphism
that annihilates the monomials containing $e_{-1}$, and $e_{a}\mapsto e_{-a/z^{2}}$
in the third map denotes the ring homomorphism that maps $e_{a}$
to $e_{-a/z^{2}}$ for $a\in\{0,-z^{2}\}$. By definition and the
arguments above, we have
\begin{equation}
\Reg_{z\to+0}L_{z}(u)=L_{1}(\reg_{z\to0}u)\label{eq:reg_z0_val}
\end{equation}
for $u\in\mathcal{B}$. Furthermore, by the construction and Lemma
\ref{lem:depth_preserve_wp}, each step preserves the depth filtration.
Hence we have
\begin{equation}
\reg_{z\to0}(\mathcal{B}_{k,d})\subset\tilde{\mathcal{A}}_{k,d}^{0}(\{0,1,-1\})\label{eq:reg_z0_kd}
\end{equation}
for any $k,d$.
\end{defn}

\subsection{Auxiliary congruences for $\protect\reg_{z\to0}$}

In this section, we will prove the following auxiliary congruences
which will be used in the proof of Lemma \ref{lem:main_step2_2}.
\begin{lem}
\label{lem:reg_z0_mod4_congruence}For $d>0$ and $k_{1},\dots,k_{d}>0$,
we have
\begin{multline*}
\reg_{z\to0}(e_{-z^{2}}e_{0}^{k_{1}-1}\cdots e_{-z^{2}}e_{0}^{k_{d}-1})\\
\equiv\begin{cases}
2e_{-1}e_{0}^{k-1} & \text{\ensuremath{k} is even}\\
0 & \text{\ensuremath{k} is odd}
\end{cases}+\begin{cases}
2(e_{1}+e_{-1})^{k-1}e_{0} & \Bbbk=\{1\}^{k}\\
0 & \Bbbk\neq\{1\}^{k}
\end{cases}\pmod{4\mathcal{A}_{k,d}^{0}(\{0,1,-1\})}.
\end{multline*}
\end{lem}

\begin{proof}
For any monomial $u$ in $\mathcal{C},$ 
\[
\dist\circ\shreg\circ\varrho(u)\equiv\begin{cases}
1 & u=1\\
(-1)^{k}2(e_{1}+e_{-1})^{k-1}e_{0} & u=e_{1}e_{0}^{k-1}\ \text{for some }k>1\\
0 & \text{otherwise}
\end{cases}
\]
modulo $4\mathcal{A}^{0}(\{0,1,-1\})$, since if there are more than
one $e_{1}$ in $u$ then there are more than one $e_{0}$ in each
monomial that appears in $\shreg\circ\varrho(u)$ and $\dist$ maps
$e_{0}$ to $2e_{0}$. Thus
\[
\theta'(\Bbbk)\equiv\begin{cases}
1 & \Bbbk=\emptyset\\
(-1)^{k+1}2(e_{1}+e_{-1})^{k-1}e_{0} & \Bbbk=(k),\ k>1\\
0 & \text{otherwise}
\end{cases}
\]
modulo $4\mathcal{A}^{0}(\{0,1,-1\})$. Therefore
\begin{align}
\theta_{1}(\Bbbk) & \equiv\begin{cases}
1 & \Bbbk=\emptyset\\
(-1)^{k+1}2(e_{1}+e_{-1})^{k-1}e_{0} & \Bbbk=(k),\ k>1\\
2e_{-1}e_{0}^{2m-1} & \Bbbk=\{1\}^{2m},\ m>1\\
0 & \text{otherwise}
\end{cases}\label{eq:theta1_modulo4}
\end{align}
modulo $4\mathcal{A}^{0}(\{0,1,-1\})$ since
\[
2e_{-1}e_{0}^{2m-1}\shuffle\theta'(\Bbbk')\equiv0\pmod{4\mathcal{A}^{0}(\{0,1,-1\})}
\]
if $\Bbbk'$ is not empty. Note that 
\[
\varrho(\word^{\star}(l_{1},\dots,l_{r}))=e_{1}(e_{0}-e_{1})^{l_{1}-1}e_{0}(e_{0}-e_{1})^{l_{2}-1}\cdots e_{0}(e_{0}-e_{1})^{l_{r}-1}.
\]
Now, let us calculate
\begin{align*}
\reg_{z\to0}(e_{-z^{2}}e_{0}^{k_{1}-1}\cdots e_{-z^{2}}e_{0}^{k_{d}-1}) & =\wp(e_{1}e_{0}^{k_{1}-1}\cdots e_{1}e_{0}^{k_{d}-1})
\end{align*}
modulo $4\mathcal{A}^{0}(\{0,1,-1\})$. Put $k=k_{1}+\cdots+k_{d}$.
Assume that $k>0$. Note that
\[
e_{1}e_{0}^{k_{1}-1}\cdots e_{1}e_{0}^{k_{d}-1}=e_{1}e_{0}^{k_{1}-1}(e_{0}-(e_{0}-e_{1}))e_{0}^{k_{2}-1}\cdots(e_{0}-(e_{0}-e_{1}))e_{0}^{k_{d}-1}.
\]
By expanding this, we can obtain the expression of $e_{1}e_{0}^{k_{1}-1}\cdots e_{1}e_{0}^{k_{d}-1}$
as a linear sum of $\varrho(\word^{\star}(l_{1},\dots,l_{r}))$. In
other words
\[
e_{1}e_{0}^{k_{1}-1}\cdots e_{1}e_{0}^{k_{d}-1}=\sum_{\substack{{\bf l}=(l_{1},\dots,l_{r})}
}c_{{\bf l}}\varrho(\word^{\star}({\bf l}))\ \ \ c_{{\bf l}}\in\mathbb{Z}.
\]
Then $c_{(k)}$ (the coefficient of $\varrho(\word^{\star}(k))=e_{1}(e_{0}-e_{1})^{k-1}$)
is equal to
\[
\begin{cases}
(-1)^{k-1} & k_{1}=\cdots=k_{d}=1\\
0 & \text{otherwise},
\end{cases}
\]
and $c_{\{1\}^{k}}$ (the coefficient of $\varrho(\word^{\star}(\{1\}^{k}))=e_{1}e_{0}^{k-1}$)
is equal to $1$ for any $k_{1},\dots,k_{d}$. Thus
\begin{align*}
\wp(e_{1}e_{0}^{k_{1}-1}\cdots e_{1}e_{0}^{k_{d}-1}) & =\sum_{\substack{{\bf l}=(l_{1},\dots,l_{r})}
}c_{{\bf l}}\wp(\varrho(\word^{\star}({\bf l})))\\
 & =\sum_{\substack{{\bf l}=(l_{1},\dots,l_{r})}
}c_{{\bf l}}\theta_{1}({\bf l})\\
 & \equiv\begin{cases}
2e_{-1}e_{0}^{k-1} & \text{\ensuremath{k} is even}\\
0 & \text{\ensuremath{k} is odd}
\end{cases}+\begin{cases}
2(e_{1}+e_{-1})^{k-1}e_{0} & \Bbbk=\{1\}^{k}\\
0 & \Bbbk\neq\{1\}^{k}
\end{cases}
\end{align*}
modulo $4\mathcal{A}^{0}(\{0,1,-1\})$.
\end{proof}

\subsection{Confluence relations for Euler sums}
\begin{defn}
\label{def:diff_operator_words}Note that any monomial $u=e_{b_{1}}\cdots e_{b_{k}}\in\mathcal{B}$
can be uniquely expressed in the form 
\[
u=e_{a_{1}}^{l_{1}}\cdots e_{a_{r}}^{l_{r}}\qquad(\substack{a_{1}\neq a_{2}\neq\cdots\neq a_{r}\\
l_{1},l_{2},\ldots,l_{r}>0
}
).
\]
For $c\in\{0,1,-1\}$, we define a linear map $\partial_{c}:\mathcal{B}\to\mathcal{B}$
by
\[
\partial_{c}(e_{a_{1}}^{l_{1}}\cdots e_{a_{r}}^{l_{r}})=\sum_{i=1}^{r}{\rm ord}_{z=c}\left(\frac{a_{i}-a_{i+1}}{a_{i}-a_{i-1}}\right)e_{a_{1}}^{l_{1}}\cdots e_{a_{i-1}}^{l_{i-1}}e_{a_{i}}^{l_{i}-1}e_{a_{i+1}}^{l_{i+1}}\cdots e_{a_{r}}^{l_{r}}
\]
with $a_{0}=0$ and $a_{r+1}=z$.
\end{defn}

As a special case of the differential formula for a general iterated
integral (\cite[Lemma 3.3.30]{Panzer}), we obtain

\begin{equation}
\frac{d}{dz}L_{z}(u)=\sum_{c\in\{0,1,-1\}}\frac{1}{z-c}L_{z}(\partial_{c}u)\label{eq:B_Lz_dif}
\end{equation}
for $u\in\mathcal{B}$. It should be noted that $\partial_{c}$ maps
$\mathcal{B}_{k,d}$ to $\mathcal{B}_{k-1,d-|c|}$ for $c\in\{0,1,-1\}$.

We define a linear map $L_{z}^{\shuffle}:\mathcal{A}(\{0,1,-1\})\to\mathbb{R}$
by
\[
L_{z}^{\shuffle}(u)=I(0';u;z)
\]
where $0'=\overrightarrow{1}_{0}.$ Note that
\begin{equation}
\Reg_{z\to+0}L_{z}^{\shuffle}(e_{a_{1}}\cdots e_{a_{k}})=\delta_{k,0},\label{eq:Lz_sh_lim}
\end{equation}
\begin{equation}
\frac{d}{dz}L_{z}^{\shuffle}(ue_{a})=\frac{1}{z-a}L_{z}^{\shuffle}(u).\label{eq:Lz_sh_dif}
\end{equation}

\begin{defn}
We define $\varphi_{\otimes}:\mathcal{B}\to\tilde{\mathcal{A}}^{0}(\{0,1,-1\})\otimes\mathcal{A}(\{0,1,-1\})$
by
\[
\varphi_{\otimes}(u)=\sum_{l=0}^{\infty}\sum_{c_{1},\dots,c_{l}\in\{0,1,-1\}}\reg_{z\to0}(\partial_{c_{1}}\cdots\partial_{c_{l}}u)\otimes e_{c_{1}}\cdots e_{c_{l}}.
\]
\end{defn}

\begin{lem}
\label{lem:conf_tensor}For $u\in\mathcal{B}$, we have 
\[
L_{z}(u)=(L_{1}\otimes L_{z}^{\shuffle})\circ\varphi_{\otimes}(u).
\]
\end{lem}

\begin{proof}
We prove the claim by induction on the degree of $u$. By (\ref{eq:Lz_sh_lim})
and (\ref{eq:reg_z0_val}), we have
\begin{equation}
\Reg_{z\to+0}(L_{1}\otimes L_{z}^{\shuffle})\circ\varphi_{\otimes}(u)=L_{1}(\reg_{z\to0}u)=\Reg_{z\to+0}L_{z}(u).\label{eq:conf_e1}
\end{equation}
By (\ref{eq:Lz_sh_dif}), we have
\begin{align}
\frac{d}{dz}(L_{1}\otimes L_{z}^{\shuffle})\circ\varphi_{\otimes}(u) & =\sum_{l=1}^{\infty}\sum_{c_{1},\dots,c_{l}\in\{0,1,-1\}}\frac{1}{z-c_{l}}(L_{1}\otimes L_{z}^{\shuffle})\left(\reg_{z\to0}(\partial_{c_{1}}\cdots\partial_{c_{l}}u)\otimes e_{c_{1}}\cdots e_{c_{l-1}}\right)\nonumber \\
 & =\sum_{c\in\{0,1,-1\}}\frac{1}{z-c}(L_{1}\otimes L_{z}^{\shuffle})\circ\varphi_{\otimes}(\partial_{c}u).\label{eq:conf_e2}
\end{align}
Thus by (\ref{eq:conf_e2}), induction hypothesis, and (\ref{eq:B_Lz_dif}),
we have
\begin{align}
\frac{d}{dz}(L_{1}\otimes L_{z}^{\shuffle})\circ\varphi_{\otimes}(u) & =\sum_{c\in\{0,1,-1\}}\frac{1}{z-c}L_{z}(\partial_{c}u)\nonumber \\
 & =\frac{d}{dz}L_{z}(u).\label{eq:conf_e3}
\end{align}
Now the lemma readily follows from (\ref{eq:conf_e1}) and (\ref{eq:conf_e3}).
\end{proof}
\begin{defn}
Define $\varphi:\mathcal{B}\to\mathcal{A}^{0}(\{0,1,-1\})$ as the
composite map
\[
\mathcal{B}\xrightarrow{\varphi_{\otimes}}\tilde{\mathcal{A}}^{0}(\{0,1,-1\})\otimes\mathcal{A}(\{0,1,-1\})\xrightarrow{{\rm id}\otimes\shreg}\tilde{\mathcal{A}}^{0}(\{0,1,-1\})\otimes\mathcal{A}^{0}(\{0,1,-1\})\xrightarrow{\shuffle}\mathcal{A}^{0}(\{0,1,-1\}).
\]
\end{defn}

Since 
\[
\varphi_{\otimes}(\mathcal{B}_{k,d})\subset\sum_{\substack{k'+k''=k\\
d'+d''=d
}
}\tilde{\mathcal{A}}_{k',d'}^{0}(\{0,1,-1\})\otimes\mathcal{A}_{k'',d''}(\{0,1,-1\})
\]
we have
\begin{equation}
\varphi(\mathcal{B}_{k,d})\subset\mathcal{A}_{k,d}^{0}(\{0,1,-1\})\label{eq:phi_Bkd_in_Akd}
\end{equation}
for any $k,d\geq0$.

\begin{defn}[Confluence relations of level two]
We put
\[
\iconf\coloneqq\{\restr u{z\to1}-\varphi(u)\mid u\in\mathbb{Q}\otimes\mathcal{B}\}\subset\mathbb{Q}\otimes\mathcal{A}^{0}(\{0,1,-1\})
\]
where $\restr u{z\to1}$ is an element of $\mathbb{Q}\otimes\mathcal{A}(\{0,1,-1\})$
obtained by replacing all $e_{p(z)}$ in $u$ with $e_{p(1)}$.
\end{defn}

\begin{thm}
\label{thm:Confluecne_relations_for_Euler_sums}We have 
\[
\iconf\subset\ker\left(L_{1}:\mathbb{Q}\otimes\mathcal{A}^{0}(\{0,1,-1\})\to\mathbb{R}\right).
\]
\end{thm}

\begin{proof}
Let $u\in\mathcal{B}$. We prove the theorem by taking the regularized
limit $\Reg_{z\to1-0}$ of Lemma \ref{lem:conf_tensor}:
\begin{equation}
\Reg_{z\to1-0}L_{z}(u)=\Reg_{z\to1-0}(L_{1}\otimes L_{z}^{\shuffle})\circ\varphi_{\otimes}(u).\label{eq:conf_e7}
\end{equation}
For the left-hand side, we have
\begin{equation}
\Reg_{z\to1-0}L_{z}(u)=\lim_{z\to1-0}L_{z}(u)=L_{1}(\restr u{z\to1})\label{eq:conf_e4}
\end{equation}
by definition. 

Now we compute the right-hand side. We obtain
\begin{equation}
\Reg_{z\to1-0}L_{z}^{\shuffle}(w)=L_{1}\circ\shreg(w)\ \ \ (w\in\mathcal{A}(\{0,1,-1\}))\label{eq:conf_e6}
\end{equation}
by the following properties of $f(w)\coloneqq\Reg_{z\to1-0}L_{z}^{\shuffle}(w)$:
\[
f(w_{1}\shuffle w_{2})=f(w_{1})\times f(w_{2}),\ \ \ f(e_{0})=f(e_{1})=0,\ \ \ \restr f{\mathcal{A}^{0}(\{0,1,-1\})}=L_{1}.
\]
Here, the first and the third properties are obvious while the second
property follows from the expressions $L_{z}^{\shuffle}(e_{0})=\log z$
and $L_{z}^{\shuffle}(e_{1})=\log(1-z)$. By (\ref{eq:conf_e6}),
we have
\begin{align}
\Reg_{z\to1-0}(L_{1}\otimes L_{z}^{\shuffle})\circ\varphi_{\otimes}(u) & =(L_{1}\otimes(L_{1}\circ\shreg))\circ\varphi_{\otimes}(u)\nonumber \\
 & =L_{1}\circ\varphi(u).\label{eq:conf_e8}
\end{align}
By (\ref{eq:conf_e7}), (\ref{eq:conf_e4}) and (\ref{eq:conf_e8}),
we have
\[
L_{1}(\restr u{z\to1})=L_{1}\circ\varphi(u).
\]
Thus
\[
\restr u{z\to1}-\varphi(u)\in\ker(L_{1}),
\]
which proves the theorem.
\end{proof}

\section{\label{sec:Main-theorem}Main theorem}

\subsection{The statement and the proof overview of the main theorem}

We put 
\[
\hconf\coloneqq\left(\mathbb{Q}\otimes\mathcal{A}^{0}(\{0,1,-1\})\right)/\iconf,
\]
and
\begin{align*}
\indset^{\mathrm{D}}(k,d) & \coloneqq\{(k_{1},\dots,k_{r})\in\indset(k,d)\mid k_{1},\dots,k_{r-1}>0,k_{r}<0,k_{2}\equiv\cdots\equiv k_{r}\equiv1\pmod{2}\}\\
X^{\bullet}(k,d) & \coloneqq\vspan_{\mathbb{Z}_{(2)}}\{\word(\Bbbk)\bmod\iconf\mid\Bbbk\in\indset^{\bullet}(k,d)\}\subset\hconf\ \ \ (\bullet=\emptyset\text{ or }\mathrm{D})\\
\modx^{\bullet}(k,d) & \coloneqq\vspan_{\mathbb{Z}_{(2)}}\{\modword(\Bbbk)\bmod\iconf\mid\Bbbk\in\indset^{\bullet}(k,d)\}\subset\hconf\ \ \ (\bullet=\emptyset\text{ or }\mathrm{D})
\end{align*}
where $\mathbb{Z}_{(2)}\coloneqq\{p/q\in\mathbb{Q}\mid p\in\mathbb{Z},q\in1+2\mathbb{Z}\}$.
The purpose of this section is to prove the following theorem.
\begin{thm}
\label{thm:main_main}For any $k\geq0$ and $d\geq0$, 
\[
\modx(k,d)=\modx^{\mathrm{D}}(k,d).
\]
\end{thm}

As a restatement of this theorem, we obtain the following.
\begin{thm}
\label{thm:main_general_zeta}Let $k,d\geq0$, $R$ a $\mathbb{Q}$-vector
space, and $Z:\mathcal{A}_{k,d}^{0}(\{0,1,-1\})\to R$ a linear map
such that $Z(\iconf)=\{0\}$. Then for $\Bbbk\in\indset(k,d)$, $Z(\modword(\Bbbk))$
is a $\mathbb{Z}_{(2)}$-linear combination of $\{Z(\modword(\Bbbk'))\mid\Bbbk'\in\indset^{\mathrm{D}}(k,d)\}$.
Especially, for any $\Bbbk\in\indset(k,d)$, $\modzeta(\Bbbk)$ is
a $\mathbb{Z}_{(2)}$-linear combination of $\{\modzeta(\Bbbk')\mid\Bbbk'\in\indset^{\mathrm{D}}(k,d)\}$.
\end{thm}

The proof of Theorem \ref{thm:main_main} proceeds as follows. Note
that $\vspan_{\mathbb{Z}}\{\word(\Bbbk)\mid\Bbbk\in\indset(k,d)\}=\mathcal{A}_{k,d}^{0}(\{0,1,-1\}).$
In order to reduce $\modx(k,d)$ to $\modx^{\mathrm{D}}(k,d)$, we
introduce an intermediate space
\begin{align*}
Y(k,d) & \coloneqq\vspan_{\mathbb{Z}_{(2)}}\{\word(k_{1},\dots,k_{r})\bmod\iconf\mid(k_{1},\dots,k_{r})\in\indset(k,d),k_{1},\dots,k_{r-1}>0,k_{r}<0\}\\
 & =\vspan_{\mathbb{Z}_{(2)}}\{u\bmod\iconf\mid u\in\mathcal{A}_{k,d}^{0}(\{0,-1\})\}
\end{align*}
and consider the two-step reduction, firstly $X(k,d)$ to $Y(k,d)$
and secondly $Y(k,d)$ to $X^{\mathrm{D}}(k,d)$ (modulo $2$ and
lower depth). More precisely, we will first prove
\[
X(k,d)\subset Y(k,d)+2X(k,d)
\]
in the next section (Lemma \ref{lem:main_step1}) and then in the
next next section, we will prove
\[
Y(k,d)\subset X^{\mathrm{D}}(k,d)+2X(k,d)+\frac{1}{2}X(k,d-1)
\]
(Lemma \ref{lem:main_step2}) thereby derive Theorem \ref{thm:main_main}
by combining the two inclusions above.

\subsection{\label{subsec:proof_step1}Step 1: the proof of $X(k,d)\subset Y(k,d)+2X(k,d)$}
\begin{lem}
\label{lem:main_step1}We have
\[
X(k,d)\subset Y(k,d)+2X(k,d).
\]
\end{lem}

\begin{proof}
It suffices to prove
\begin{equation}
\left(u\bmod\iconf\right)\in Y(k,d)+2X(k,d)\label{eq:main_e1}
\end{equation}
for $u\coloneqq e_{a_{1}}\cdots e_{a_{k}}\in\mathcal{A}_{k,d}^{0}(\{0,1,-1\})$.
We will prove (\ref{eq:main_e1}) using the element $\restr{u'}{z\to1}-\varphi(u')\in\iconf$
for
\[
u'\coloneqq e_{b_{1}}\cdots e_{b_{k}}\in\mathcal{B}(k,d)\cap\mathbb{Z}\left\langle e_{0},e_{-1},e_{z}\right\rangle 
\]
where
\[
b_{j}\coloneqq\begin{cases}
a_{j} & a_{j}=0,-1\\
z & a_{j}=1.
\end{cases}
\]
By definition, we have
\begin{equation}
\restr{u'}{z\to1}=u.\label{eq:main_e4}
\end{equation}
Recall that $\varphi=(\shuffle)\circ({\rm id}\otimes\shreg)\circ\varphi_{\otimes}$.
By definition, we can write 
\[
\varphi_{\otimes}(u')=\sum_{l=0}^{k}w_{l}
\]
where
\[
w_{l}\in\sum_{d'+d''=d}\tilde{\mathcal{A}}_{l,d'}^{0}(\{0,1,-1\})\otimes\mathcal{A}_{k-l,d''}(\{0,-1\}).
\]
Note that the right side of the tensor symbol is $\mathcal{A}_{k-l,d''}(\{0,-1\})$
rather than $\mathcal{A}_{k-l,d''}(\{0,1,-1\})$ since $u'\in\mathbb{Z}\left\langle e_{0},e_{-1},e_{z}\right\rangle $.
For $w_{0}$, we have
\begin{equation}
(\shuffle)\circ({\rm id}\otimes\shreg)(w_{0})\in\mathcal{A}_{k,d}^{0}(\{0,-1\})\label{eq:main_e3}
\end{equation}
by definition, while for $w_{l}\:(l>0)$, we have
\begin{equation}
(\shuffle)\circ({\rm id}\otimes\shreg)(w_{l})\in2\mathcal{A}_{k,d}^{0}(\{0,1,-1\})\ \ \ (l>0)\label{eq:main_e2}
\end{equation}
by the definition of $\tilde{\mathcal{A}}_{l,d'}^{0}(\{0,1,-1\})$.
Therefore, by (\ref{eq:main_e3}) and (\ref{eq:main_e2}), we find
that
\begin{equation}
\varphi(u')\in\mathcal{A}_{k,d}^{0}(\{0,-1\})+2\mathcal{A}_{k,d}^{0}(\{0,1,-1\}).\label{eq:main_e5}
\end{equation}
By (\ref{eq:main_e4}) and (\ref{eq:main_e5}), we have
\begin{align*}
(u\bmod\iconf) & =(\restr{u'}{z\to1}\bmod\iconf)\\
 & =(\varphi(u')\bmod\iconf)\\
 & \in Y(k,d)+2X(k,d),
\end{align*}
which proves (\ref{eq:main_e1}).
\end{proof}

\subsection{\label{subsec:proof_step2}Step 2: the proofs of $Y(k,d)\subset X^{\mathrm{D}}(k,d)+2X(k,d)+\frac{1}{2}X(k,d-1)$
and the main theorem}
\begin{defn}
For $k\geq0$ and $d\geq0$, we define subspaces $M_{k,d}$ and $N_{k,d}$
of $\tilde{\mathcal{A}}^{0}(\{0,1,-1\})\otimes\mathcal{A}(\{0,1,-1\})$
by
\begin{align*}
M_{k,d} & \coloneqq\sum_{\substack{k'+k''=k\\
d'+d''=d
}
}\tilde{\mathcal{A}}_{k',d'}^{0}(\{0,1,-1\})\otimes\mathcal{A}_{k'',d''}(\{0,1,-1\})\\
N_{k,d} & \coloneqq\sum_{\substack{k'+k''=k\\
d'+d''=d\\
k'>0
}
}\tilde{\mathcal{A}}_{k',d'}^{0}(\{0,1,-1\})\otimes\mathcal{A}_{k'',d''}(\{0,1,-1\})=2\sum_{\substack{k'+k''=k\\
d'+d''=d\\
k'>0
}
}\mathcal{A}_{k',d'}^{0}(\{0,1,-1\})\otimes\mathcal{A}_{k'',d''}(\{0,1,-1\}).
\end{align*}
\end{defn}

We introduce ``cutting-off lower depth version'' $\bar{\partial}_{c}$
(resp. $\bar{\varphi}_{\otimes}$) and ``mod $2$ version'' $\bar{\bar{\partial}}_{c}$
(resp. $\bar{\bar{\varphi}}_{\otimes}$) of $\partial_{c}$ (resp.
$\varphi_{\otimes}$) as follows. Note that $\partial_{c}$ can also
be written as 
\[
\partial_{c}(e_{a_{1}}\cdots e_{a_{k}})\coloneqq\sum_{i=1}^{k}e_{a_{1}}\cdots e_{a_{i-1}}e_{a_{i+1}}\cdots e_{a_{k}}\times\sum_{p\in\{1,-1\}}p\cdot{\rm ord}_{z-c}(a_{i+p}-a_{i}),
\]
where we have put ${\rm ord}_{z-c}(0)\coloneqq0$. 
\begin{defn}
For $c\in\{0,1,-1\}$, we define linear maps $\bar{\partial}_{c},\bar{\bar{\partial}}_{c}:\mathcal{B}\to\mathcal{B}$
by
\begin{align*}
\bar{\partial}_{c}(e_{a_{1}}\cdots e_{a_{k}}) & \coloneqq\sum_{i=1}^{k}e_{a_{1}}\cdots e_{a_{i-1}}e_{a_{i+1}}\cdots e_{a_{k}}\times\sum_{p\in\{1,-1\}}\begin{cases}
0 & c=0\text{ and }a_{i}\neq0\\
p\cdot{\rm ord}_{z-c}(a_{i+p}-a_{i}) & c\neq0\text{ or }a_{i}=0
\end{cases}\\
\bar{\bar{\partial}}_{c}(e_{a_{1}}\cdots e_{a_{k}}) & \coloneqq\sum_{i=1}^{k}e_{a_{1}}\cdots e_{a_{i-1}}e_{a_{i+1}}\cdots e_{a_{k}}\times\sum_{p\in\{1,-1\}}\begin{cases}
0 & c=0\text{ and }a_{i}\neq0\\
p\cdot\left({\rm ord}_{z-c}(a_{i+p}-a_{i})\bmod2\right) & c\neq0\text{ or }a_{i}=0
\end{cases}
\end{align*}
with $a_{0}=0$ and $a_{k+1}=z$. Here, $\left(x\bmod2\right)$ is
considered to be an element of $\{0,1\}$. By definition, $\partial_{c}=\bar{\partial}_{c}=\bar{\bar{\partial}}_{c}$
if $|c|=1$.
\end{defn}

\begin{defn}
We define $\bar{\varphi}_{\otimes},\bar{\bar{\varphi}}_{\otimes}:\mathcal{B}\to\tilde{\mathcal{A}}^{0}(\{0,1,-1\})\otimes\mathcal{A}(\{0,1,-1\})$
by
\begin{align*}
\bar{\varphi}_{\otimes}(u) & =\sum_{l=0}^{\infty}\sum_{c_{1},\dots,c_{l}\in\{0,1,-1\}}\reg_{z\to0}(\bar{\partial}_{c_{1}}\cdots\bar{\partial}_{c_{l}}u)\otimes e_{c_{1}}\dots e_{c_{l}},\\
\bar{\bar{\varphi}}_{\otimes}(u) & =\sum_{l=0}^{\infty}\sum_{c_{1},\dots,c_{l}\in\{0,1,-1\}}\reg_{z\to0}(\bar{\bar{\partial}}_{c_{1}}\cdots\bar{\bar{\partial}}_{c_{l}}u)\otimes e_{c_{1}}\dots e_{c_{l}}.
\end{align*}
\end{defn}

By definition, 
\begin{align*}
\bar{\partial}_{c}(u) & \equiv\partial_{c}(u)\pmod{\mathcal{B}_{k-1,d-1-|c|}}\\
\bar{\bar{\partial}}_{c}(u) & \equiv\partial_{c}(u)\pmod{\mathcal{B}_{k-1,d-1-|c|}+2\mathcal{B}_{k-1,d-|c|}}
\end{align*}
for $u\in\mathcal{B}_{k,d}$. Thus we have
\begin{align}
\bar{\varphi}_{\otimes}(u) & \equiv\varphi_{\otimes}(u)\pmod{M_{k,d-1}}\label{eq:main_phibar}\\
\bar{\bar{\varphi}}_{\otimes}(u) & \equiv\varphi_{\otimes}(u)\pmod{2M_{k,d}+M_{k,d-1}}\label{eq:main_phibarbar}
\end{align}
for $u\in\mathcal{B}_{k,d}$.
\begin{lem}
\label{lem:main_step2_1}For all $u\in\mathcal{B}_{k,d}\cap\mathbb{Z}\left\langle e_{0},e_{-1},e_{-z^{2}}\right\rangle $,
we have
\[
\varphi_{\otimes}(u)\equiv1\otimes\restr u{z\to1}\pmod{N_{k,d}+2M_{k,d}+M_{k,d-1}}.
\]
\end{lem}

\begin{proof}
Since 
\[
\varphi_{\otimes}(u)\equiv\bar{\bar{\varphi}}_{\otimes}(u)\pmod{2M_{k,d}+M_{k,d-1}}
\]
by (\ref{eq:main_phibarbar}) and 
\[
\bar{\bar{\varphi}}_{\otimes}(u)\equiv\sum_{c_{1},\dots,c_{k}\in\{0,1,-1\}}\bar{\bar{\partial}}_{c_{1}}\cdots\bar{\bar{\partial}}_{c_{k}}u\otimes e_{c_{1}}\cdots e_{c_{k}}\pmod{N_{k,d}}
\]
by definition of $N_{k,d}$ and $\bar{\bar{\varphi}}_{\otimes}$,
the claim is equivalent to
\[
\sum_{c_{1},\dots,c_{k}\in\{0,1,-1\}}\bar{\bar{\partial}}_{c_{1}}\cdots\bar{\bar{\partial}}_{c_{k}}u\otimes e_{c_{1}}\cdots e_{c_{k}}\equiv1\otimes\restr u{z\to1}\pmod{N_{k,d}+2M_{k,d}+M_{k,d-1}}
\]
for $u\in\mathcal{B}_{k,d}\cap\mathbb{Z}\left\langle e_{0},e_{-1},e_{-z^{2}}\right\rangle $.
In fact, we shall prove a more precise equality
\begin{equation}
\sum_{c_{1},\dots,c_{k}\in\{0,1,-1\}}\bar{\bar{\partial}}_{c_{1}}\cdots\bar{\bar{\partial}}_{c_{k}}u\otimes e_{c_{1}}\cdots e_{c_{k}}=1\otimes\restr u{z\to1}\label{eq:main_t1}
\end{equation}
by induction on $k$. The $k=0$ case is obvious. Assume that $k>0$.
It is sufficient to consider the case where $u$ is a monomial. Define
$a_{1},\dots,a_{k}\in\{0,-1,-z^{2}\}$ and $a'_{1},\dots,a'_{k}\in\{0,-1\}$
by
\begin{align*}
e_{a_{1}}\cdots e_{a_{k}} & =u,\\
e_{a'_{1}}\cdots e_{a'_{k}} & =\restr u{z\to1}.
\end{align*}
By definition of $\bar{\bar{\partial}}_{c}$, we have
\[
\bar{\bar{\partial}}_{c}u=|c|\cdot u'+\delta_{c,a'_{k}}e_{a_{1}}\cdots e_{a_{k-1}}
\]
where
\[
u'=\sum_{\substack{1\leq i\leq k-1\\
\{a_{i},a_{i+1}\}=\{-1,-z^{2}\}
}
}e_{a_{1}}\cdots e_{a_{i-1}}(e_{a_{i+1}}-e_{a_{i}})e_{a_{i+2}}\cdots e_{a_{k}}.
\]
Thus, by induction hypothesis, 
\begin{align*}
 & \sum_{c_{1},\dots,c_{k}\in\{0,1,-1\}}\bar{\bar{\partial}}_{c_{1}}\cdots\bar{\bar{\partial}}_{c_{k}}u\otimes e_{c_{1}}\cdots e_{c_{k}}\\
 & =\sum_{c\in\{0,1,-1\}}\sum_{c_{1},\dots,c_{k-1}\in\{0,1,-1\}}\bar{\bar{\partial}}_{c_{1}}\cdots\bar{\bar{\partial}}_{c_{k-1}}\left(|c|\cdot u'+\delta_{c,a'_{k}}e_{a_{1}}\cdots e_{a_{k-1}}\right)\otimes e_{c_{1}}\cdots e_{c_{k-1}}e_{c}\\
 & =\sum_{c\in\{0,1,-1\}}1\otimes\restr{\left(|c|\cdot u'+\delta_{c,a'_{k}}e_{a_{1}}\cdots e_{a_{k-1}}\right)}{z\to1}e_{c}\\
 & =\sum_{c\in\{0,1,-1\}}1\otimes\delta_{c,a'_{k}}e_{a'_{1}}\cdots e_{a'_{k-1}}e_{c}\\
 & =1\otimes\restr u{z\to1},
\end{align*}
which proves (\ref{eq:main_t1}).
\end{proof}
\begin{lem}
\label{lem:main_step2_2}Let $d\geq1$, $k_{1},\dots,k_{d}\geq1$,
and $b_{1},\dots,b_{d}\in\{-1,-z^{2}\}$ such that $b_{1}=-1$ or
$k_{1}=1$. Put
\[
u\coloneqq e_{b_{1}}e_{0}^{k_{1}-1}\cdots e_{b_{d}}e_{0}^{k_{d}-1}\in\mathcal{B}_{k,d}
\]
where $k\coloneqq k_{1}+\cdots+k_{d}$. Assume that $e_{-1}$ and
$e_{-z^{2}}$ are not adjacent in $u$, i.e., $k_{i}>1$ for all $i$
such that $b_{i}\neq b_{i+1}$. Then we have
\[
\varphi_{\otimes}(u)\equiv1\otimes\biggl(2\sum_{\substack{1\leq i<j\leq d\\
b_{i}\neq b_{i+1}\\
k_{j}\in1+2\mathbb{Z}
}
}\word^{-}(\Bbbk^{(i,j)})+\Bigl(1+2\sum_{\substack{1\leq j\leq d\\
b_{j}=-z^{2}\\
k_{j}\in2\mathbb{Z}
}
}1\Bigr)\word^{-}(\Bbbk)\biggr)\pmod{2N_{k,d}+4M_{k,d}+M_{k,d-1}}
\]
where we put
\[
\Bbbk\coloneqq(k_{1},\dots,k_{d}),
\]
\[
\Bbbk^{(i,j)}\coloneqq(k_{1},\dots,k_{i-1},k_{i}-1,k_{i+1},\dots,k_{j-1},k_{j}+1,k_{j+1},\dots,k_{d}),
\]
and
\[
\word^{-}(l_{1},\dots,l_{d})\coloneqq\word(l_{1},\dots,l_{d-1},-l_{d})=e_{-1}e_{0}^{l_{1}-1}\cdots e_{-1}e_{0}^{l_{d}-1}
\]
for positive integers $l_{1},\dots,l_{d}$.
\end{lem}

\begin{proof}
Define $a_{1},\dots,a_{k}\in\{0,-1,-z^{2}\}$ and $a'_{1},\dots,a'_{k}\in\{0,-1\}$
by
\begin{align*}
e_{a_{1}}\cdots e_{a_{k}} & =u,\\
e_{a'_{1}}\cdots e_{a'_{k}} & =\restr u{z\to1}.
\end{align*}
By (\ref{eq:main_phibar}), we have
\begin{equation}
\varphi_{\otimes}(u)\equiv\bar{\varphi}_{\otimes}(u)\pmod{M_{k,d-1}}.\label{eq:main_phi_phibar}
\end{equation}
The difference of $\bar{\partial}_{c}$ and $\bar{\bar{\partial}}_{c}$
is expressed as
\[
\bar{\partial}_{c}-\bar{\bar{\partial}}_{c}=2D_{c}
\]
where $D_{c}$ is a linear map defined by $D_{\pm1}\coloneqq0$ and
\[
D_{0}(e_{x_{1}}\cdots e_{x_{l}})\coloneqq\sum_{i=1}^{l}e_{x_{1}}\cdots e_{x_{i-1}}e_{x_{i+1}}\cdots e_{x_{l}}\times\sum_{p\in\{1,-1\}}p\cdot\begin{cases}
1 & (x_{i},x_{i+p})=(0,-z^{2})\\
0 & {\rm otherwise}.
\end{cases}
\]
Thus
\begin{align}
\bar{\varphi}_{\otimes}(u) & =\sum_{l=0}^{k}\sum_{c_{1},\dots,c_{l}\in\{0,1,-1\}}\reg_{z\to0}(\bar{\partial}_{c_{1}}\cdots\bar{\partial}_{c_{l}}u)\otimes e_{c_{1}}\dots e_{c_{l}}\nonumber \\
 & =\sum_{s=0}^{k}\sum_{c_{s+1},\dots,c_{k}\in\{0,1,-1\}}\reg_{z\to0}(\bar{\partial}_{c_{s+1}}\cdots\bar{\partial}_{c_{k}}u)\otimes e_{c_{s+1}}\dots e_{c_{k}}\ \ \ (s=k-l)\nonumber \\
 & =2F_{1}+F_{2}\label{eq:main_e8}
\end{align}
where
\[
F_{1}=\sum_{0\leq s<t\leq k}\sum_{c_{s+1},\dots,c_{k}\in\{0,1,-1\}}\reg_{z\to0}(\bar{\partial}_{c_{s+1}}\cdots\bar{\partial}_{c_{t-1}}D_{c_{t}}\bar{\bar{\partial}}_{c_{t+1}}\cdots\bar{\bar{\partial}}_{c_{k}}u)\otimes e_{c_{s+1}}\dots e_{c_{k}}
\]
and
\[
F_{2}=\sum_{s=0}^{k}\sum_{c_{s+1},\dots,c_{k}\in\{0,1,-1\}}\reg_{z\to0}(\bar{\bar{\partial}}_{c_{s+1}}\cdots\bar{\bar{\partial}}_{c_{k}}u)\otimes e_{c_{s+1}}\dots e_{c_{k}},
\]
which is obtained by expanding as 
\begin{align*}
\bar{\partial}_{c_{s+1}}\cdots\bar{\partial}_{c_{k}} & =2\bar{\partial}_{c_{s+1}}\cdots\bar{\partial}_{c_{k-1}}D_{c_{k}}+\bar{\partial}_{c_{s+1}}\cdots\bar{\partial}_{c_{k-1}}\bar{\bar{\partial}}_{c_{k}}\\
 & =2\bar{\partial}_{c_{s+1}}\cdots\bar{\partial}_{c_{k-1}}D_{c_{k}}+2\bar{\partial}_{c_{s+1}}\cdots\bar{\partial}_{c_{k-2}}D_{c_{k-1}}\bar{\bar{\partial}}_{c_{k}}+\bar{\partial}_{c_{s+1}}\cdots\bar{\partial}_{c_{k-2}}\bar{\bar{\partial}}_{c_{k-1}}\bar{\bar{\partial}}_{c_{k}}\\
 & =\cdots.
\end{align*}
Here, we can easily check that
\[
\bar{\bar{\partial}}_{c_{t+1}}\cdots\bar{\bar{\partial}}_{c_{k}}u=\begin{cases}
e_{a_{1}}\cdots e_{a_{t}} & c_{j}=a'_{j}\text{ for all }t+1\leq j\leq k\\
0 & {\rm otherwise}
\end{cases}
\]
by the assumption that $e_{-1}$ and $e_{-z^{2}}$ are not adjacent
in $u$. Thus we have
\begin{align}
F_{1} & =\sum_{0\leq s<t\leq k}\sum_{c_{s+1},\dots,c_{t}\in\{0,1,-1\}}\reg_{z\to0}(\bar{\partial}_{c_{s+1}}\cdots\bar{\partial}_{c_{t-1}}D_{c_{t}}(e_{a_{1}}\cdots e_{a_{t}}))\otimes e_{c_{s+1}}\dots e_{c_{t}}e_{a'_{t+1}}\cdots e_{a'_{k}}\nonumber \\
 & =\sum_{0\leq s<t\leq k}\left(\sum_{c_{s+1},\dots,c_{t-1}\in\{0,1,-1\}}\reg_{z\to0}(\bar{\partial}_{c_{s+1}}\cdots\bar{\partial}_{c_{t-1}}D_{0}(e_{a_{1}}\cdots e_{a_{t}}))\otimes e_{c_{s+1}}\dots e_{c_{t-1}}\right)\times\left(1\otimes e_{0}e_{a'_{t+1}}\cdots e_{a'_{k}}\right)\nonumber \\
 & =\sum_{0<t\leq k}\bar{\varphi}_{\otimes}(D_{0}(e_{a_{1}}\cdots e_{a_{t}}))\times\left(1\otimes e_{0}e_{a'_{t+1}}\cdots e_{a'_{k}}\right)\nonumber \\
 & =\sum_{\substack{u=u'u''\\
u',u'':{\rm monomial}
}
}\bar{\varphi}_{\otimes}(D_{0}(u'))\times\left(1\otimes\restr{e_{0}u''}{z\rightarrow1}\right).\label{eq:main_e9}
\end{align}
By Lemma \ref{lem:main_step2_1}, we have
\begin{align}
\sum_{\substack{u=u'u''\\
u',u'':{\rm monomial}
}
}\bar{\varphi}_{\otimes}(D_{0}u')\times\left(1\otimes\restr{e_{0}u''}{z\rightarrow1}\right) & \equiv1\otimes\biggl(\sum_{\substack{u=u'u''\\
u',u'':{\rm monomial}
}
}\restr{\left(D_{0}(u')e_{0}u''\right)}{z\to1}\biggr)\pmod{N_{k,d}+2M_{k,d}+M_{k,d-1}}.\label{eq:main_e10}
\end{align}
Furthermore, $\sum_{\substack{u=u'u''\\
u',u'':{\rm monomial}
}
}\restr{\left(D_{0}(u')e_{0}u''\right)}{z\to1}=\sum_{1\leq j\leq d}G_{j}$ where
\begin{align*}
G_{j} & =\sum_{s=0}^{k_{j}-1}\restr{\left(D_{0}(e_{b_{1}}e_{0}^{k_{1}-1}\cdots e_{b_{j-1}}e_{0}^{k_{j-1}-1}e_{b_{j}}e_{0}^{s})\right)}{z\to1}e_{0}(e_{0}^{k_{j}-1-s}e_{-1}e_{0}^{k_{j+1}-1}\cdots e_{-1}e_{0}^{k_{d}-1})\\
 & =\sum_{s=0}^{k_{j}-1}\sum_{\substack{1\leq i<j\\
k_{i}\neq1
}
}\left(\delta_{b_{i+1},-z^{2}}-\delta_{b_{i},-z^{2}}\right)\\
 & \ \ \ \ \times\restr{\left(e_{b_{1}}e_{0}^{k_{1}-1}\cdots e_{b_{i}}e_{0}^{k_{i}-2}\cdots e_{b_{j-1}}e_{0}^{k_{j-1}-1}e_{b_{j}}e_{0}^{s}\right)}{z\to1}e_{0}^{k_{j}-s}e_{-1}e_{0}^{k_{j+1}-1}\cdots e_{-1}e_{0}^{k_{d}-1}\\
 & \ \ -\sum_{s=1}^{k_{j}-1}\delta_{b_{j},-z^{2}}\restr{\left(e_{b_{1}}e_{0}^{k_{1}-1}\cdots e_{b_{j-1}}e_{0}^{k_{j-1}-1}e_{b_{j}}e_{0}^{s-1}\right)}{z\to1}e_{0}^{k_{j}-s}e_{-1}e_{0}^{k_{j+1}-1}\cdots e_{-1}e_{0}^{k_{d}-1}\\
 & =k_{j}\sum_{\substack{1\leq i<j\\
k_{i}\neq1
}
}\left(\delta_{b_{i+1},-z^{2}}-\delta_{b_{i},-z^{2}}\right)\word^{-}(\Bbbk^{(i,j)})-(k_{j}-1)\delta_{b_{j},-z^{2}}\word^{-}(\Bbbk).
\end{align*}
It follows that
\begin{align}
\sum_{\substack{u=u'u''\\
u',u'':{\rm monomial}
}
}\restr{\left(D_{0}(u')e_{0}u''\right)}{z\to1} & \equiv\sum_{\substack{1\leq i<j\leq d\\
b_{i}\neq b_{i+1}\\
k_{j}\in1+2\mathbb{Z}
}
}\word^{-}(\Bbbk^{(i,j)})+\Bigl(\sum_{\substack{1\leq j\leq d\\
b_{j}=-z^{2}\\
k_{j}\in2\mathbb{Z}
}
}1\Bigr)\word^{-}(\Bbbk)\mod{2\mathcal{A}_{k,d}(\{0,1,-1\})}.\label{eq:main_e11}
\end{align}
By (\ref{eq:main_e9}), (\ref{eq:main_e10}) and (\ref{eq:main_e11}),
we have
\begin{equation}
F_{1}\equiv1\otimes\biggl(\sum_{\substack{1\leq i<j\leq d\\
b_{i}\neq b_{i+1}\\
k_{j}\in1+2\mathbb{Z}
}
}\word^{-}(\Bbbk^{(i,j)})+\sum_{\substack{1\leq j\leq d\\
b_{j}=-z^{2}\\
k_{j}\in2\mathbb{Z}
}
}\word^{-}(\Bbbk)\biggr)\pmod{N_{k,d}+2M_{k,d}+M_{k,d-1}}.\label{eq:main_e12}
\end{equation}
Similarly, we have
\begin{align*}
F_{2} & =\sum_{s=0}^{k}\sum_{c_{s+1},\dots,c_{k}\in\{0,1,-1\}}\reg_{z\to0}(\bar{\bar{\partial}}_{c_{s+1}}\cdots\bar{\bar{\partial}}_{c_{k}}u)\otimes e_{c_{s+1}}\dots e_{c_{k}}\\
 & =\sum_{s=0}^{k}\reg_{z\to0}(e_{a_{1}}\cdots e_{a_{s}})\otimes e_{a'_{s+1}}\dots e_{a'_{k}}.
\end{align*}
Here, $\reg_{z\to0}(e_{a_{1}}\cdots e_{a_{s}})=0$ except when $a_{1},\dots,a_{s}\in\{0,-z^{2}\}$.
On the other hand, if $a_{1},\dots,a_{s}\in\{0,-z^{2}\}$ and $s>0$,
we have
\[
\reg_{z\to0}(e_{a_{1}}\cdots e_{a_{s}})\equiv\begin{cases}
2e_{-1}e_{0}^{s-1} & \text{\ensuremath{s} is even}\\
0 & \text{\ensuremath{s} is odd}
\end{cases}+\begin{cases}
2(e_{1}+e_{-1})^{s-1}e_{0} & a_{1}=\cdots=a_{s}=-z^{2}\\
0 & \text{otherwise}
\end{cases}
\]
modulo $4\mathcal{A}_{s,d'}^{0}(\{0,1,-1\})$ with $d'=\#\{1\leq j\leq s\mid a_{j}=-z^{2}\}$
by Lemma \ref{lem:reg_z0_mod4_congruence}. Furthermore, 
\[
2e_{-1}e_{0}^{s-1}\in\mathcal{A}_{s,d'-1}^{0}(\{0,1,-1\})
\]
if $d'>1$, and if $a_{1}=\cdots=a_{s}=-z^{2}$ then
\[
2(e_{1}+e_{-1})^{s-1}e_{0}\in\mathcal{A}_{s,d'-1}^{0}(\{0,1,-1\})
\]
since $d'=s$ in this case. Thus for any $a_{1},\dots,a_{s}\in\{0,-z^{2}\}$
\[
\reg_{z\to0}(e_{a_{1}}\cdots e_{a_{s}})\otimes e_{a'_{s+1}}\dots e_{a'_{k}}
\]
is congruent to
\[
\begin{cases}
2e_{-1}e_{0}^{s-1}\otimes e_{a'_{s+1}}\dots e_{a'_{k}} & s\leq k_{1}\text{ and \ensuremath{s} is even}\\
0 & \text{otherwise}
\end{cases}
\]
modulo $2N_{k,d}+M_{k,d-1}$. Hence
\begin{align}
F_{2} & \equiv1\otimes e_{a'_{1}}\cdots e_{a'_{k}}+\delta_{b_{1},-z^{2}}\sum_{\substack{2\leq s\leq k_{1}\\
s:{\rm even}
}
}2e_{-1}e_{0}^{s-1}\otimes e_{a'_{s+1}}\dots e_{a'_{k}}\nonumber \\
 & \equiv1\otimes\word^{-}(\Bbbk)+2\delta_{b_{1},-z^{2}}\sum_{\substack{2\leq s\leq k_{1}\\
s:{\rm even}
}
}e_{-1}e_{0}^{s-1}\otimes e_{0}^{k_{1}-s}\word^{-}(k_{2},\dots,k_{d})\label{eq:main_e13}
\end{align}
modulo $2N_{k,d}+M_{k,d-1}$. Finally, putting (\ref{eq:main_phi_phibar}),
(\ref{eq:main_e8}), (\ref{eq:main_e12}) and (\ref{eq:main_e13})
together, we get
\begin{align*}
\varphi_{\otimes}(u) & \equiv1\otimes\biggl(2\sum_{\substack{1\leq i<j\leq d\\
b_{i}\neq b_{i+1}\\
k_{j}\in1+2\mathbb{Z}
}
}\word^{-}(\Bbbk^{(i,j)})+(1+2\sum_{\substack{1\leq j\leq d\\
b_{j}=-z^{2}\\
k_{j}\in2\mathbb{Z}
}
}1)\word^{-}(\Bbbk)\biggr)\\
 & \ \ \ +2\delta_{b_{1},-z^{2}}\sum_{\substack{2\leq s\leq k_{1}\\
s:{\rm even}
}
}e_{-1}e_{0}^{s-1}\otimes e_{0}^{k_{1}-s}\word^{-}(k_{2},\dots,k_{d})\pmod{2N_{k,d}+4M_{k,d}+M_{k,d-1}},
\end{align*}
and since the second term
\[
2\delta_{b_{1},-z^{2}}\sum_{\substack{2\leq s\leq k_{1}\\
s:{\rm even}
}
}e_{-1}e_{0}^{s-1}\otimes e_{0}^{k_{1}-s}\word^{-}(k_{2},\dots,k_{d})
\]
vanishes by the assumption $b_{1}=-1$ or $k_{1}=1$, this proves
the lemma.
\end{proof}
\begin{lem}
\label{lem:main_step2}We have
\[
Y(k,d)\subset X^{\mathrm{D}}(k,d)+2X(k,d)+\frac{1}{2}X(k,d-1).
\]
\end{lem}

\begin{proof}
It suffices to show
\begin{equation}
(\word^{-}(k_{1},\dots,k_{d})\bmod\iconf)\in X^{\mathrm{D}}(k,d)+2X(k,d)+\frac{1}{2}X(k,d-1)\label{eq:main_step2}
\end{equation}
for positive integers $k_{1},\dots,k_{d}$ such that $k_{1}+\cdots+k_{d}=k$
where $\word^{-}(k_{1},\dots,k_{d})$ is the notation used in Lemma
\ref{lem:main_step2_2}. Put $\Bbbk=(k_{1},\dots,k_{d})$. If $k_{j}$
is odd for all $2\leq j\leq d$, then $\word^{-}(\Bbbk)\in X^{\mathrm{D}}(k,d)$
by definition of $X^{\mathrm{D}}$. Therefore, we may assume that
$k_{j}$ is even for some $2\leq j\leq d$. Suppose that $j_{\max}$
is the maximal such $j$, so that $k_{j'}$ are odd for $j'>j_{\max}$,
and $i=i(k_{1},\dots,k_{d})$ is the minimal non-negative integer
such that $i<j_{\max}$ and $k_{i+1}=k_{i+2}=\cdots=k_{j_{\mathrm{max}}-1}=1$.
Then we put
\[
u\coloneqq e_{-1}e_{0}^{k_{1}-1}\cdots e_{-1}e_{0}^{k_{i}-1}e_{-z^{2}}e_{0}^{k_{i+1}-1}\cdots e_{-z^{2}}e_{0}^{k_{d}-1}.
\]
Notice that by the choice of $i$, $u$ satisfies the assumptions
of Lemma \ref{lem:main_step2_2}, i.e., $e_{-1}$ and $e_{-z^{2}}$
are not adjacent in $u$, and $u$ starts with $e_{-1}$ or else $k_{1}=1$
. Now, let us calculate the element $\restr u{z\to1}-\varphi(u)\in\iconf$.
By definition, 
\begin{equation}
\restr u{z\to1}=\word^{-}(\Bbbk).\label{eq:main_e14}
\end{equation}
For $\varphi(u)$, we have
\[
\varphi_{\otimes}(u)\equiv1\otimes\biggl(2\sum_{\substack{i<j\leq d\\
k_{j}\in1+2\mathbb{Z}
}
}\word^{-}(\Bbbk^{(i,j)})-\word^{-}(\Bbbk)\biggr)\pmod{2N_{k,d}+4M_{k,d}+M_{k,d-1}}
\]
by Lemma \ref{lem:main_step2_2} (we regard $\word^{-}(\Bbbk^{(i,j)})=0$
if $i=0$) and thus
\begin{equation}
\varphi(u)\equiv2\sum_{\substack{i<j\leq d\\
k_{j}\in1+2\mathbb{Z}
}
}\word^{-}(\Bbbk^{(i,j)})-\word^{-}(\Bbbk)\pmod{4\mathcal{A}_{k,d}^{0}(\{0,1,-1\})+\mathcal{A}_{k,d-1}^{0}(\{0,1,-1\})}.\label{eq:main_e15}
\end{equation}
By (\ref{eq:main_e14}) and (\ref{eq:main_e15}), we have
\[
\word^{-}(\Bbbk)\equiv2\sum_{\substack{i<j\leq d\\
k_{j}\in1+2\mathbb{Z}
}
}\word^{-}(\Bbbk^{(i,j)})-\word^{-}(\Bbbk)\pmod{\iconf+4\mathcal{A}_{k,d}^{0}(\{0,1,-1\})+\mathcal{A}_{k,d-1}^{0}(\{0,1,-1\})},
\]
by which we find that
\begin{equation}
\left(\word^{-}(\Bbbk)\bmod\iconf\right)-\sum_{\substack{i<j\leq d\\
k_{j}\in1+2\mathbb{Z}
}
}\left(\word^{-}(\Bbbk^{(i,j)})\bmod\iconf\right)\in2X(k,d)+\frac{1}{2}X(k,d-1).\label{eq:main_e16}
\end{equation}
Since $\Bbbk^{(i,j)}$ is always smaller than $\Bbbk$ in lexicographic
order, (\ref{eq:main_step2}) now follows by using (\ref{eq:main_e16})
repeatedly.
\end{proof}
\begin{proof}[Proof of Theorem \ref{thm:main_main}]
 By Lemmas \ref{lem:main_step1} and \ref{lem:main_step2}, we have
\[
X(k,d)\subset X^{\mathrm{D}}(k,d)+2X(k,d)+\frac{1}{2}X(k,d-1).
\]
Equivalently,
\[
\modx(k,d)\subset\modx^{\mathrm{D}}(k,d)+\modx(k,d-1)+2\modx(k,d).
\]
Therefore by Nakayama's Lemma, we have
\begin{equation}
\modx(k,d)\subset\modx^{\mathrm{D}}(k,d)+\modx(k,d-1).\label{eq:main_e17}
\end{equation}
By using (\ref{eq:main_e17}) repeatedly, we obtain
\[
\modx(k,d)\subset\modx^{\mathrm{D}}(k,d),
\]
which is the desired conclusion.
\end{proof}

\subsection{A minimal family of relations and explicit expansion of an Euler
sum by the basis}

By the proof of Theorem \ref{thm:main_main}, we can obtain the explicit
version of Corollary \ref{thm:main_general_zeta}. Put $\indset^{Y}(k,d)=\{(k_{1},\dots,k_{r})\in\indset(k,d)\mid k_{1},\dots,k_{r-1}>0,k_{r}<0\}$.
Fix $k,d\geq0$. Define the total order structure $\prec$ of $\indset(k,d)$
by
\begin{enumerate}
\item If $\Bbbk\in\indset^{\mathrm{D}}(k,d)$ and $\Bbbk'\notin\indset^{\mathrm{D}}(k,d)$,
then $\Bbbk\prec\Bbbk'$.
\item If $\Bbbk,\Bbbk'\notin\indset^{\mathrm{D}}(k,d)$ and $\depth(\Bbbk)\prec\depth(\Bbbk')$,
then $\Bbbk\prec\Bbbk'$.
\item If $\Bbbk,\Bbbk'\notin\indset^{\mathrm{D}}(k,d)$, $\depth(\Bbbk)=\depth(\Bbbk')$,
$\Bbbk\in\indset^{Y}(k,d)$, and $\Bbbk'\notin\indset^{Y}(k,d)$ then
$\Bbbk\prec\Bbbk'$.
\item For the pair of indices $\Bbbk,\Bbbk'$ whose order can not be determined
by (1), (2) and (3) above, we define their order according to their
lexicographical order.
\end{enumerate}
For each $\Bbbk=(k_{1},\dots,k_{r})\in\indset(k,d)\setminus\indset^{\mathrm{D}}(k,d)$,
define $f(\Bbbk)\in\mathcal{B}_{k,d}$ by
\[
f(\Bbbk)\coloneqq\begin{cases}
\restr{\modword(\Bbbk)}{1\to z} & \Bbbk\in\indset(k,d)\setminus\indset^{Y}(k,d)\\
2^{d-1}e_{-1}e_{0}^{\left|k_{1}\right|-1}\cdots e_{-1}e_{0}^{\left|k_{i}\right|-1}e_{-z^{2}}e_{0}^{\left|k_{i+1}\right|-1}\cdots e_{-z^{2}}e_{0}^{\left|k_{d}\right|-1} & \Bbbk\in\indset^{Y}(k,d)\setminus\indset^{\mathrm{D}}(k,d)
\end{cases}
\]
where the map $u\mapsto\restr u{1\to z}$ is the ring homomorphism
defined by $\restr{e_{0}}{1\to z}\coloneqq e_{0}$, $\restr{e_{-1}}{1\to z}\coloneqq e_{-1}$,
$\restr{e_{1}}{1\to z}\coloneqq e_{z}$ and $i=i(\Bbbk)$ is as in
the proof of Lemma \ref{lem:main_step2}. For $\Bbbk\in\indset(k,d)\setminus\indset^{\mathrm{D}}(k,d)$
and $\Bbbk'\in\indset(k,d)$, define $c_{\Bbbk,\Bbbk'}\in\mathbb{Z}$
as the coefficient of $\modword(\Bbbk')$ in $\restr{f(\Bbbk)}{z\to1}-\varphi(f(\Bbbk))\in\iconf$,
i.e., 
\[
\restr{f(\Bbbk)}{z\to1}-\varphi(f(\Bbbk))=\sum_{\Bbbk'\in\indset(k,d)}c_{\Bbbk,\Bbbk'}\modword(\Bbbk').
\]
Then we have $c_{\Bbbk,\Bbbk}\equiv1\pmod{2}$ and $c_{\Bbbk,\Bbbk'}\equiv0\pmod{2}$
if $\Bbbk<\Bbbk'$ (see the explicit reduction process in the proof
of Lemmas \ref{lem:main_step1} and \ref{lem:main_step2}). Therefore
the matrix 
\[
C\coloneqq(c_{\Bbbk,\Bbbk'})_{\Bbbk,\Bbbk'\in\indset(k,d)\setminus\indset^{\mathrm{D}}(k,d)}
\]
is congruent to a lower unitriangular matrix modulo $2$, and hence
invertible. Put 
\[
C'\coloneqq(c_{\Bbbk,\Bbbk'})_{\Bbbk\in\indset(k,d)\setminus\indset^{\mathrm{D}}(k,d),\Bbbk'\in\indset^{\mathrm{D}}(k,d)}.
\]
Then we get the following theorem.
\begin{thm}
\label{thm:main_explicit}Let $k,d\geq0$, $R$ a $\mathbb{Q}$-vector
space, and $Z:\mathcal{A}_{k,d}^{0}(\{0,1,-1\})\to R$ a linear map
such that $Z(\iconf)=\{0\}$. Let $\bm{\alpha}=(\alpha_{\Bbbk,\Bbbk'})_{\Bbbk\in\indset(k,d)\setminus\indset^{\mathrm{D}}(k,d),\Bbbk'\in\indset^{\mathrm{D}}(k,d)}$
be a matrix defined by 
\[
\bm{\alpha}=-C^{-1}C'.
\]
Then, for $\Bbbk\in\indset(k,d)\setminus\indset^{\mathrm{D}}(k,d)$,
we have
\[
Z(\modword(\Bbbk))=\sum_{\Bbbk'\in\indset^{\mathrm{D}}(k,d)}\alpha_{\Bbbk,\Bbbk'}Z(\modword(\Bbbk')).
\]
In particular,
\[
\modzeta(\Bbbk)=\sum_{\Bbbk'\in\indset^{\mathrm{D}}(k,d)}\alpha_{\Bbbk,\Bbbk'}\modzeta(\Bbbk').
\]
\end{thm}

\begin{rem}
Furthermore, by the motivicity of the confluence relations proved
in later sections and by the dimensional bound given by the motivic
theory, we can show that
\[
\bigcup_{k=0}^{\infty}\{\restr{f(\Bbbk)}{z\to1}-\varphi(f(\Bbbk))\mid\Bbbk\in\indset(k,\infty)\setminus\indset^{\mathrm{D}}(k,\infty)\}
\]
forms a $\mathbb{Q}$-basis of $\iconf$. Nevertheless, the authors
are not sure if this fact may be derived without the motivic theory
at the moment.
\end{rem}

\begin{example}
Let us consider the case $(k,d)=(3,3)$. To save space, we write $\bar{n}$
for $-n$ in an index $(k_{1},k_{2},...,k_{r})$ of $\indset(k,d)$,
e.g. $(-1,2,-3)$ is denoted as $(\bar{1},2,\bar{3})$. Then
\[
\indset^{\mathrm{D}}(k,d)=\{(\bar{3})\prec(1,1,\bar{1})\prec(2,\bar{1})\},
\]
\[
\indset(k,d)\setminus\indset^{\mathrm{D}}(k,d)=\{(3)\prec(1,\bar{2})\prec(\bar{2},\bar{1})\prec(\bar{1},\bar{2})\prec(\bar{1},2)\prec(1,2)\prec(\bar{1},\bar{1},\bar{1})\prec(\bar{1},1,\bar{1})\prec(1,\bar{1},\bar{1})\}
\]
 and the matrices $C$, $C'$ are given by
\[
C=\left(\begin{array}{rrr}
-4 & 0 & 0\\
5 & 0 & 0\\
-2 & 0 & -2\\
4 & 0 & 1\\
-2 & 0 & -1\\
0 & 0 & 0\\
0 & -1 & 2\\
0 & -1 & 0\\
-4 & -1 & 2
\end{array}\right),\ C'=\left(\begin{array}{rrrrrrrrr}
-3 & 0 & 0 & 0 & 0 & 0 & 0 & 0 & 0\\
4 & 1 & 0 & 0 & 0 & 0 & 0 & 0 & 0\\
0 & -3 & -1 & -2 & -2 & 0 & 0 & 0 & 0\\
0 & 3 & 2 & 3 & 2 & 0 & 0 & 0 & 0\\
0 & -4 & -2 & -2 & -1 & 0 & 0 & 0 & 0\\
0 & -2 & 0 & -2 & -2 & -1 & 0 & 0 & 0\\
0 & 0 & 0 & 0 & 0 & 0 & 1 & 0 & 0\\
0 & 2 & 0 & 0 & 0 & 0 & 0 & 1 & 0\\
0 & 2 & 0 & 0 & 0 & 0 & 0 & 0 & 1
\end{array}\right).
\]
The confluence relations of $f(\Bbbk)$'s say
\[
C\left(\begin{array}{c}
\modzeta(\bar{3})\\
\modzeta(1,1,\bar{1})\\
\modzeta(2,\bar{1})
\end{array}\right)+C'\left(\begin{array}{c}
\modzeta(3)\\
\modzeta(1,\bar{2})\\
\modzeta(\bar{2},\bar{1})\\
\modzeta(\bar{1},\bar{2})\\
\modzeta(\bar{1},2)\\
\modzeta(1,2)\\
\modzeta(\bar{1},\bar{1},\bar{1})\\
\modzeta(\bar{1},1,\bar{1})\\
\modzeta(1,\bar{1},\bar{1})
\end{array}\right)=0.
\]
Notice that the matrix $C'$ is indeed congruent to a lower unitriangular
matrix modulo $2$ and hence invertible. More explicitly, $\det(C')=-3$
and
\[
(C')^{-1}=\left(\begin{array}{rrrrrrrrr}
-\frac{1}{3} & 0 & 0 & 0 & 0 & 0 & 0 & 0 & 0\\
\frac{4}{3} & 1 & 0 & 0 & 0 & 0 & 0 & 0 & 0\\
-\frac{20}{3} & -5 & -1 & -2 & -2 & 0 & 0 & 0 & 0\\
\frac{20}{3} & 5 & 2 & 3 & 2 & 0 & 0 & 0 & 0\\
-\frac{16}{3} & -4 & -2 & -2 & -1 & 0 & 0 & 0 & 0\\
-\frac{16}{3} & -4 & 0 & -2 & -2 & -1 & 0 & 0 & 0\\
0 & 0 & 0 & 0 & 0 & 0 & 1 & 0 & 0\\
-\frac{8}{3} & -2 & 0 & 0 & 0 & 0 & 0 & 1 & 0\\
-\frac{8}{3} & -2 & 0 & 0 & 0 & 0 & 0 & 0 & 1
\end{array}\right).
\]
Thus
\[
\left(\begin{array}{c}
\modzeta(3)\\
\modzeta(1,\bar{2})\\
\modzeta(\bar{2},\bar{1})\\
\modzeta(\bar{1},\bar{2})\\
\modzeta(\bar{1},2)\\
\modzeta(1,2)\\
\modzeta(\bar{1},\bar{1},\bar{1})\\
\modzeta(\bar{1},1,\bar{1})\\
\modzeta(1,\bar{1},\bar{1})
\end{array}\right)=-(C')^{-1}C\left(\begin{array}{c}
\modzeta(\bar{3})\\
\modzeta(1,1,\bar{1})\\
\modzeta(2,\bar{1})
\end{array}\right)=\left(\begin{array}{rrr}
-\frac{4}{3} & 0 & 0\\
\frac{1}{3} & 0 & 0\\
\frac{1}{3} & 0 & -2\\
-\frac{7}{3} & 0 & 3\\
\frac{2}{3} & 0 & -3\\
\frac{8}{3} & 0 & 0\\
0 & 1 & -2\\
-\frac{2}{3} & 1 & 0\\
\frac{10}{3} & 1 & -2
\end{array}\right)\left(\begin{array}{c}
\modzeta(\bar{3})\\
\modzeta(1,1,\bar{1})\\
\modzeta(2,\bar{1})
\end{array}\right).
\]

The determinant of the $C'$ clearly gives an upper bound of the LCM
of the denominators of the coefficients $\alpha_{\Bbbk,\Bbbk'}$,
and they actually coincide in the example above. However, it is not
the case in general. For example, if $(k,d)=(7,3)$ then the LCM is
$3^{3}\cdot5\cdot7^{2}\cdot17\cdot31\cdot61$ whereas the determinant
of $C'$ is $-3^{30}\cdot5^{18}\cdot7^{5}\cdot17\cdot31^{6}\cdot41\cdot61\cdot107\cdot176779$.
It would be interesting to seek for a better upper bound of the denominator
of the coefficients $\alpha_{\Bbbk,\Bbbk'}$.
\end{example}

\pagebreak{}

\part{\label{part:Motivicity}Motivicity of confluence relations}

In this part, we will prove the motivicity of the confluence relations
i.e., $\iconf\subset\ker L^{\mathfrak{m}}$. Basically, there are
two possible approaches to this problem.

The first approach is the following. The confluence relations for
(real-valued) Euler sums are proved by using the differential formula
for iterated integrals. These differential formula can be proved by
the combination of simple properties, and so the confluence relations
can be proved by a combination of the rules in \cite{Kontsevich-Zagier};
additivity, change of variables, Newton-Leibniz formula. Thus one
can possibly prove the lift of confluence relations to the identity
between the periods of Nori motives. However, in the proof of confluence
relations, very complicated integrals of the rational functions whose
poles lie in the boundaries of the domains of integration appear,
so it does not fit to the definition of formal periods which requires
the regularity of the differential forms. One might be able to avoid
such ``singular'' integrals by blowing-up the singular locus, however,
the authors are not successful in this approach so far.

The second approach is to use Brown's lemma. We employ this approach
in this article.

First, we generalize the confluence relations to a more general setting.
Recall that the confluence relations are relations by considering
the limit of the expansion of iterated integrals $I(p_{0}(z);p_{1}(z),\dots,p_{k}(z);p_{k+1}(z))$
where $p_{1}(z),\dots,p_{k}(z)\in\{0,-1,z,-z^{2}\}$, $p_{0}(z)=0$,
$p_{k+1}(z)=z$. In this section, we will treat a more general case
where $p_{0}(z),\dots,p_{k+1}(z)\in\mathbb{Q}(z)$ with some additional
conditions. To define the confluence relation, we need a map which
we call the evaluation map. The proof proceeds by the following steps:
\begin{enumerate}
\item Define the general evaluation map.
\begin{enumerate}
\item Introduce the notions of motivic iterated integrals with tangential
vectors containing an indeterminate, and prove their fundamental properties.
\item Calculate limits of (complex-valued) iterated integrals for simple
case.
\item Construct the ``evaluation map'' which is a motivic lift of the limiting
map in (b)
\item Extend the definition of the evaluation map to general case.
\end{enumerate}
\item Prove fundamental properties of the general evaluation map.
\item Calculate the motivic coproduct of confluence relations combinatorially
with the help of the fundamental properties of general evaluation
map, by which we prove the motivicity of confluence relations.
\end{enumerate}
Except for the very last section, the contents of Part \ref{part:Motivicity}
is completely independent of Part \ref{part:Confluence_Euler_sum}
where the theorems are stated in a very general framework. For this
reason, we employ slightly different setup from that in Part \ref{part:Confluence_Euler_sum},
which is closely related to Goncharov's setup. Also, the notations
used in Part \ref{part:Motivicity} except in Section \ref{sec:Proof-of-Theorem}
are completely unrelated from those in Part \ref{part:Confluence_Euler_sum}
and we do not care about possible overlaps. 

\section{\label{sec:Part2_preliminaries}Preliminaries}

\subsection{Tangential base points}

Let $F=\mathbb{Q}$ or $\mathbb{C}$, and $t$ be a standard coordinate
of $\mathbb{P}^{1}(F)$. For $x\in\mathbb{P}^{1}(F)$ and $v\in\mathbb{Q}^{\times}$,
we denote by $\overrightarrow{v}_{x}$ the tangential base point at
$x$ with the tangential vector 
\[
\overrightarrow{v}\coloneqq\begin{cases}
v\frac{d}{dt} & x\neq\infty\\
v\frac{d}{d(\frac{1}{t})} & x=\infty.
\end{cases}
\]
For $X\subset\mathbb{P}^{1}(\mathbb{Q})$, let
\[
T^{\times}X\coloneqq\{\overrightarrow{v}_{a}\mid v\in\mathbb{Q}^{\times},a\in X\},
\]
be the set of tangential base points. For tangential base point $x=\overrightarrow{v}_{a}$,
we put $\bar{x}=a$ and ${\rm vec}(x)=v$.

\subsection{Motivic iterated integrals and Goncharov's coaction formula.}

Let $\mathcal{H}:=\bigoplus_{k\geq0}\mathcal{H}_{k}$ be the graded
ring of effective motivic periods of mixed Tate motives over $\mathbb{Q}$.
For $k\geq0$, $a_{0},a_{k+1}\in T^{\times}\mathbb{P}^{1}(\mathbb{Q})$,
$a_{1},\ldots,a_{k}\in\mathbb{P}^{1}(\mathbb{Q})$, and a path $\gamma$
from $a_{0}$ to $a_{k+1}$ on $\mathbb{P}^{1}(\mathbb{Q})\setminus\{\infty,a_{1},\dots,a_{k}\}$,
we denote by
\[
I_{\gamma}^{\mathfrak{m}}(a_{0};a_{1},\dots,a_{k};a_{k+1})\in\mathcal{H}_{k}
\]
the associated \emph{motivic iterated integral}. Here we understand
$I_{\gamma}^{\mathfrak{m}}(a_{0};a_{1},\dots,a_{k};a_{k+1})=0$ if
one of $a_{1},\dots,a_{k}$ is $\infty$. We define $\mathfrak{A}=\bigoplus_{k=0}^{\infty}\mathfrak{A}_{k}$
by $\mathfrak{A}\coloneqq\mathcal{H}/\mu\mathcal{H}$ where $\mu\in\mathcal{H}_{1}$
is the motivic $2\pi\sqrt{-1}$., i.e., $\mu=I_{\gamma}^{\mathfrak{m}}(1;0;1)$
where $\gamma:[0,1]\to\mathbb{C}^{\times};t\mapsto\exp(2\pi it)$
is a unit circle. Furthermore, define $\mathfrak{L}=\bigoplus_{k=1}^{\infty}\mathfrak{L}_{k}$
by $\mathfrak{L}\coloneqq\mathfrak{A}_{>0}/\left(\mathfrak{A}_{>0}\cdot\mathfrak{A}_{>0}\right)$.
We denote by
\[
I^{\mathfrak{a}}(a_{0};a_{1},\dots,a_{k};a_{k+1})\in\mathfrak{A}_{k}
\]
and
\[
I^{\mathfrak{l}}(a_{0};a_{1},\dots,a_{k};a_{k+1})\in\mathfrak{L}_{k}
\]
the images of $I_{\gamma}^{\mathfrak{m}}(a_{0};a_{1},\dots,a_{k};a_{k+1})$
in $\mathfrak{A}_{k}$ and $\mathfrak{L}_{k}$ respectively, which
are independent of the choice of $\gamma$. It is known that $\mathfrak{A}$
is equipped with the graded Hopf algebra structure, whose coproduct
is denoted by
\[
\Delta:\mathfrak{A}\to\mathfrak{A}\otimes\mathfrak{A},
\]
and $\mathcal{H}$ becomes a graded $\mathfrak{A}$-comodule, whose
coaction is also denoted by
\[
\Delta:\mathcal{H}\to\mathfrak{A}\otimes\mathcal{H}.
\]
If the tangential base points appear in the central argument ($=$
not as the endpoints) of iterated integral symbol, we interpret them
by the obvious way, e.g., for $a_{0},\dots,a_{k+1}\in T^{\times}\mathbb{P}^{1}(\mathbb{Q})$,
the symbol $I(a_{0};a_{1},\dots,a_{k};a_{k+1})$ means $I(a_{0};\overline{a_{1}},\dots,\overline{a_{k}};a_{k+1})$.
We give a coproduct formula for the motivic iterated integral which
slightly generalizes\footnote{In our formula, $\infty$ is allowed to be the endpoints of the integral.
The authors have not checked in detail, but probably our proof of
Proposition \ref{prop:coproduct} is essentially just a trace of known
proof for the case of non-infinite endpoints.} a formula given by Goncharov \cite{GonSym} and Brown \cite{Bro_mix}. 
\begin{prop}
\label{prop:coproduct}Let $k\geq0$, $a_{0},a_{k+1}\in T^{\times}\mathbb{P}^{1}(\mathbb{Q})$,
$a_{1},\dots,a_{k}\in T^{\times}\mathbb{Q}$, and $\gamma$ a path
from $a_{0}$ to $a_{k+1}$. If $\#\{a_{1},\dots,a_{k}\}=\#\{\overline{a_{1}},\dots,\overline{a_{k}}\}$
(i.e., the same point has the same tangential vector) then we have
\begin{multline*}
\Delta I_{\gamma}^{\mathfrak{m}}(a_{0};a_{1},\dots,a_{k};a_{k+1})\\
\coloneqq\sum_{r=0}^{k}\sum_{0=i_{0}<i_{1}<\cdots<i_{r}<i_{r+1}=k+1}\prod_{j=0}^{r}I^{\mathfrak{a}}(a_{i_{j}};a_{i_{j}+1},\dots,a_{i_{j+1}-1};a_{i_{j+1}})\otimes I_{\gamma}^{\mathfrak{m}}(a_{i_{0}};a_{i_{1}},\dots,a_{i_{r}};a_{i_{r+1}}).
\end{multline*}
\end{prop}

\begin{proof}
For any $x,y\in T^{\times}\mathbb{P}^{1}(\mathbb{Q})$, we denote
by $\pio xy$ the de Rham fundamental torsor of paths on $X=\mathbb{P}^{1}(\mathbb{Q})\setminus\{\infty,\overline{a_{1}},\dots,\overline{a_{k}}\}$
from $x$ to $y$. Then the affine ring $\mathcal{O}(\pio xy)$ is
identified with the free non-commutative ring $\mathbb{Q}\left\langle e_{s}\mid s\in\{a_{1},\dots,a_{k}\}\right\rangle $
equipped with the shuffle product. Here each $e_{s}$ corresponds
to the $1$-form $\frac{dt}{t-\overline{s}}$. Let $\mathcal{U}$
be the prounipotent part of the motivic Galois group of mixed Tate
motives over $\mathbb{Q}$. Then $\mathfrak{A}$ can be identified
with $\mathcal{O}(\mathcal{U})$. For each $x,y\in T^{\times}\mathbb{P}^{1}(\mathbb{Q})$
and a path $\gamma$ from $x$ to $y$ on $X$, we denote by $L_{\gamma}:\mathcal{O}(\pio xy)\to\mathcal{H}$
the linear map defined by $L_{\gamma}(e_{b_{1}}\cdots e_{b_{n}})=I_{\gamma}^{\mathfrak{m}}(x;b_{1},\dots,b_{n};y)$.
The action of the motivic Galois group
\[
\mathcal{U}\times\pio xy\to\pio xy
\]
gives rise to the coaction
\[
\delta_{x,y}:\mathcal{O}(\pio xy)\to\mathfrak{A}\otimes\mathcal{O}(\pio xy).
\]
Let us recall the following properties of $\delta_{x,y}$:
\begin{enumerate}
\item For any $x,y\in T^{\times}\mathbb{P}^{1}(\mathbb{Q})$ and a path
$\gamma$ from $x$ to $y$ on $X$, the following diagram is commutative.
\[
\xymatrix{\mathcal{O}(\pio xy)\ar[r]\sp(0.45){\delta_{x,y}}\ar^{L_{\gamma}}[d] & \mathfrak{A}\otimes\mathcal{O}(\pio xy)\ar^{{\rm id}\otimes L_{\gamma}}[d]\\
\mathcal{H}\ar^{\Delta}[r] & \mathfrak{A}\otimes\mathcal{H}
}
\]
\item The action of $\mathcal{U}$ preserves the groupoid structure of $(\pio xy)_{x,y\in T^{\times}\mathbb{P}^{1}(\mathbb{Q})}$,
i.e., for any $x,y,z\in T^{\times}\mathbb{P}^{1}(\mathbb{Q})$, the
following diagram commutes.
\[
\xymatrix{\mathcal{O}(\pio xz)\ar^{\delta_{x,y}}[rr]\ar^{{\rm dec}}[d] &  & \mathfrak{A}\otimes\mathcal{O}(\pio xz)\ar^{\mathrm{id}\otimes{\rm dec}}[rd]\\
\mathcal{O}(\pio xy)\otimes\mathcal{O}(\pio yz)\ar[rr]\sb(0.4){\delta_{x,y}\otimes\delta_{y,z}} &  & \mathfrak{A}\otimes\mathcal{O}(\pio xy)\otimes\mathfrak{A}\otimes\mathcal{O}(\pio yz)\ar[r]_{\lambda} & \mathfrak{A}\otimes\mathcal{O}(\pio xy)\otimes\mathcal{O}(\pio yz)
}
\]
Here, ${\rm dec}$ is the deconcatenation map and $\lambda$ is the
linear map given by $\lambda(a\otimes b\otimes c\otimes d)=ac\otimes b\otimes d$
\item Let ${\rm const}:\mathcal{O}(\pio xy)=\mathbb{Q}\left\langle e_{s}\mid s\in\{a_{1},\dots,a_{k}\}\right\rangle \to\mathbb{Q}$
be the map which sends $u$ to the constant term of $u$. Then
\[
({\rm id}\otimes{\rm const})\circ\delta_{x,y}(e_{s_{1}}\cdots e_{s_{n}})=I^{\mathfrak{a}}(x;s_{1},\dots,s_{n};y)
\]
for any $e_{s_{1}}\cdots e_{s_{n}}\in\mathcal{O}(\pio xy)$.
\item Let $x\in\{a_{1},\dots,a_{k}\}$ and $p:\mathcal{O}(\pio xx)\to\mathbb{Q}\left\langle e_{x}\right\rangle $
be the projection map which sends $e_{x}$ to $e_{x}$ and $e_{s}$
to $0$ for $s\neq x$. Then the following diagram commutes.
\[
\xymatrix{\mathcal{O}(\pio xx)\ar[r]\sp(0.45){\delta_{x,x}}\ar[d]_{p} & \mathfrak{A}\otimes\mathcal{O}(\pio xx)\ar[d]^{{\rm id}\otimes p}\\
\mathbb{Q}\left\langle e_{x}\right\rangle \ar[r]\sb(0.4){v\mapsto1\otimes v} & \mathfrak{A}\otimes\mathbb{Q}\left\langle e_{x}\right\rangle 
}
\]
\end{enumerate}
Put $w=e_{a_{1}}\cdots e_{a_{k}}$. Then by (1), we have
\[
\Delta I_{\gamma}^{\mathfrak{m}}(a_{0};a_{1},\dots,a_{k};a_{k+1})=\Delta L_{\gamma}(w)=\sum_{\substack{u\in\mathcal{O}(\pio{a_{0}}{a_{k+1}})\\
u:{\rm word}
}
}c(u)\otimes L_{\gamma}(u)
\]
where $c(u)$ is given by
\[
\delta_{a_{0},a_{k+1}}(w)=\sum_{\substack{u\in\mathcal{O}(\pio{a_{0}}{a_{k+1}})\\
u:{\rm word}
}
}c(u)\otimes u.
\]
Let us calculate $c(u)$. Put $u=e_{b_{1}}\cdots e_{b_{n}}$. By using
(2) repeatedly, we have the commutative diagram:
\[
\xymatrix{\mathcal{O}(\pio{a_{0}}{a_{k+1}})\ar[r]\sp(0.45){\delta_{a_{0},a_{k+1}}}\ar^{{\rm dec}}[d] & \mathfrak{A}\otimes\mathcal{O}(\pio{a_{0}}{a_{k+1}})\ar^{{\rm id}\otimes{\rm dec}}[d]\\
R\ar^{\delta'}[r] & \mathfrak{A}\otimes R
}
\]
where 
\[
R=\bigotimes_{j=0}^{2n}\mathcal{O}(\pio{x_{j}}{x_{j+1}}),
\]
$(x_{0},\dots,x_{2n+1})=(a_{0},\{b_{1}\}^{2},\{b_{2}\}^{2},\cdots,\{b_{n}\}^{2},a_{k+1})$,
${\rm dec}$ is the deconcatenation coproduct, and $\delta'$ is the
map induced from $\delta_{x_{0},x_{1}},\dots,\delta_{x_{2n},x_{2n+1}}$.
Furthermore, let $f:\mathcal{O}(\pio{a_{0}}{a_{k+1}})\to\mathbb{Q}$
be the map which takes the coefficient of $u$ of the element of $\mathcal{O}(\pio{a_{0}}{a_{k+1}})$,
and $f':R\to\mathbb{Q}$ be the map which takes the coefficient of
$1\otimes e_{b_{1}}\otimes\cdots\otimes1\otimes e_{b_{n}}\otimes1$
of the element of $R$. Then we have $f=f'\circ{\rm dec}$. Thus we
get the following commutative diagram.
\[
\xymatrix{\mathcal{O}(\pio{a_{0}}{a_{k+1}})\ar[r]\sp(0.45){\delta_{a_{0},a_{k+1}}}\ar^{{\rm dec}}[d] & \mathfrak{A}\otimes\mathcal{O}(\pio{a_{0}}{a_{k+1}})\ar_{{\rm id}\otimes{\rm dec}}[d]\ar[r]\sp(0.7){{\rm id}\otimes f} & \mathfrak{A}\\
R\ar^{\delta'}[r] & \mathfrak{A}\otimes R\ar_{{\rm id}\otimes f'}[ru]
}
\]
Thus we have
\begin{align*}
c(u) & =({\rm id}\otimes f)\circ\delta_{a_{0},a_{k+1}}(e_{a_{1}}\cdots e_{a_{k}})\\
 & =(g_{0}\otimes g_{1}\otimes\cdots\otimes g_{2n})\circ{\rm dec}(e_{a_{1}}\cdots e_{a_{k}})
\end{align*}
where $g_{j}:\mathcal{O}(\pio{x_{j}}{x_{j+1}})\to\mathfrak{A}$ is
the composite map
\[
\mathcal{O}(\pio{x_{j}}{x_{j+1}})\xrightarrow{\delta_{x_{j},x_{j+1}}}\mathfrak{A}\otimes\mathcal{O}(\pio{x_{j}}{x_{j+1}})\xrightarrow{{\rm id}\otimes f_{j}}\mathfrak{A}
\]
where $f_{j}$ is the map which takes the coefficient of $1$ (resp.
$e_{b_{(j+1)/2}}$) of the element of $\mathcal{O}(\pio{x_{j}}{x_{j+1}})$
for an even (resp. odd) $j$. Then we have
\[
g_{2j}(e_{s_{1}}\cdots e_{s_{m}})=I^{\mathfrak{a}}(x_{2j};s_{1},\dots,s_{m};x_{2j+1})
\]
by (3) and 
\[
g_{2j-1}(e_{s_{1}}\cdots e_{s_{m}})=\begin{cases}
1 & e_{s_{1}}\cdots e_{s_{m}}=e_{b_{j}}\\
0 & e_{s_{1}}\cdots e_{s_{m}}\neq e_{b_{j}}
\end{cases}
\]
by (4). Thus we have
\[
c(u)=\sum_{\substack{0=i_{0}<i_{1}<\cdots<i_{n}<i_{n+1}=k+1\\
a_{i_{1}}=b_{1},\dots,a_{i_{n}}=b_{n}
}
}\prod_{j=0}^{n}I^{\mathfrak{a}}(a_{i_{j}};a_{i_{j}+1},\dots,a_{i_{j+1}-1};a_{i_{j+1}}).
\]
This completes the proof of the proposition.
\end{proof}
For $r>0$, Brown's infinitesimal coaction 
\[
D_{r}:\mathcal{H}\to\mathfrak{L}_{r}\otimes\mathcal{H}
\]
is defined as the composite map
\[
\mathcal{H}\xrightarrow{\Delta}\mathfrak{A}\otimes\mathcal{H}\to\mathfrak{A}_{r}\otimes\mathcal{H}\to\mathfrak{L}_{r}\otimes\mathcal{H}.
\]
We put $D=\bigoplus_{r=1}^{\infty}D_{r}$. Let $\compmap:\mathcal{H}\longrightarrow\mathbb{C}$
be the period map. The following criterion is well-known (see \cite[Lemma 2.7]{Bro_mix}).
\begin{lem}
\label{lem:Brown}Let $u\in\mathcal{H}_{k}$. If $\compmap(u)=0$
and $D_{r}(u)=0$ for all $r<k$ then $u=0$.
\end{lem}

As a corollary of Proposition \ref{prop:coproduct}, we obtain the
following explicit formula for $D_{r}$ of motivic iterated integrals.
\begin{cor}
Let $k\geq0$, $a_{0},a_{k+1}\in T^{\times}\mathbb{P}^{1}(\mathbb{Q})$,
$a_{1},\dots,a_{k}\in T^{\times}\mathbb{Q}$, and $\gamma$ a path
from $a_{0}$ to $a_{k+1}$. If $\#\{a_{1},\dots,a_{k}\}=\#\{\overline{a_{1}},\dots,\overline{a_{k}}\}$
then we have
\[
D_{r}(I_{\gamma}^{\mathfrak{m}}(a_{0};a_{1},\dots,a_{k};a_{k+1}))=\sum_{i=0}^{k-r}I^{\mathfrak{l}}(a_{i};a_{i+1},\dots,a_{i+r};a_{i+r+1})\otimes I_{\gamma}^{\mathfrak{m}}(a_{0};a_{1},\dots,a_{i},a_{i+r+1},\dots,a_{k};a_{k+1}).
\]
\end{cor}

\subsection{Motivic iterated integrals with extended tangential base points }

For a subset $X\subset\mathbb{P}^{1}(\mathbb{C})$ of $\mathbb{P}^{1}(\mathbb{C})$
such that $\#(\mathbb{P}^{1}(\mathbb{C})\setminus X)<\infty$, $a,a'\in\mathbb{P}^{1}(\mathbb{C})$,
$v,v'\in\mathbb{R}^{\times}$ and $c,c'\in\mathbb{R}_{>0}$, there
exists a canonical one-to-one correspondence
\[
\Phi:\pi_{1}(X,\overrightarrow{v}_{a},\overrightarrow{v'}_{a'})\simeq\pi_{1}(X,\overrightarrow{cv}_{a},\overrightarrow{c'v'}_{a'}).
\]
If there is no risk of confusion, we identify $\pi_{1}(X,\overrightarrow{v}_{a},\overrightarrow{v'}_{a'})$
with $\pi_{1}(X,\overrightarrow{cv}_{a},\overrightarrow{c'v'}_{a'})$,
and we simply write $\Phi(\gamma)$ as $\gamma$. Let $T$ be an indeterminate
and $\widehat{\mathbb{Q}^{\times}}\coloneqq\{\left.v\exp(mT)\right|v\in\mathbb{Q}^{\times},m\in\mathbb{Z}\}$.
For $X\subset\mathbb{P}^{1}(\mathbb{C})$, we define the set of \emph{extended
tangential base} points by
\[
\widehat{T^{\times}}X\coloneqq\{(x,v)\mid v\in\widehat{\mathbb{Q}^{\times}},x\in X\}
\]
and denote its element $(x,v)$ as $\overrightarrow{v}_{x}$ just
like a usual tangential base point. For $s\in\mathbb{Q}_{>0}$, we
define
\[
\iota_{s}:\widehat{T^{\times}}X\to T^{\times}X
\]
by 
\[
\iota_{s}(\overrightarrow{v\exp(mT)}_{a})=\overrightarrow{vs^{m}}_{a}.
\]
We define a path from $x\in\widehat{T^{\times}}\mathbb{P}^{1}(\mathbb{C})$
to $y\in\widehat{T^{\times}}\mathbb{P}^{1}(\mathbb{C})$ to be the
path from $\iota_{1}(x)$ to $\iota_{1}(y)$. Note that this is equivalent
to the path from $\iota_{s}(x)$ to $\iota_{s'}(y)$ with arbitrary
choice of $s,s'\in\mathbb{Q}_{>0}$ by the aforementioned identification
$\Phi$ of the fundamental torsors. We put $\log^{\mathfrak{m}}(x)\coloneqq I_{{\rm dch}}^{\mathfrak{m}}(1;0;x)$
for $x\in\mathbb{Q}_{>0}$ where the ${\rm dch}$ is the straight
path,
\begin{defn}
For $k\geq0$ and $a_{0},a_{k+1}\in\widehat{T^{\times}}\mathbb{P}^{1}(\mathbb{Q})$,
$a_{1},\dots,a_{k}\in\mathbb{Q}$, and a path $\gamma$ from $a_{0}$
to $a_{k+1}$ on $\mathbb{P}^{1}(\mathbb{Q})\setminus\{\infty,a_{1},\dots,a_{k}\}$,
we denote by
\[
I_{\gamma}^{\mathfrak{m}}(a_{0};a_{1},\dots,a_{k};a_{k+1})
\]
the unique polynomial $p(T)\in\mathcal{H}[T]$ such that
\[
p(\log^{\mathfrak{m}}(s))=I_{\gamma}^{\mathfrak{m}}(\iota_{s}(a_{0});a_{1},\dots,a_{k};\iota_{s}(a_{k+1}))
\]
for $s\in\mathbb{Q}_{>0}$.
\end{defn}

The definition above requires the existence as well as the uniqueness
of $p(T)$, which can be proved as follows.
\begin{proof}[Proof of the existence and the uniqueness of $p(T)$]
The uniqueness is obvious since $\mathcal{H}$ is an integral domain.
So let us prove the existence by an explicit construction. We put
$\overrightarrow{v\exp(mT)}_{x}=a_{0}$ and $\overrightarrow{v'\exp(m'T)}_{y}=a_{k+1}$.
Then $\iota_{s}(a_{0})=\overrightarrow{vs^{m}}_{x}$ and $\iota_{s}(a_{k+1})=\overrightarrow{v's^{m'}}_{y}$.
The path composition formula gives
\[
I_{\gamma}^{\mathfrak{m}}(\overrightarrow{vs^{m}}_{x};a_{1},\dots,a_{k};\overrightarrow{v's^{m'}}_{y})\coloneqq\sum_{0\leq i\leq j\leq k}c_{i}c_{j}'I_{\gamma}^{\mathfrak{m}}(\overrightarrow{v}_{x};a_{i+1},\dots,a_{j};\overrightarrow{v'}_{y})
\]
for $s,s'\in\mathbb{Q}_{>0}$ where
\begin{align*}
c_{i} & =I_{{\rm triv}}^{\mathfrak{m}}(\overrightarrow{vs^{m}}_{x};a_{1},\dots,a_{i};\overrightarrow{v}_{x})=\frac{\log^{\mathfrak{m}}(1/s^{m})^{i}}{i!}\prod_{l=1}^{i}{\rm res}(a_{l},x)\\
c_{j}' & =I_{{\rm triv}}^{\mathfrak{m}}(\overrightarrow{v'}_{y};a_{j+1},\dots,a_{k};\overrightarrow{v's^{m'}}_{y})=\frac{\log^{\mathfrak{m}}(s^{m'})^{k-j}}{(k-j)!}\prod_{l=j+1}^{k}{\rm res}(a_{l},y)
\end{align*}
where ${\rm triv}$ is the trivial path, and 
\[
{\rm res}(a,b):=\left(\text{residue of \ensuremath{\frac{dt}{t-a}} at \ensuremath{t=b}}\right)=\begin{cases}
-1 & b=\infty\\
1 & a=b\in\mathbb{Q}\\
0 & a\neq b\in\mathbb{Q}.
\end{cases}
\]
Thus 
\[
p(T)=\sum_{0\leq i\leq j\leq k}\widetilde{c_{i}}\widetilde{c_{j}}'I_{\gamma}^{\mathfrak{m}}(\overrightarrow{v}_{x};a_{i+1},\dots,a_{j};\overrightarrow{v'}_{y})\in\mathcal{H}[T]
\]
with
\[
\widetilde{c_{i}}\coloneqq\frac{(-mT)^{i}}{i!}\prod_{l=1}^{i}{\rm res}(a_{l},x),\ \ \ \widetilde{c_{j}}'\coloneqq\frac{(m'T)^{k-j}}{(k-j)!}\prod_{l=j+1}^{k}{\rm res}(a_{l},y)
\]
satisfies the condition
\[
p(\log^{\mathfrak{m}}(s))=I_{\gamma}^{\mathfrak{m}}(\iota_{s}(a_{0});a_{1},\dots,a_{k};\iota_{s}(a_{k+1})).\qedhere
\]
\end{proof}
Note that the shuffle product formula, the path decomposition formula,
and the coproduct formula are naturally generalized to the motivic
iterated integrals with extended tangential base-points above. By
regarding the weight of $T$ as $1$ and defining
\[
\Delta(T)\coloneqq1\otimes T+T\otimes1,
\]
one can equip $\mathfrak{A}^{(T)}\coloneqq\mathfrak{A}[T]$ with a
structure of graded Hopf algebra, and $\mathcal{H}[T]$ with a structure
of graded $\mathfrak{A}^{(T)}$-comodule. We thus extend $D_{r}$
to
\[
\mathcal{H}[T]\to\mathfrak{L}_{r}^{(T)}\otimes\mathcal{H}[T]
\]
by the composite map
\[
\mathcal{H}[T]\xrightarrow{\Delta}\mathfrak{A}^{(T)}\otimes\mathcal{H}[T]\to\mathfrak{A}_{r}^{(T)}\otimes\mathcal{H}[T]\to\mathfrak{L}_{r}^{(T)}\otimes\mathcal{H}[T]
\]
where $\mathfrak{L}^{(T)}=\mathfrak{A}_{>0}^{(T)}/\left(\mathfrak{A}_{>0}^{(T)}\cdot\mathfrak{A}_{>0}^{(T)}\right)\simeq\mathfrak{L}\oplus\mathbb{Q}T$.
\begin{prop}
\label{prop:coproduct_extended}Let $k\geq0$, $a_{0},a_{k+1}\in\widehat{T^{\times}}\mathbb{P}^{1}(\mathbb{Q})$,
$a_{1},\dots,a_{k}\in\widehat{T^{\times}}\mathbb{Q}$, and $\gamma$
a path from $a_{0}$ to $a_{k+1}$. If $\#\{a_{1},\dots,a_{k}\}=\#\{\overline{a_{1}},\dots,\overline{a_{k}}\}$
then we have
\[
\Delta I_{\gamma}^{\mathfrak{m}}(a_{0};a_{1},\dots,a_{k};a_{k+1})\coloneqq\sum_{r=0}^{k}\sum_{0=i_{0}<i_{1}<\cdots<i_{r+1}=k+1}\prod_{j=0}^{r}I^{\mathfrak{a}}(a_{i_{j}};a_{i_{j}+1},\dots,a_{i_{j+1}-1};a_{i_{j+1}})\otimes I_{\gamma}^{\mathfrak{m}}(a_{i_{0}};a_{i_{1}},\dots,a_{i_{r}};a_{i_{r+1}}).
\]
In particular, 
\[
D_{r}(I_{\gamma}^{\mathfrak{m}}(a_{0};a_{1},\dots,a_{k};a_{k+1}))=\sum_{i=0}^{k-r}I^{\mathfrak{l}}(a_{i};a_{i+1},\dots,a_{i+r};a_{i+r+1})\otimes I_{\gamma}^{\mathfrak{m}}(a_{0};a_{1},\dots,a_{i},a_{i+r+1},\dots,a_{k};a_{k+1}).
\]
\end{prop}

\begin{proof}
It follows from Proposition \ref{prop:coproduct} and the definition
of the motivic iterated integrals with extended tangential base-points.
\end{proof}

\subsection{\label{subsec:Goncharov's_Hopf_algebra}Goncharov's Hopf algebra
of formal iterated integrals}

Let $S$ be a set. In this section we review the Hopf algebra $\mathcal{I}_{\bullet}(S)$
of formal iterated integrals introduced by Goncharov \cite{GonSym}.
Following Goncharov, we define the commutative graded Hopf algebra
$\mathcal{I}_{\bullet}(S)$ over $\mathbb{Q}$ as follows. As a commutative
algebra, $\mathcal{I}_{\bullet}(S)$ is generated by the formal symbols
\[
\mathbb{I}(s_{0};s_{1},\ldots,s_{m};s_{m+1})\quad\left(m\geq0,\,s_{0},\ldots,s_{m+1}\in S\right)
\]
whose degree is defined to be homogeneous of $\deg\mathbb{I}(s_{0};s_{1},\ldots,s_{m};s_{m+1})\coloneqq m$,
with the following relations among the generators:
\begin{enumerate}
\item The unit: for $a,b\in S$, $\mathbb{I}(a;b)\coloneqq\mathbb{I}(a;\emptyset;b)\coloneqq1.$
\item The shuffle product formula: for $m,n\geq0$ and $a,b,s_{1},\ldots,s_{m+n}\in S$,
\begin{align*}
\mathbb{I}(a;s_{1},\ldots,s_{m};b)\mathbb{I}(a;s_{m+1},\ldots,s_{m+n};b) & =\mathbb{I}(a;(s_{1},\ldots,s_{m})\shuffle(s_{m+1},\ldots,s_{m+n});b)\\
 & \coloneqq\sum_{\sigma\in\sum_{n,m}}\mathbb{I}(a;s_{\sigma(1)},\ldots,s_{\sigma(n+m)};b)
\end{align*}
 where $\sum_{n,m}\subset\mathfrak{S}_{n+m}$ denotes the set of all
shuffles of $(1,\ldots,m)$ and $(m+1,\ldots,m+n)$.
\item The path composition formula: for $m\geq0$ and $a,b,c,s_{1},\ldots,s_{m}\in S$,
\[
\mathbb{I}(a;s_{1},\ldots,s_{m};c)=\sum_{i=0}^{m}\mathbb{I}(a;s_{1},\ldots,s_{i};b)\mathbb{I}(b;s_{i+1},\ldots,s_{m};c).
\]
\item For $m>0$ and $a,s_{1},\ldots,s_{m}\in S$, $\mathbb{I}(a;s_{1},\ldots,s_{m};a)=0.$
\end{enumerate}
Next, define the $\mathbb{Q}$-linear map $\Delta^{(S)}:\mathcal{I}_{\bullet}(S)\rightarrow\mathcal{I}_{\bullet}(S)\otimes\mathcal{I}_{\bullet}(S)$
by 
\[
\Delta^{(S)}\mathbb{I}(s_{0};s_{1},\ldots,s_{m};s_{m+1})\coloneqq\sum_{\substack{k\geq0\\
0=i_{0}<\cdots<i_{k+1}=m+1
}
}\mathbb{I}(s_{0};s_{i_{1}},\ldots,s_{i_{k}};s_{m+1})\otimes\prod_{p=0}^{k}\mathbb{I}(s_{i_{p}};s_{i_{p}+1},\ldots,s_{i_{p+1}-1};s_{i_{p+1}})
\]
Then, $\Delta^{(S)}$ is co-associative. Moreover, Goncharov proved
the following.
\begin{thm}[{\cite[Proposition 2.2]{GonSym}}]
\label{thm:Goncharov_Hopf_alg}$\Delta^{(S)}$ is a well-defined
coproduct on $\mathcal{I}_{\bullet}(S)$, which provides $\mathcal{I}_{\bullet}(S)$
with the structure of a commutative, graded Hopf algebra. 
\end{thm}

For a graded algebra $A$, we denote by $A_{r}$ its homogeneous degree
$r$ part, and define the ``linearized version'' $A^{\mathfrak{l}}$
to be 
\[
A^{\mathfrak{l}}\coloneqq A_{>0}/(A_{>0}\times A_{>0})
\]
where $A_{>0}\coloneqq\bigoplus_{r>0}A_{r}$. Under these notations,
we denote by $x^{\mathfrak{l}}\in A^{\mathfrak{l}}$ the image of
$x\in A_{>0}$ under the quotient map, and for $r>0$, define the
infinitesimal version\emph{ $D_{r}^{(S)}:\mathcal{I}_{\bullet}(S)\rightarrow\mathcal{I}_{\bullet}(S)_{r}^{\mathfrak{l}}\otimes\mathcal{I}_{\bullet}(S)$}
of $\Delta^{(S)}$ by the composition
\[
\mathcal{I}_{\bullet}(S)\xrightarrow{\Delta^{(S)}}\mathcal{I}_{\bullet}(S)\otimes\mathcal{I}_{\bullet}(S)\xrightarrow{\mathrm{pr}_{r}\otimes\mathrm{id}}\mathcal{I}_{\bullet}(S)_{r}\otimes\mathcal{I}_{\bullet}(S)\xrightarrow{(x\mapsto x^{\mathfrak{l}})\otimes\mathrm{id}}\mathcal{I}_{\bullet}(S)_{r}^{\mathfrak{l}}\otimes\mathcal{I}_{\bullet}(S)
\]
where $\mathrm{pr}_{r}$ is the projection to the degree $r$ part.

For a graded $\mathbb{Q}$-algebra homomorphism $\varphi:\mathcal{I}_{\bullet}(S)\to\mathcal{H}[T]$
and a sequence $\bm{s}=(s_{0};s_{1},\dots,s_{m};s_{m+1})$ of the
elements of $S$, the condition
\[
D_{r}(\varphi(\mathbb{I}(s_{0};s_{1},\ldots,s_{m};s_{m+1}))=\sum_{i=0}^{m-r}\varphi^{\mathfrak{l}}(\mathbb{I}(s_{i};s_{i+1},\dots,s_{i+r};s_{i+r+1}))\otimes\varphi(\mathbb{I}(s_{0};s_{1},\dots,s_{i},s_{i+r+1},\dots,s_{m};s_{m+1}))
\]
where $\varphi^{\mathfrak{l}}$ is the composite map $\mathcal{I}_{\bullet}(S)\xrightarrow{\varphi}\mathcal{H}[T]\to\mathfrak{L}^{(T)}$,
is equivalent to
\[
D_{r}(\varphi(x))=(\varphi^{\mathfrak{l}}\otimes\varphi)\circ D_{r}^{(S)}(x)
\]
where $x\coloneqq\mathbb{I}(\bm{s})$.
\begin{lem}
\label{lem:compatibility_of_Brown_with_prod}Let $\varphi:\mathcal{I}_{\bullet}(S)\to\mathcal{H}[T]$
be a homomorphism of graded $\mathbb{Q}$-algebras and $x,y$ elements
of $\mathcal{I}_{\bullet}(S)$. If $D_{r}(\varphi(x))=(\varphi^{\mathfrak{l}}\otimes\varphi)\circ D_{r}^{(S)}(x)$
and $D_{r}(\varphi(y))=(\varphi^{\mathfrak{l}}\otimes\varphi)\circ D_{r}^{(S)}(y)$
then
\[
D_{r}(\varphi(xy))=(\varphi^{\mathfrak{l}}\otimes\varphi)\circ D_{r}^{(S)}(xy).
\]
\end{lem}

\begin{proof}
It follows from the following calculation.
\begin{align*}
D_{r}(\varphi(xy)) & =D_{r}(\varphi(x)\cdot\varphi(y))\\
 & =D_{r}(\varphi(x))\cdot(1\otimes\varphi(y))+D_{r}(\varphi(y))\cdot(1\otimes\varphi(x))\\
 & =\left((\varphi^{\mathfrak{l}}\otimes\varphi)\circ D_{r}^{(S)}(x)\right)\cdot(1\otimes\varphi(y))+\left((\varphi^{\mathfrak{l}}\otimes\varphi)\circ D_{r}^{(S)}(y)\right)\cdot(1\otimes\varphi(x))\\
 & =\left(\varphi^{\mathfrak{l}}\otimes\varphi\right)(D_{r}^{(S)}(x)\cdot(1\otimes y))+\left(\varphi^{\mathfrak{l}}\otimes\varphi\right)(D_{r}^{(S)}(x)\cdot(1\otimes y))\\
 & =\left(\varphi^{\mathfrak{l}}\otimes\varphi\right)\circ D_{r}^{(S)}(xy).\qedhere
\end{align*}
\end{proof}

\section{\label{sec:Evaluation-map}The motivic counterpart of regularized
limit of iterated integrals}

For $X\subset\mathbb{Q}(z)$ we define the set of tangential base
points with a variable $z$ by
\[
T_{z}^{\times}X\coloneqq\{(x,v)\mid v\in\mathbb{Q}(z)^{\times},x\in X\},
\]
and denote its element $(x,v)$ as $\overrightarrow{v}_{x}$ just
like a usual tangential base point. For $p=\overrightarrow{v}_{x}\in T_{z}^{\times}X$,
we continue to use the notations $\bar{p}=x$ and ${\rm vec}(p)=v$.

Let $z>0$ be a real variable which is assumed to be sufficiently
small. For $a(z),b(z)\in T_{z}^{\times}\mathbb{Q}(z)$ and $X\subset\mathbb{Q}(z)$,
a ``path $\gamma(z)$ from $a(z)$ to $b(z)$ on $\mathbb{C}\setminus X$''
is defined as follows. For a fixed $z$, $\gamma(z)$ is a homotopy
class of paths from $a(z)$ to $b(z)$ on $\mathbb{C}\setminus\{p(z)\mid p\in X\}$
and when $z$ varies $z\mapsto\gamma(z)$ is a continuous function. 

Now, let $k\geq0$, $p_{0}=p_{0}(z),\dots,p_{k+1}=p_{k+1}(z)\in T_{z}^{\times}\mathbb{Q}(z)$
and $\gamma=\gamma(z)$ be a path from $p_{0}(z)$ to $p_{k+1}(z)$
on $\mathbb{C}\setminus\{\overline{p_{1}(z)},\ldots,\overline{p_{k}(z)}\}$.
Then,
\[
\rho(z)=I_{\gamma(z)}(p_{0}(z);p_{1}(z),\dots,p_{k}(z);p_{k+1}(z))
\]
becomes a function of $z$. Then, we can show that there exists a
polynomial $P(T)\in\mathbb{C}[T]$ such that
\[
\rho(z)=P(\log z)+O(z(\log z)^{M})\ \ \ (z\to+0)
\]
for sufficiently large $M>0$. Such a $P(T)$ is denoted as
\[
\Reg_{z\rightarrow+0}^{(T)}\left(I_{\gamma(z)}(p_{0}(z);p_{1}(z),\dots,p_{k}(z);p_{k+1}(z))\right).
\]
The purpose of this section is to construct an element of $\mathcal{H}[T]$
denote by
\begin{equation}
J_{\gamma(z)}(p_{0}(z);p_{1}(z),\dots,p_{k}(z);p_{k+1}(z))\label{eq:eval_map}
\end{equation}
such that
\[
\compmap\left(J_{\gamma(z)}(p_{0}(z);p_{1}(z),\dots,p_{k}(z);p_{k+1}(z))\right)=P(T)
\]
where $\compmap:\mathcal{H}[T]\longrightarrow\mathbb{R}[T]$ is the
natural extension of the period map $\compmap:\mathcal{H}\longrightarrow\mathbb{C}$,
based on the idea explained in \cite[Section 3.3.3]{Panzer}, and
investigate its properties. As the quantity $J_{\gamma}(p_{0};p_{1},\ldots,p_{k};p_{k+1})$
is the motivic counterpart of the regularized limit $\Reg_{z\rightarrow+0}^{(T)}\left(I_{\gamma(z)}(p_{0}(z);p_{1}(z),\dots,p_{k}(z);p_{k+1}(z))\right)$
of a complex-valued iterated integral, those objects should satisfy
the same algebraic properties. From this perspective, we define $J_{\gamma}(p_{0};p_{1},\ldots,p_{k};p_{k+1})$
by reducing it to special cases. First of all, we introduce the simplest
path called the \emph{upper path} as follows.
\begin{defn}
We call $\gamma$ the upper path if $\gamma$ lies on the set $\{s\in\mathbb{C}\mid\Im(s)\geq0\}$,
i.e., the upper half plane including the real axis. Note that once
two points $p,q\in\mathbb{R}$ and $Y\subset\mathbb{R}$ are given,
the upper path from $p$ to $q$ on $\mathbb{C}\setminus Y$ is unique
up to the homotopy equivalence and so we denote it by ``$\mathrm{up}$''. 
\end{defn}

The most general definition of $J_{\gamma}(p_{0};p_{1},\ldots,p_{k};p_{k+1})$
(for a general sequence $p_{0},p_{1},\ldots,,p_{k+1}\in T_{z}^{\times}\mathbb{Q}(z)$
and a general path $\gamma$) is not given by a simple formula, and
so we break down the definition into several steps. Here we give an
overview of the contents of this Section. Throughout (sub-)sections
\ref{subsec:Behavior-at-inf}--\ref{subsec:path-composition-for-J},
we specify ourselves to the case where $\gamma$ is the upper path.
In Section \ref{subsec:Behavior-at-inf}, we discuss the behavior
of the complex-valued iterated integral $I_{\mathrm{up}}$ with a
tangential point at infinity, and in Section \ref{subsec:calc-for-regular-seq},
we show that $\Reg_{z\rightarrow+0}^{(T)}\left(I_{\gamma(z)}(p_{0}(z);p_{1}(z),\dots,p_{k}(z);p_{k+1}(z))\right)$
can be calculated simply by ``substitution of $z=0$'' if the sequence
$(p_{0},p_{1},\dots,,p_{k+1})$ satisfy a certain regularity condition,
and thus define $J_{\mathrm{up}}(p_{0};p_{1},\ldots,p_{k};p_{k+1})$
for a regular sequence $(p_{0},p_{1},\dots,,p_{k+1})$. Then by certain
decomposition techniques using the shuffle product formula, we define
$J_{\mathrm{up}}(p_{0};p_{1},\ldots,p_{k};p_{k+1})$ for a general
admissible sequence with different endpoints (Section \ref{subsec:calc-for-adm-seq}),
a general sequence with different endpoints (Section \ref{subsec:calc-for-nonadm-seq}),
and a general sequence with the same endpoints (Section \ref{subsec:calc-for-equal-endpts}).
We then prove that the quantity $J_{\mathrm{up}}(p_{0};p_{1},\ldots,p_{k};p_{k+1})$
that we have defined satisfy the path composition formula (Section
\ref{subsec:path-composition-for-J}), and by extending this property
for general path we define $J_{\gamma}(p_{0};p_{1},\ldots,p_{k};p_{k+1})$
for arbitrary path $\gamma$ (Section \ref{subsec:limit-for-any-path}).
Furthermore, we shall prove that $J_{\gamma}(p_{0};p_{1},\ldots,p_{k};p_{k+1})$
thus defined enjoys the infinitesimal coaction just as the motivic
iterated integrals do (Section \ref{subsec:inf-coaction-for-J}).

\subsection{\label{subsec:Behavior-at-inf}The behavior of iterated integrals
at infinity}
\begin{prop}
\label{prop:behavior_at_infinity}For $a_{1},\ldots,a_{k+1}\in T^{\times}\mathbb{Q}$,
\[
I_{\mathrm{up}}(y^{-1};a_{1},\ldots,a_{k};a_{k+1})=I_{\mathrm{up}}(\overrightarrow{y}_{\infty};a_{1},\ldots,a_{k};a_{k+1})+O(y(\log y)^{k+1})
\]
as $y$ goes to $+0$.
\end{prop}

\begin{proof}
Let $P(T)$ be a polynomial such that $I_{\mathrm{up}}((y\epsilon)^{-1};a_{1},\ldots,a_{k};a_{k+1})=P(\log\epsilon)+O(\epsilon(\log\epsilon)^{k+1})$
for $\epsilon\rightarrow+0$. By definition 
\[
I_{\mathrm{up}}(\overrightarrow{y}_{\infty};a_{1},\ldots,a_{k};a_{k+1})=P(0).
\]
Let $Q(T)$ be a polynomial such that $I_{\mathrm{up}}(s^{-1};a_{1},\ldots,a_{k};a_{k+1})=Q(\log s)+O(s(\log s)^{k+1})$
for $s\rightarrow+0$. Then, by putting $s=y\epsilon$, 
\begin{align*}
Q(\log y+\log\epsilon) & =I_{\mathrm{up}}((y\epsilon)^{-1},a_{1},\ldots,a_{k};a_{k+1})+O(\epsilon(\log\epsilon)^{k+1})\\
 & =P(\log\epsilon)+O(\epsilon(\log\epsilon)^{k+1})
\end{align*}
for $\epsilon\rightarrow+0$. Thus as a polynomial $P(T)=Q(T+\log y)$.
Now that $Q(\log y)=P(0)=I_{\mathrm{up}}(\overrightarrow{y}_{\infty},a_{1},\ldots,a_{k};a_{k+1})$,
it readily follows that 
\begin{align*}
I_{\mathrm{up}}(y^{-1};a_{1},\ldots,a_{k};a_{k+1}) & =I_{\mathrm{up}}(\overrightarrow{y}_{\infty};a_{1},\ldots,a_{k};a_{k+1})+O(y(\log y)^{k+1}).\qedhere
\end{align*}
\end{proof}
Now consider the case where $y=y(z)=cz^{m}(1+p(z))$ with $c\in\mathbb{Q}^{\times}$,
$m\in\mathbb{Z}_{>0}$ and $p(z)\in\mathbb{Q}(z)$ with $p(0)=0$
in Proposition \ref{prop:behavior_at_infinity}. In this case, we
have
\begin{align*}
I_{\mathrm{up}}(y(z)^{-1};a_{1},\ldots,a_{k};a_{k+1}) & =I_{\mathrm{up}}(\overrightarrow{y(z)}_{\infty};a_{1},\ldots,a_{k};a_{k+1})+O(z(\log z)^{k+1})\\
 & =I_{\mathrm{up}}(\overrightarrow{cz^{m}}_{\infty};a_{1},\ldots,a_{k};a_{k+1})+O(z(\log z)^{k+1})\\
 & =\restr{I_{\mathrm{up}}(\overrightarrow{ce^{mT}}_{\infty};a_{1},\ldots,a_{k};a_{k+1})}{T=\log z}+O(z(\log z)^{k+1}).
\end{align*}
Thus we find that $\Reg_{z\rightarrow+0}^{(T)}I_{\mathrm{up}}(y(z)^{-1};a_{1},\ldots,a_{k};a_{k+1})=I_{\mathrm{up}}(\overrightarrow{ce^{mT}}_{\infty};a_{1},\ldots,a_{k};a_{k+1})$,
which leads us to the following definition of the limit as $z\rightarrow+0$
of a tangential base point with a variable.
\begin{defn}
For $p=\overrightarrow{v}_{x}\in T_{z}^{\times}\mathbb{Q}(z)$, we
define $p[0]\in\widehat{T^{\times}}\mathbb{P}^{1}(\mathbb{Q})$ by
$\overline{p[0]}\coloneqq x(0)$ and
\[
{\rm vec}(p[0])\coloneqq\begin{cases}
\hat{v} & \text{if }x(0)\neq\infty\\
\hat{x}^{-1} & \text{if }x(0)=\infty.
\end{cases}
\]
Here, for $x\in\mathbb{Q}(z)$, we set $\hat{x}:=x'(0)e^{mT}\in\mathbb{Q}[[T]]$
if $x=x'z^{m}$ with $m=\mathrm{ord}_{z}x\in\mathbb{Z}$.
\end{defn}

By this notation, the above identity can be written simply as 
\[
\Reg_{z\rightarrow+0}^{(T)}I_{\mathrm{up}}(p(z);a_{1},\dots,a_{k};a_{k+1})=I_{\mathrm{up}}(p[0];a_{1},\dots,a_{k};a_{k+1})
\]
if $p(0)=\infty$ and $a_{1},\ldots,a_{k+1}\in T^{\times}\mathbb{Q}$.
In fact, we can show more generally that 
\begin{equation}
\Reg_{z\rightarrow+0}^{(T)}I_{\mathrm{up}}(p(z);a_{1},\dots,a_{k};q(z))=I_{\mathrm{up}}(p[0];a_{1},\dots,a_{k};q[0])\label{eq:calc_Reg_const}
\end{equation}
for $a_{1},\dots,a_{k}\in\mathbb{Q}$ (here, we don't even need the
conditions $p(0)\neq a_{1}$ and $q(0)\neq a_{k}$).

\subsection{\label{subsec:calc-for-regular-seq}The calculation of the regularized
limit for regular sequences}

For $k\geq0$, the sequence $(p_{0},\dots,p_{k+1})\in\mathbb{Q}(z)^{k+2}$
of length $k+2$ is called empty if $k=0$ and admissible if $k=0$
or $p_{0}\neq p_{1}$ and $p_{k}\neq p_{k+1}$.
\begin{defn}
Let $k>0$ and $p_{0},\dots,p_{k+1}\in\mathbb{Q}(z)$. We say that
the admissible sequence $(p_{0},\dots,p_{k+1})$ is \emph{very regular}
if and only if $p_{0}\neq p_{k+1}$, $p_{0}(0)\neq p_{1}(0)$, $p_{k}(0)\neq p_{k+1}(0)$,
and for $t\in\{1,\dots,k\}$ neither $p_{0}(0)=p_{t}(0)=\infty$ nor
$p_{k+1}(0)=p_{t}(0)=\infty$ occurs.
\end{defn}

We say that $(p_{0},\dots,p_{k+1})\in\left(T_{z}^{\times}\mathbb{Q}(z)\right)^{k+2}$
is very regular (resp. admissible) if $(\overline{p_{0}},\dots,\overline{p_{k+1}})\in\mathbb{Q}(z)^{k+2}$
is very regular (resp. admissible). 
\begin{lem}
Let $k>0$ and $p_{0},\dots,p_{k+1}\in T_{z}^{\times}\mathbb{Q}(z)$.
If $(p_{0},\dots,p_{k+1})$ is very regular then
\[
\Reg_{z\rightarrow+0}^{(T)}\left(I_{\mathrm{up}}(p_{0}(z);p_{1}(z),\dots,p_{k}(z);p_{k+1}(z))\right)=I_{\mathrm{up}}(p_{0}[0];p_{1}[0],\dots,p_{k}[0];p_{k+1}[0]).
\]
\end{lem}

The proof easily follows from equation (\ref{eq:calc_Reg_const})
and the theorem of dominated convergence (c.f. \cite[Lemma 3.3.35]{Panzer}).

Now, let us consider the affine transformation of an iterated integral.
For $p=(x,v)\in T_{z}^{\times}\mathbb{Q}(z)$, we define a multiplication
by $\alpha\in\mathbb{Q}(z)^{\times}$ by $\alpha\cdot p\coloneqq(\alpha x,\alpha v)\in T_{z}^{\times}\mathbb{Q}(z)$
and an addition by $y\in\mathbb{Q}(z)$ by $p+y\coloneqq(x+y,v)\in T_{z}^{\times}\mathbb{Q}(z)$. 
\begin{defn}
Let $k\geq0$ and $p_{0},\dots,p_{k+1},q_{0},\dots,q_{k+1}\in M$
where $M=T_{z}^{\times}\mathbb{Q}(z)$ or $\mathbb{Q}(z)$. We say
that $(p_{0},\dots,p_{k+1})$ and $(q_{0},\dots,q_{k+1})$ are \emph{affine-equivalent}
if there exists $m\in\mathbb{Z}$ and $y\in\mathbb{Q}(z)$ such that
$q_{i}=z^{m}\cdot p_{i}+y$ for all $0\leq i\leq k+1$. 
\end{defn}

Note that
\[
I_{{\rm up}}(p_{0};p_{1},\dots,p_{k};p_{k+1})=I_{{\rm up}}(q_{0};q_{1},\dots,q_{k};q_{k+1})
\]
if $(p_{0},\dots,p_{k+1})$ and $(q_{0},\dots,q_{k+1})$ are affine-equivalent.
\begin{defn}
Let $k\geq0$ and $p_{0},\dots,p_{k+1}\in\mathbb{Q}(z)$. We say that
the admissible sequence $(p_{0},\dots,p_{k+1})$ is \emph{regular}
if it is affine-equivalent to a very regular sequence. 
\end{defn}

Here again, we say that $(p_{0},\dots,p_{k+1})\in\left(T_{z}^{\times}\mathbb{Q}(z)\right)^{k+2}$
is regular if $(\overline{p_{0}},\dots,\overline{p_{k+1}})\in\mathbb{Q}(z)^{k+2}$
is regular.

For $p,q\in\mathbb{Q}(z)$, define a distance\footnote{The number ``$2$'' in the definition of $d(p,q)$ is not important.
One may use any distance function where $2$ is replace by some real
number $\varepsilon>1$, and the following argument goes just parallelly.} of $p$ and $q$ by
\[
d(p,q)=\begin{cases}
2^{-{\rm ord}_{z=0}(p-q)} & p\neq q\\
0 & p=q.
\end{cases}
\]
Now we can classify all regular sequences to three patterns. For
$m\geq1$, we denote by $\mathcal{R}_{m}$ the set of non-empty admissible
sequences $(p_{0},\dots,p_{k+1})$ of rational functions such that
$d(p_{0},p_{i})\leq d(p_{0},p_{1})=2^{-m}d(p_{0},p_{k+1})\neq0$ for
$1\leq i\leq k$. For $m\geq1$, we denote by $\mathcal{R}_{-m}$
the set of non-empty admissible sequences $(p_{0},\dots,p_{k+1})$
of rational functions such that $d(p_{k+1},p_{i})\leq d(p_{k+1},p_{k})=2^{-m}d(p_{0},p_{k+1})\neq0$
for $1\leq i\leq k$. We denote by $\mathcal{R}_{0}$ the set of non-empty
admissible sequences $(p_{0},\dots,p_{k+1})$ of rational functions
such that $d(p_{0},p_{1}),d(p_{k},p_{k+1})\geq d(p_{0},p_{k+1})\neq0$.
\begin{lem}
The set of regular sequences of rational functions can be decomposed
as $\bigsqcup_{n\in\mathbb{Z}}\mathcal{R}_{n}$.
\end{lem}

\begin{proof}
First let us check disjointness of $(\mathcal{R}_{n})_{n\in\mathbb{Z}}$.
Put $\bm{p}=(p_{0};p_{1},\dots,p_{k};p_{k+1})$. Then $(\mathcal{R}_{n})_{n\in\mathbb{Z}}$
are disjoint since
\begin{itemize}
\item if $\bm{p}\in\bigsqcup_{m=1}^{\infty}\mathcal{R}_{m}$ then $d(p_{0},p_{1})<d(p_{0},p_{k+1})$
and $d(p_{k},p_{k+1})=d(p_{0},p_{k+1})$;
\item if $\bm{p}\in\bigsqcup_{m=1}^{\infty}\mathcal{R}_{-m}$ then $d(p_{0},p_{1})=d(p_{0},p_{k+1})$
and $d(p_{k},p_{k+1})<d(p_{0},p_{k+1})$;
\item if $\bm{p}\in\mathcal{R}_{0}$ then $d(p_{0},p_{1})\geq d(p_{0},p_{k+1})$
and $d(p_{k},p_{k+1})\geq d(p_{0},p_{k+1})$.
\end{itemize}
Let us show $\bm{p}\in\bigsqcup_{n\in\mathbb{Z}}\mathcal{R}_{n}$
for all regular sequence $\bm{p}=(p_{0};p_{1},\dots,p_{k};p_{k+1})$.
Since each $\mathcal{R}_{n}$ is invariant under the affine transformation,
we can assume that $\bm{p}$ is very regular. Note that if $p_{0}(0)=p_{k+1}(0)=\infty$
then $(zp_{0};zp_{1},\dots,zp_{k};zp_{k+1})$ is also very regular.
Thus, by multiplying $z^{m}$ with some positive $m$, we can assume,
without loss of generality, that at least one of $p_{0}(0)$ or $p_{k+1}(0)$
is finite. If $p_{0}(0)\neq\infty$ and $p_{k+1}(0)=\infty$ then
we have $d(p_{0},p_{k+1})>1$, $d(p_{0},p_{1})=1$, and $d(p_{0},p_{i})\leq1$
for $1\leq i\leq k$, and thus $\bm{p}\in\bigsqcup_{m=1}^{\infty}\mathcal{R}_{m}$.
Similarly, if $p_{0}(0)=\infty$ and $p_{k+1}(0)\neq\infty$ then
$\bm{p}\in\bigsqcup_{m=1}^{\infty}\mathcal{R}_{-m}$. If $p_{0}(0)\neq\infty$
and $p_{k+1}(0)\neq\infty$ then $d(p_{0},p_{k+1})\leq1$, $d(p_{0},p_{1})\geq1$
and $d(p_{k},p_{k+1})\geq1$, and thus $\bm{p}\in\mathcal{R}_{0}$.

Conversely, let us show that any $\bm{p}\in\bigsqcup_{n\in\mathbb{Z}}\mathcal{R}_{n}$
is a regular sequence. Let $m>0$. Let us show $\bm{p}\in\mathcal{R}_{m}$
is a regular sequence. Since $p_{0}\neq p_{1}$, we may assume without
loss of generality that $p_{0}=0$ and $d(p_{0},p_{1})=1$. Then,
by definition, $d(p_{0},p_{k+1})=2^{m}$ and $d(p_{0},p_{i})\leq1$
for all $1\leq i\leq k$. Thus $p_{0}(0)=0$, $p_{1}(0)\neq0$, and
$p_{i}(0)\neq\infty$ for all $1\leq i\leq k$ and $p_{k+1}(0)=\infty$.
Thus $\bm{p}$ is very regular. Similarly, $\bm{p}\in\mathcal{R}_{-m}$
is also a regular sequence. Now, assume that $\bm{p}\in\mathcal{R}_{0}$.
Since $p_{0}\neq p_{k+1}$, we may assume without loss of generality
that $p_{0}=0$ and $d(p_{0},p_{k+1})=1$. Since $\bm{p}\in\mathcal{R}_{0}$,
we have $d(p_{0},p_{1})\geq1$ and $d(p_{k},p_{k+1})\geq1$. Thus
$p_{0}(0)=0\neq\infty$, $p_{k+1}(0)\neq\infty$, $p_{0}(0)\neq p_{1}(0)$,
and $p_{k}(0)\neq p_{k+1}(0)$, and hence $\bm{p}$ is very regular.
This completes the proof of the lemma.
\end{proof}
\begin{defn}
\label{def:J_for_regular_sequence}Let $k>0$ and $p_{0},\dots,p_{k+1}\in T_{z}^{\times}\mathbb{Q}(z)$.
Assume that $(p_{0},\dots,p_{k+1})$ is regular\emph{.} Then we define
$J_{\mathrm{up}}(p_{0};p_{1},\dots,p_{k};p_{k+1})\in\mathcal{H}[T]$
by\emph{
\[
J_{\mathrm{up}}(p_{0};p_{1},\dots,p_{k};p_{k+1})\coloneqq I_{\mathrm{up}}^{\mathfrak{m}}(p_{0}'[0];p_{1}'[0],\dots,p_{k}'[0];p_{k+1}'[0])
\]
}where $(p_{0}',\dots,p_{k+1}')$ is a very regular sequence which
is affine-equivalent to $(p_{0},\dots,p_{k+1})$.
\end{defn}

Let us check well-definedness of Definition \ref{def:J_for_regular_sequence}.
i.e.,
\[
I_{\mathrm{up}}^{\mathfrak{m}}(p_{0}[0];p_{1}[0],\dots,p_{k}[0];p_{k+1}[0])=I_{\mathrm{up}}^{\mathfrak{m}}(p_{0}'[0];p_{1}'[0],\dots,p_{k}'[0];p_{k+1}'[0])
\]
for affine-equivalent very regular sequences $(p_{0},\dots,p_{k+1}),(p_{0}',\dots,p_{k+1}')\in(T_{z}^{\times}\mathbb{Q}(z))^{k+2}$
for $k>0$.
\begin{lem}
\label{lem:shifting_invariance}Let $k>0$, $p_{0},\dots,p_{k+1}\in T_{z}^{\times}\mathbb{Q}(z)$,
and $y\in\mathbb{Q}(z)$. Put $p_{i}'=p_{i}+y$ for $0\leq i\leq k+1$.
If $(p_{0},\dots,p_{k+1})$ and $(p_{0}',\dots,p_{k+1}')$ are very
regular sequences then
\[
I_{\mathrm{up}}^{\mathfrak{m}}(p_{0}[0];p_{1}[0],\dots,p_{k}[0];p_{k+1}[0])=I_{\mathrm{up}}^{\mathfrak{m}}(p_{0}'[0];p_{1}'[0],\dots,p_{k}'[0];p_{k+1}'[0]).
\]
\end{lem}

\begin{proof}
The case $y(0)\neq\infty$ is obvious since
\[
I_{\mathrm{up}}^{\mathfrak{m}}(p_{0}'[0];p_{1}'[0],\dots,p_{k}'[0];p_{k+1}'[0])=I_{\mathrm{up}}^{\mathfrak{m}}(p_{0}[0]+y(0),p_{1}[0]+y(0),\dots,p_{k}[0]+y(0);p_{k+1}[0]+y(0)).
\]
Assume that $y(0)=\infty$. Then $\overline{p_{0}}(0)=\infty$ or
$\overline{p_{0}'}(0)=\infty$. By symmetry, we can assume that $\overline{p_{0}}(0)=\infty$.
Since $(p_{0},\dots,p_{k+1})$ is very regular, $\overline{p_{i}}(0)\neq\infty$
for all $1\leq i\leq k$. Thus $\overline{p_{i}'}(0)=\infty$ for
all $1\leq i\leq k$. Since $(p_{0}',\dots,p_{k+1}')$ is very regular,
we have
\[
\overline{p_{0}'}(0)\neq\infty,\ \overline{p_{k+1}'}(0)\neq\infty.
\]
By $\overline{p_{0}'}(0)\neq\infty,\ \overline{p_{k+1}'}(0)\neq\infty$,
we can easily check $p_{0}[0]=p_{k+1}[0]$. Therefore, 
\begin{align*}
I_{\mathrm{up}}^{\mathfrak{m}}(p_{0}'[0];p_{1}'[0],\dots,p_{k}'[0];p_{k+1}'[0]) & =I_{\mathrm{up}}^{\mathfrak{m}}(p_{0}'[0];\infty,\dots,\infty;p_{k+1}'[0])=0\\
I_{\mathrm{up}}^{\mathfrak{m}}(p_{0}[0];p_{1}[0],\dots,p_{k}[0];p_{k+1}[0]) & =I_{\mathrm{up}}^{\mathfrak{m}}(p_{0}[0];p_{1}[0],\dots,p_{k}[0];p_{0}[0])=0.
\end{align*}
\end{proof}
\begin{lem}
Definition \ref{def:J_for_regular_sequence} does not depend on the
choice of very regular sequences.
\end{lem}

\begin{proof}
Recall that the affine-equivalence of two sequence $(p_{0},\dots,p_{k+1}),(p_{0}',\dots,p_{k+1}')\in T_{z}^{\times}\mathbb{Q}(z)$
means that there exists $m\in\mathbb{Z},y\in\mathbb{Q}(z)$ such that
$p_{i}'=z^{m}p_{i}+y$ for $0\leq i\leq k+1$. By symmetry, we may
assume that $m\geq0$. As we have already treated the case for $m=0$,
we assume that $m>0.$ So let $(p_{0},\dots,p_{k+1}),(p_{0}',\dots,p_{k+1}')\in T_{z}^{\times}\mathbb{Q}(z)$
be two very-regular sequences such that $p_{i}'=z^{m}p_{i}+y$ for
$0\leq i\leq k+1$ with some $m>0$ and $y\in\mathbb{Q}(z)$. If suffices
to prove that 
\[
I_{\mathrm{up}}^{\mathfrak{m}}(p_{0}[0];p_{1}[0],\dots,p_{k}[0];p_{k+1}[0])=I_{\mathrm{up}}^{\mathfrak{m}}(p_{0}'[0];p_{1}'[0],\dots,p_{k}'[0];p_{k+1}'[0]).
\]
Let us first consider the case $\overline{p_{0}}(0)\neq\infty$ and
$\overline{p_{k+1}}(0)\neq\infty$. Then $p_{0}'[0]=p_{k+1}'[0]$.
Since $\overline{p_{0}'}(0)\neq\overline{p_{1}'}(0)$, $\overline{p_{1}}(0)=\infty$.
Thus
\[
I_{\mathrm{up}}^{\mathfrak{m}}(p_{0}[0];p_{1}[0],\dots,p_{k}[0];p_{k+1}[0])=0=I_{\mathrm{up}}^{\mathfrak{m}}(p_{0}'[0];p_{1}'[0],\dots,p_{k}'[0];p_{k+1}'[0]).
\]

Next, let us consider the case $\overline{p_{0}}(0)\neq\infty$ and
$\overline{p_{k+1}}(0)=\infty$. Then $\overline{p_{1}}(0)\neq\infty$
by the very-regularity of $(p_{0},\dots,p_{k+1})$. Thus $p_{0}'[0]=p_{1}'[0]$,
which contradicts with the very-regularity of $(p_{0}',\dots,p_{k+1}')$.
By the same reasoning, $\overline{p_{0}}(0)=\infty$ and $\overline{p_{k+1}}(0)\neq\infty$
contradicts with the assumption.

Finally, let us consider the case $\overline{p_{0}}(0)=\infty$ and
$\overline{p_{k+1}}(0)=\infty$. Since $(p_{0},\dots,p_{k+1})$ is
very regular, $\overline{p_{1}}(0),\dots,\overline{p_{k}}(0)\in\mathbb{Q}$.
Thus if $(z^{m}p_{0},\dots,z^{m}p_{k+1})$ is not very regular then
$(\overline{z^{m}p_{0}},\dots,\overline{z^{m}p_{k+1}})|_{z=0}=(0,\ldots,0)$
and thus $(z^{m}p_{0}+y,\dots,z^{m}p_{k+1}+y)$ is also not very regular.
This in turn says that if $(p_{0}',\dots,p_{k+1}')$ is very regular,
$(z^{m}p_{0},\dots,z^{m}p_{k+1})$ is already very regular and so
by Lemma \ref{lem:shifting_invariance}, it is sufficient to consider
the case $y=0$. Put $n_{1}\coloneqq-{\rm ord}_{z=0}(p_{0})$, $a_{1}\coloneqq\lim_{z\to0}(p_{0}z^{n_{1}})$,
$n_{2}\coloneqq-{\rm ord}_{z=0}(p_{k+1})$, $a_{2}\coloneqq\lim_{z\to0}(p_{k+1}z^{n_{2}})$.
Then $(z^{m}p_{0},\dots,z^{m}p_{k+1})$ is very regular iff $m\leq\min(n_{1},n_{2})$.
For example, if $n_{1}=m<n_{2}$ and $a_{2}/a_{1}>0$, then
\begin{align*}
I_{\mathrm{up}}^{\mathfrak{m}}(p_{0}[0];p_{1}[0],\dots,p_{k}[0];p_{k+1}[0]) & =I_{\mathrm{up}}^{\mathfrak{m}}(\overrightarrow{a_{1}^{-1}e^{-mT}}_{\infty};\overline{p_{1}}(0),\dots,\overline{p_{k}}(0);\overrightarrow{a_{2}^{-1}e^{-n_{2}T}}_{\infty})\\
 & =\frac{\left(\log^{\mathfrak{m}}(\frac{a_{2}}{a_{1}})+(n_{2}-m)T\right)^{k}}{k!}
\end{align*}
since the residue of $\frac{dt}{t-p_{i}}$ at $\infty$ is $-1$,
and
\begin{align*}
I_{\mathrm{up}}^{\mathfrak{m}}(p_{0}'[0];p_{1}'[0],\dots,p_{k}'[0];p_{k+1}'[0]) & =I_{\mathrm{up}}^{\mathfrak{m}}(a_{1};\{0\}^{k};\overrightarrow{a_{2}^{-1}e^{-(n_{2}-m)T}}_{\infty})\\
 & =(-1)^{k}I_{\mathrm{up}}^{\mathfrak{m}}(a_{1}^{-1};\{0\}^{k};\overrightarrow{a_{2}^{-1}e^{-(n_{2}-m)T}}_{0})\\
 & =(-1)^{k}I_{\mathrm{up}}^{\mathfrak{m}}((a_{2}/a_{1})e^{(n_{2}-m)T};\{0\}^{k};\overrightarrow{1}_{0})\\
 & =I_{\mathrm{up}}^{\mathfrak{m}}(\overrightarrow{1}_{0};\{0\}^{k};(a_{2}/a_{1})e^{(n_{2}-m)T})\\
 & =\frac{\left(\log^{\mathfrak{m}}(\frac{a_{2}}{a_{1}})+(n_{2}-m)T\right)^{k}}{k!}.
\end{align*}
Here, we use M\"{o}bius transformation $t\mapsto t^{-1}$ for the
second equality. The other cases also follow from similar calculations.
\end{proof}

\subsection{\label{subsec:calc-for-adm-seq}The calculation of the regularized
limit for admissible sequences}

Put $\Sigma:=\bigsqcup_{k\geq0}T_{z}^{\times}\mathbb{Q}(z)^{k}$ and
for a subset $I$ of $\Sigma$, we define $\mathbb{Q}I$ to be the
formal $\mathbb{Q}$-span of $I$. Fix $p,q\in T_{z}^{\times}\mathbb{Q}(z)$
such that $\overline{p}\neq\overline{q}$. We put
\[
I_{{\rm adm}}^{(p,q)}=I_{{\rm adm}}\coloneqq\{(p_{1},\dots,p_{k})\in\Sigma\mid(p,p_{1},\dots,p_{k},q)\text{ is admissible}\}.
\]
For $n\in\mathbb{Z}$, we put 
\[
I_{n}\coloneqq\{(p_{1},\dots,p_{k})\in\Sigma\mid k=0\text{ or }(p,p_{1},\dots,p_{k},q)\in\mathcal{R}_{n}\}.
\]
Then the map
\begin{equation}
\shuffle:\bigotimes_{n\in\mathbb{Z}}\mathbb{Q}I_{{\rm n}}\to\mathbb{Q}I_{{\rm adm}}\label{eq:shu_decom_for_admseq}
\end{equation}
is isomorphic. Thus, we can calculate
\[
\Reg_{z\rightarrow+0}^{(T)}\left(I_{{\rm up}}(p_{0}(z);p_{1}(z),\dots,p_{k}(z);p_{k+1}(z))\right)
\]
for any admissible sequence $(p_{0},\dots,p_{k+1})$. 
\begin{example}
\label{exa:calc_of_limit_for_admseq}Let us calculate
\[
\Reg_{z\rightarrow+0}^{(T)}\left(I_{{\rm up}}(0;z,1;2)\right).
\]
Then sequence $(0;z,1;2)$ is not regular, but by (\ref{eq:shu_decom_for_admseq}),
we can write
\[
(z,1)=(z)\shuffle(1)-(1,z).
\]
Thus
\[
I_{{\rm up}}(0;z,1;2)=I_{{\rm up}}(0;z;2)I_{{\rm up}}(0;1;2)-I_{{\rm up}}(0;1,z;2)
\]
where the sequences
\[
(0,z,2)\in\mathcal{R}_{1},\ (0,1,2)\in\mathcal{R}_{0},\ (0,1,z,2)\in\mathcal{R}_{0}
\]
are all regular. Since
\begin{align*}
\Reg_{z\rightarrow+0}^{(T)}I_{{\rm up}}(0;z;2) & =\Reg_{z\rightarrow+0}^{(T)}I_{{\rm up}}(0;1;\frac{2}{z})=I_{{\rm up}}(0;1;\overrightarrow{\exp(T)/2}_{\infty})\\
\Reg_{z\rightarrow+0}^{(T)}I_{{\rm up}}(0;1;2) & =I_{{\rm up}}(0;1;2)\\
\Reg_{z\rightarrow+0}^{(T)}I_{{\rm up}}(0;1,z;2) & =I_{{\rm up}}(0;1,0;2),
\end{align*}
we obtain
\[
\Reg_{z\rightarrow+0}^{(T)}\left(I_{{\rm up}}(0;z,1;2)\right)=I_{{\rm up}}(0;1;\overrightarrow{\exp(T)/2}_{\infty})I_{{\rm up}}(0;1;2)-I_{{\rm up}}(0;1,0;2).
\]
\end{example}

Based on this calculation, we define $J_{\mathrm{up}}(;;)$ for admissible
sequences as follows.
\begin{defn}
We define a map $F_{p,q}:\mathbb{Q}I_{{\rm adm}}\to\mathcal{H}[T]$
by the composition
\[
\mathbb{Q}I_{{\rm adm}}\xrightarrow{(\shuffle)^{-1}}\bigotimes_{n\in\mathbb{Z}}\mathbb{Q}I_{{\rm n}}\xrightarrow{(\bm{p}_{1}\otimes\cdots\otimes\bm{p}_{r})\mapsto\prod_{i=1}^{r}J_{\mathrm{up}}(p;\bm{p}_{i};q)}\mathcal{H}[T]
\]
where $J_{\mathrm{up}}(p;\bm{p}_{i};q)$ is defined in Definition
\ref{def:J_for_regular_sequence} since $(p,\bm{p}_{i},q)$'s are
regular sequences.
\end{defn}

\begin{defn}
For $(p_{1},\dots,p_{k})\in I_{{\rm adm}}$, we define
\[
J_{\mathrm{up}}(p;p_{1},\dots,p_{k};q)\coloneqq F_{p,q}((p_{1},\dots,p_{k}))\in\mathcal{H}[T].
\]
\end{defn}

\begin{example}
$J_{\mathrm{up}}(0;z,1;2)$ is obtained by replacing (complex-valued)
iterated integrals with the corresponding motivic iterated integral
in the result of Example \ref{exa:calc_of_limit_for_admseq}, i.e.,
\[
J_{\mathrm{up}}(0;z,1;2)=I_{{\rm up}}^{\mathfrak{m}}(0;1;\overrightarrow{\exp(T)/2}_{\infty})I_{{\rm up}}^{\mathfrak{m}}(0;1;2)-I_{{\rm up}}^{\mathfrak{m}}(0;1,0;2).
\]
\end{example}

\subsection{\label{subsec:calc-for-nonadm-seq}The calculation of the regularized
limit for non-admissible sequences}

Fix $p,q\in T_{z}^{\times}\mathbb{Q}(z)$ such that $\overline{p}\neq\overline{q}$.
For $r\in T_{z}^{\times}\mathbb{Q}(z)$, we put $I_{r}\coloneqq\{(r_{1},\dots,r_{k})\in(T_{z}^{\times}\mathbb{Q}(z))^{k}\mid k\geq0,\overline{r_{1}}=\cdots=\overline{r_{k}}=\overline{r}\}$.
Then
\begin{equation}
\shuffle:\mathbb{Q}I_{p}\otimes\mathbb{Q}I_{q}\otimes\mathbb{Q}I_{{\rm adm}}^{(p,q)}\simeq\mathbb{Q}\Sigma.\label{eq:sh_decom_for_nonadm}
\end{equation}
Now for $(p_{1},\dots,p_{k})\in I_{p}\cup I_{q}$, we have
\[
\Reg_{z\rightarrow+0}^{(T)}\left(I_{{\rm up}}(p(z);p_{1}(z),\dots,p_{k}(z);q(z))\right)=I_{{\rm up}}(p'[0];p_{1}'[0],\dots,p_{k}'[0];q'[0])
\]
where $(p',p_{1}',\dots,p_{k}',q')$ is a sequence which is affine-equivalent
to $(p,p_{1},\ldots,p_{k},q)$ satisfying $\{\overline{p}(0),\overline{q}(0),\infty\}=3$.
Therefore, by the isomorphism (\ref{eq:sh_decom_for_nonadm}) we can
calculate $\Reg_{z\rightarrow+0}^{(T)}\left(I_{{\rm up}}(p(z);p_{1}(z),\dots,p_{k}(z);q(z))\right)$
for $(p_{1},\dots,p_{k})\in\Sigma$.

Keeping this calculation in mind, we define 
\[
J_{\mathrm{up}}(p;p_{1},\dots,p_{k};q)\in\mathcal{H}[T]
\]
for $(p_{1},\dots,p_{k})\in\Sigma$ as follows. For $(p_{1},\dots,p_{k})\in I_{p}\cup I_{q}$,
we set
\[
J_{\mathrm{up}}(p;p_{1},\dots,p_{k};q)\coloneqq I_{\mathrm{up}}^{\mathfrak{m}}(p'[0];p_{1}'[0],\dots,p_{k}'[0];q'[0])
\]
where $(p',p_{1}',\dots,p_{k}',q')$ is a sequence which is affine-equivalent
to $(p,p_{1},\ldots,p_{k},q)$ satisfying $\{\overline{p}(0),\overline{q}(0),\infty\}=3$.
Hence, by (\ref{eq:sh_decom_for_nonadm}), we employ the following
definition.
\begin{defn}
\label{def:J_for_the_differnet_endpoints_general}We define $J_{\mathrm{up}}(p;p_{1},\dots,p_{k};q)$
for $(p_{1},\dots,p_{k})\in\Sigma$ as the composed map
\[
\mathbb{Q}\Sigma\xrightarrow{(\shuffle)^{-1}}\mathbb{Q}I_{p}\otimes\mathbb{Q}I_{q}\otimes\mathbb{Q}I_{{\rm adm}}^{(p,q)}\xrightarrow{(\bm{p}_{1}\otimes\bm{p}_{2}\otimes\bm{p}_{3})\mapsto\prod_{i=1}^{3}J_{\mathrm{up}}(p;\bm{p}_{i};q)}\mathcal{H}[T].
\]
\end{defn}

\subsection{\label{subsec:calc-for-equal-endpts}The case of trivial closed paths}

Fix $p,q\in T_{z}^{\times}\mathbb{Q}(z)$ such that $\overline{p}=\overline{q}$.
In this case, one can show for $(p_{1},\dots,p_{k})\in\Sigma$ that
\[
\Reg_{z\rightarrow+0}^{(T)}\left(I_{{\rm up}}(p(z);p_{1}(z),\dots,p_{k}(z);q(z))\right)\coloneqq\begin{cases}
I_{\mathrm{up}}(p'[0];p_{1}'[0],\dots,p_{k}'[0];q'[0]) & \overline{p}=\overline{p_{1}}=\cdots=\overline{p_{k}}\;(=\overline{q})\\
0 & {\rm otherwise}.
\end{cases}
\]
where $(p',p_{1}',\dots,p_{k}',q')$ is a sequence which is affine-equivalent
to $(p,p_{1},\dots,p_{k},q)$ satisfying $\overline{p'}(0)\neq\infty$.
Noting this, we define the value of $J_{\mathrm{up}}$ with the same
endpoints as follows.
\begin{defn}
\label{def:J_for_the_same_endpoints}With the same $(p',p_{1}',\dots,p_{k}',q')$
as above, we define
\[
J_{\mathrm{up}}(p;p_{1},\dots,p_{k};q)\coloneqq\begin{cases}
I_{\mathrm{up}}^{\mathfrak{m}}(p'[0];p_{1}'[0],\dots,p_{k}'[0];q'[0]) & \overline{p}=\overline{p_{1}}=\cdots=\overline{p_{k}}\;(=\overline{q})\\
0 & {\rm otherwise}.
\end{cases}
\]
\end{defn}

\subsection{\label{subsec:path-composition-for-J}The path composition formula
for $J_{\mathrm{up}}$}

In this section we will show that the ring-homomorphism $\mathcal{I}_{\bullet}(T_{z}^{\times}\mathbb{Q}(z))\longrightarrow\mathcal{H}[T]$
which maps $\mathbb{I}(p_{0};p_{1},\ldots,p_{k};p_{k+1})$ to $J_{\mathrm{up}}(p_{0};p_{1},\ldots,p_{k};p_{k+1})$
is well-defined. In other words, $J$ satisfies the same relations
as the defining relations (1)--(4) of $\mathbb{I}$ stated in Section
\ref{subsec:Goncharov's_Hopf_algebra}. Out of the four relations,
(1), (2) and (4) are trivial by definition. Therefore, we shall prove
the property (3) for $J$. Firstly, we state the following two lemmas
which are straightforward by definition.
\begin{lem}
\label{lem:J_vanish}$J_{\mathrm{up}}(p_{0};p_{1},\dots,p_{k};p_{k+1})=0$
if there exists $1\leq i\leq k$ such that $d(p_{0},p_{k+1})<d(p_{0},p_{i})$.
\end{lem}

\begin{lem}
\label{lem:J_ignorable_difference}If $d(q,q')<d(q,p_{1})$ then
\[
J_{\mathrm{up}}(q;p_{1},\dots,p_{k};p_{k+1})=J_{\mathrm{up}}(q';p_{1},\dots,p_{k};p_{k+1}).
\]
\end{lem}

\begin{prop}
\label{prop:path_composition_formula_J}We have
\[
\sum_{i=0}^{k}J_{\mathrm{up}}(p_{0};p_{1},\dots,p_{i};q)J_{\mathrm{up}}(q;p_{i+1},\dots,p_{k};p_{k+1})=J_{\mathrm{up}}(p_{0};p_{1},\dots,p_{k};p_{k+1}).
\]
\end{prop}

\begin{proof}
We prove the claim by induction on $k$. Since the case $k=0$ is
obvious, we can assume that $k>0$. There are two possibilities about
the configurations of $p_{0},p_{k+1},q$: the case $d(p_{0},p_{k+1})<d(p_{0},q)$
and $d(p_{0},p_{k+1})\geq d(p_{0},q)$. In fact, the first case follows
from the second case since if we assume that the claim holds for all
second cases then
\begin{align*}
 & \sum_{i=0}^{k}J_{\mathrm{up}}(p_{0};p_{1},\dots,p_{i};q)J_{\mathrm{up}}(q;p_{i+1},\dots,p_{k};p_{k+1})\\
 & =\sum_{i=0}^{k}\sum_{j=i}^{k}J_{\mathrm{up}}(p_{0};p_{1},\dots,p_{i};q)J_{\mathrm{up}}(q;p_{i+1},\dots,p_{j};p_{0})J_{\mathrm{up}}(p_{0};p_{j+1},\dots,p_{k};p_{k+1})\\
 & =\sum_{j=0}^{k}\left(\sum_{i=0}^{j}(-1)^{j-i}J_{\mathrm{up}}(p_{0};(p_{1},\dots,p_{i})\shuffle(p_{j},\dots,p_{i+1});q)\right)J_{\mathrm{up}}(p_{0};p_{j+1},\dots,p_{k};p_{k+1})\\
 & =\sum_{j=0}^{k}\delta_{j,0}J_{\mathrm{up}}(p_{0};p_{j+1},\dots,p_{k};p_{k+1})\\
 & =J_{\mathrm{up}}(p_{0};p_{1},\dots,p_{k};p_{k+1}).
\end{align*}
Thus it is enough to only consider the case $d(p_{0},p_{k+1})\geq d(p_{0},q)$.
Without loss of generality, we may assume that the sequence $(p_{0},\dots,,p_{k+1})$
is very regular since the general case follows from these special
cases, the shuffle product formula for $J_{\mathrm{up}}(;;)$, and
the compatibility of the shuffle product and the deconcatenation coproduct.

If $p_{i}(0)=\infty$ for some $1\leq i\leq k$ then $p_{0}(0)\neq\infty$
and $p_{k+1}(0)\neq\infty$. Furthermore, by assumption $d(p_{0},p_{k+1})\geq d(p_{0},q)$,
$q(0)\neq\infty$. Thus both hand sides of the claim vanish.

Thus we can assume that $p_{i}(0)\neq\infty$ for $1\leq i\leq k$.
If $p_{0}(0)=\infty$, $d(p_{0},p_{i})=d(p_{0},p_{1})$ and if $p_{0}(0)\neq\infty$,
$d(p_{0},p_{i})\leq d(p_{0},p_{1})=1$ for $1\leq i\leq k$. Therefore,
$d(p_{0},p_{i})\leq d(p_{0},p_{1})$ and by the same argument $d(p_{k+1},p_{i})\leq d(p_{k+1},p_{k})$
for $1\leq i\leq k$. If $d(p_{0},q)<d(p_{0},p_{1})$ then the claim
holds since we have
\begin{align*}
 & \sum_{i=0}^{k}J_{\mathrm{up}}(p_{0};p_{1},\dots,p_{i};q)J_{\mathrm{up}}(q;p_{i+1},\dots,p_{k};p_{k+1})\\
 & =\sum_{i=0}^{k}\delta_{i,0}J_{\mathrm{up}}(q;p_{i+1},\dots,p_{k};p_{k+1})\\
 & =J_{\mathrm{up}}(q;p_{1},\dots,p_{k};p_{k+1})\\
 & =J_{\mathrm{up}}(p_{0};p_{1},\dots,p_{k};p_{k+1})
\end{align*}
by Lemmas \ref{lem:J_vanish} (the first equality) and \ref{lem:J_ignorable_difference}
(the last equality). Thus we can assume that
\[
d(p_{0},p_{1})\leq d(p_{0},q).
\]
Similarly, we can also assume that 
\[
d(p_{k+1},p_{k})\leq d(q,p_{k+1}).
\]
Now recall that the conditions $d(p_{0},p_{i})\leq d(p_{0},p_{1})\leq d(p_{0},q)\leq d(p_{0},p_{k+1})$
and $d(p_{k+1},p_{i})\leq d(p_{k+1},p_{k})\leq d(p_{k+1},q)\leq d(p_{0},p_{k+1})$
are invariant under affine-equivalence. Our proof uses only these
conditions and does not use the very-regularity condition on the sequence
$(p_{0},\ldots,p_{k+1})$, and thus we pass from one sequence to another
affine-equivalent sequence at our convenience hereafter in the proof.
Put $a\coloneqq\min\{d(p_{0},p_{1}),d(p_{k},p_{k+1})\}$ and $U\coloneqq\{p\in T_{z}^{\times}\mathbb{Q}(z)\mid d(p,p_{1})\leq a\}$.
Then $p_{1},\dots,p_{k}\in U$. If $\overline{p_{1}}=\cdots=\overline{p_{k}}$
then the claim follows by a straight forward calculation. Thus, hereafter
we assume that $\{\overline{p_{1}},\dots,\overline{p_{k}}\}>1$. Note
that this condition implies $a\neq0$ and $\max\{d(q,p_{j})\mid1\leq j\leq k\}\neq0$.
Now we split the case into two cases, the case $q\notin U$ and the
case $q\in U$.

Let us first discuss the case $q\notin U$. By affine-equivalence
we may pass to the case where $0\in U$ and $a=1$ without loss of
generality. Now that $p_{1},\dots,p_{k}\in U$ we have $\overline{p_{1}}(0),\dots,\overline{p_{k}}(0)\in\mathbb{Q}$,
and by definition of $a$, we also have $\overline{p_{0}}(0)\neq\overline{p_{1}}(0)$,
$\overline{p_{k}}(0)\neq\overline{p_{k+1}}(0)$. Moreover, since $q\notin U$
we have $\overline{q}(0)=\infty$. Thus, by definition of $J$ for
a very regular sequence and the path composition formula for extended
motivic iterated integral, 
\begin{align*}
 & \sum_{i=0}^{k}J_{\mathrm{up}}(p_{0};p_{1},\dots,p_{i};q)J_{\mathrm{up}}(q;p_{i+1},\dots,p_{k};p_{k+1})\\
 & =\sum_{i=0}^{k}I_{\mathrm{up}}^{\mathfrak{m}}(p_{0}[0];p_{1}(0),\dots,p_{i}(0);q[0])I_{\mathrm{up}}^{\mathfrak{m}}(q[0];p_{i+1}(0),\dots,p_{k}(0);p_{k+1}[0])\\
 & =I_{\mathrm{up}}^{\mathfrak{m}}(p_{0}[0];p_{1}(0),\dots,p_{k}(0);p_{k+1}[0])\\
 & =J_{\mathrm{up}}(p_{0};p_{1},\dots,p_{k};p_{k+1}),
\end{align*}
which proves the claim for the case $q\notin U$.

To discuss the case $q\in U$, we introduce the quantity denoted by
$\Theta$ as follows. For a set $X$, we denote by $\Sigma_{X}$ the
set of (possibly empty) sequences of the elements of $X$, and for
$I\subset\Sigma_{X}$, we define $\mathbb{Q}I$ to be the formal $\mathbb{Q}$-span
of the elements in $I$. For $\bm{Q}=(Q_{1},\dots,Q_{r}),\bm{Q}'=(Q_{1}',\dots,Q_{s}'),\bm{Q}''=(Q_{1}'',\dots,Q_{t}'')\in\Sigma_{X}$,
define $\Theta(\bm{Q};\bm{Q}'';\bm{Q}')$ as the unique element of
$\mathbb{Q}\Sigma_{X}$ satisfying the property
\[
\sum_{i=0}^{r}\sum_{j=0}^{s}(Q_{1},\dots,Q_{i})\shuffle\Theta(Q_{i+1},\dots,Q_{r};\bm{Q}'';Q_{1}',\dots,Q_{j}')\shuffle(Q_{j+1}',\dots,Q_{s}')=(\bm{Q},\bm{Q}'',\bm{Q'}).
\]
We can check that $\Theta(\bm{Q};\bm{Q}'';\bm{Q}')$ is an element
of $\mathbb{Q}I$ where $I\subset\Sigma_{X}$ is the set of the sequences
$(R_{1},\dots,R_{m})$ such that
\begin{itemize}
\item $(R_{1},\dots,R_{m})$ is a permutation of $(Q_{1},\dots Q_{r},Q_{1}'',\dots,Q_{t}'',Q_{1}',\dots,Q_{s}')$,
\item $R_{1}\in\{Q_{1}'',Q_{1}',\dots,Q_{s}'\}$ if $t>0$ and $R_{1}\in\{Q_{1}',\dots,Q_{s}'\}$
if $t=0$,
\item $R_{m}\in\{Q_{t}'',Q_{1},\dots,Q_{r}\}$ if $t>0$ and $R_{m}\in\{Q_{1},\dots,Q_{r}\}$
if $t=0$.
\end{itemize}
Now let us discuss the case $q\in U$. By an affine transformation,
we can pass to the case where $\max\{d(q,p_{j})\mid1\leq j\leq k\}=1$
and $\overline{q}(0)\in\mathbb{Q}$ without loss of generality. Note
that $a\geq1$ since $a\geq\max\{d(q,p_{j})\mid1\leq j\leq k\}=1$.
Recall that by assumption, $\overline{p_{1}}(0),\ldots,\overline{p_{k}}(0)\in\mathbb{Q}$
and $\overline{p_{0}}(0)\neq\overline{p_{1}}(0)$, $\overline{p_{k}}(0)\neq\overline{p_{k+1}}(0)$.
Thus the sequence $(p_{0},\ldots,p_{k+1})$ is very regular, but the
sequences $(p_{0},\ldots,p_{i},q)$ or $(q,p_{i+1},\ldots,p_{k+1})$
may not be very regular because $d(q,p_{i})$ or $d(q,p_{i+1})$ may
be smaller than $1$. Put 
\[
\Gamma\coloneqq\{1\leq j\leq k\mid d(q,p_{j})<1\}.
\]
For each $i\in\{0,\dots,k\}$, let $l(i)$ be the minimum integer
between $0$ and $i$ such that $\{l(i)+1,\dots,i\}\subset\Gamma$,
and $r(i)$ be the maximal integer between $i$ and $k$ such that
$\{i+1,\dots,r(i)\}\subset\Gamma$. By definition of $\Theta$ and
the shuffle product formula for $J_{\mathrm{up}}$, we have
\[
J_{\mathrm{up}}(p_{0};p_{1},\dots,p_{i};q)=\sum_{l(i)\leq i'\leq i}J_{\mathrm{up}}\left(p_{0};\Theta(\emptyset;p_{1},\dots,p_{l(i)};p_{l(i)+1},\dots,p_{i'});q\right)J_{\mathrm{up}}(p_{0};p_{i'+1},\dots,p_{i};q)
\]
and
\[
J_{\mathrm{up}}(q;p_{i+1},\dots,p_{k};p_{k+1})=\sum_{i\leq i''\leq r(i)}J_{\mathrm{up}}\left(q;\Theta(p_{i''+1},\dots,p_{r(i)};p_{r(i)+1},\dots,p_{i''};\emptyset);p_{k+1}\right)J_{\mathrm{up}}(q;p_{i+1},\dots,p_{i''};p_{k+1}).
\]
Thus, we have
\begin{align}
 & \sum_{i=0}^{k}J_{\mathrm{up}}(p_{0};p_{1},\dots,p_{i};q)J_{\mathrm{up}}(q;p_{i+1},\dots,p_{k};p_{k+1})\nonumber \\
 & =\sum_{\substack{0\leq i'\leq i\leq i''\leq k\\
\{i'+1,\dots,i''\}\subset\Gamma
}
}A_{i'}B_{i''}J_{\mathrm{up}}(p_{0};p_{i'+1},\dots,p_{i};q)J_{\mathrm{up}}(q;p_{i+1},\dots,p_{i''};p_{k+1})\label{eq:path_comp_calc1}
\end{align}
where 
\begin{align*}
A_{i'} & \coloneqq J_{\mathrm{up}}\left(p_{0};\Theta(\emptyset;p_{1},\dots,p_{l(i')};p_{l(i')+1},\dots,p_{i'});q\right)\\
B_{i''} & \coloneqq J_{\mathrm{up}}\left(q;\Theta(p_{i''+1},\dots,p_{r(i'')};p_{r(i'')+1},\dots,p_{k};\emptyset);p_{k+1}\right).
\end{align*}
Since $\max\{d(q,p_{j})\mid1\leq j\leq k\}=1$, there exists $1\leq j\leq k$
such that $j\notin\Gamma$ and so the term $(i',i'')=(0,k)$ cannot
happen. Thus, we may apply the induction hypothesis and (\ref{eq:path_comp_calc1})
equals
\[
\sum_{\substack{0\leq i'\leq i''\leq k\\
\{i'+1,\dots,i''\}\subset\Gamma
}
}A_{i'}B_{i''}C_{i',i''}
\]
where $C_{i',i''}:=J_{\mathrm{up}}(p_{0};p_{i'+1},\dots,p_{i''};p_{k+1}).$
Then for any $j\in\Gamma$, we have 
\[
d(p_{0},p_{j})\geq1
\]
since $d(p_{0},q)\geq d(p_{0},p_{1})\geq a\geq1$ and $d(q,p_{j})<1$.
Hence $\overline{p_{0}}(0)\neq\overline{p_{j}}(0)$ for $l(i')+1\leq j\leq i'$.
Note that we also have $\overline{p_{0}}(0)\neq\overline{p_{1}}(0)$,
$\overline{p_{1}}(0),\dots,\overline{p_{i'}}(0)\in\mathbb{Q}$ and
$\overline{p_{l(i')}}(0)\neq\overline{q}(0)$. Thus for any sequence
$(R_{1},\ldots,R_{m})\in\varSigma_{T_{z}^{\times}\mathbb{Q}(z)}$
that appear in $\Theta(\emptyset;p_{1},\dots,p_{l(i')};p_{l(i')+1},\dots,p_{i'})$,
the sequence $(p_{0},R_{1},\ldots,R_{m},q)$ is very regular. Therefore,
\[
A_{i'}=I_{\mathrm{up}}^{\mathfrak{m}}(p_{0}[0];\Theta(\emptyset;p_{1}(0),\dots,p_{l(i')}(0);p_{l(i')+1}(0),\dots,p_{i'}(0));q[0]).
\]
Similarly, we have
\[
B_{i''}=I^{\mathfrak{m}}\left(q[0];\Theta(p_{i''+1}(0),\dots,p_{r(i'')}(0);p_{r(i'')+1}(0),\dots,p_{k}(0);\emptyset);p_{k+1}[0]\right).
\]
Furthermore, since $i'+1,i''\in\Gamma$ we have $d(p_{0},p_{i'+1}),d(p_{i''},p_{k+1})\geq1$
and thus $(p_{0},p_{i'+1},\dots,p_{i''},p_{k+1})$ is very regular.
Therefore we have
\begin{align*}
C_{i',i''} & =I_{\mathrm{up}}^{\mathfrak{m}}(p_{0}[0];p_{i'+1}(0),\dots,p_{i''}(0);p_{k+1}[0])\\
 & =\sum_{i'\leq i\leq i''}I_{\mathrm{up}}^{\mathfrak{m}}(p_{0}[0];p_{i'+1}(0),\dots,p_{i}(0);q[0])\cdot I_{\mathrm{up}}^{\mathfrak{m}}(q[0];p_{i+1}(0),\dots,p_{i''}(0);p_{k+1}[0]).
\end{align*}
It now readily follows that
\begin{align*}
\sum_{\substack{0\leq i'\leq i''\leq k\\
\{i'+1,\dots,i''\}\subset\Gamma
}
}A_{i'}B_{i''}C_{i',i''} & =\sum_{i=0}^{k}I_{\mathrm{up}}^{\mathfrak{m}}(p_{0}[0];p_{1}(0),\dots,p_{i}(0);q[0])\cdot I_{\mathrm{up}}^{\mathfrak{m}}(q[0];p_{i+1}(0),\dots,p_{k}(0);p_{k+1}[0])\\
 & =I_{\mathrm{up}}^{\mathfrak{m}}(p_{0}[0];p_{1}(0),\dots,p_{k}(0);p_{k+1}[0])\\
 & =J_{\mathrm{up}}(p_{0};p_{1},\dots,p_{k};p_{k+1}),
\end{align*}
which proves the claim for the case $q\in U$.
\end{proof}
By Proposition \ref{prop:path_composition_formula_J}, we arrive at
our desired consequence:
\begin{lem}
\label{lem:Ring_homomorphism_J_up} \sloppy The ring homomorphism
from $\mathcal{I}_{\bullet}(T_{z}^{\times}\mathbb{Q}(z))$ to $\mathcal{H}[T]$
which maps $\mathbb{I}(p_{0};p_{1},\ldots,p_{k};p_{k+1})$ to $J_{\mathrm{up}}(p_{0};p_{1},\ldots,p_{k};p_{k+1})$
is well-defined.
\end{lem}

\subsection{\label{subsec:limit-for-any-path}The regularized limit for an arbitrary
path}

Now we define $J_{\gamma}(p_{0};p_{1},\ldots,p_{k};p_{k+1})$ for
a general path $\gamma$. For a path $\gamma$ from $p_{0}(z)$ to
$p_{k+1}(z)$, we define $J_{\gamma}(p_{0};p_{1},\ldots,p_{k};p_{k+1})\in\mathcal{H}[T]$
to be the unique element satisfying the following properties.
\begin{enumerate}
\item If $\gamma$ is the upper path, $J_{\gamma}(p_{0};p_{1},\ldots,p_{k};p_{k+1})\coloneqq J_{\mathrm{up}}(p_{0};p_{1},\ldots,p_{k};p_{k+1})$.
\item If $p_{0}(z)=p_{k+1}(z)$ ($=p(z)$) and $\gamma$ is a path that
encircles the point $\overline{p}(z)$ counter-clockwisely $r\in\mathbb{Z}$
times, then 
\[
J_{\gamma}(p;p_{1},\ldots,p_{k};p)\coloneqq\begin{cases}
\frac{(r\mu)^{k}}{k!} & \text{ if }\overline{p_{1}}=\cdots=\overline{p_{k}}=\overline{p}\\
0 & \text{ otherwise.}
\end{cases}
\]
\item For any paths $\gamma$ from $p_{0}(z)$ to $q(z)$ and $\gamma'$
from $q(z)$ to $p_{k+1}(z)$, 
\[
J_{\gamma\gamma'}(p_{0};p_{1},\ldots,p_{k};p_{k+1})=\sum_{i=0}^{k}J_{\gamma}(p_{0};p_{1},\ldots,p_{i};q)J_{\gamma'}(q;p_{i+1},\ldots,p_{k};p_{k+1})
\]
where $\gamma\gamma'$ denotes the composed path of $\gamma$ and
$\gamma'$ from $p_{0}(z)$ to $p_{k+1}(z)$.
\end{enumerate}
Note that the existence and the uniqueness of such $J_{\gamma}$ follows
from the path composition formula for upper paths (Proposition \ref{prop:path_composition_formula_J})
because any path can be decomposed into upper paths and the paths
that encircle a point (the paths of the type (2) above).

Finally, let us check that the $J_{\gamma}$ thus defined can also
be related with Goncharov's Hopf algebra $\mathcal{I}_{\bullet}(S)$.
Let $X$ be a finite subset of $T_{z}^{\times}\mathbb{Q}(z)$ and
put $\overline{X}\coloneqq\{\overline{p}\mid p\in X\}\subset\mathbb{Q}(z)$.
Fix an element $q$ of $T_{z}^{\times}\mathbb{Q}(z)$ and let $S$
be the set of pair $(p,\gamma)$ of an element of $p\in X$ and a
path $\gamma$ from $q$ to $p$ on $\mathbb{C}\setminus\overline{X}$.
For $\bm{p}=(p,\gamma)\in S$, we denote by $\check{\bm{p}}$ its
first component $p$ and by ${\rm path}(\bm{p})$ its second component
$\gamma$. Then, by Lemma \ref{lem:Ring_homomorphism_J_up}, we obtain
the following proposition.
\begin{prop}
\label{prop:Ring_hom_J_gamma}\sloppy Let $S$ be as above. Then
the ring homomorphism from $\mathcal{I}_{\bullet}(S)$ to $\mathcal{H}[T]$
which maps $\mathbb{I}(\boldsymbol{p_{0}};\boldsymbol{p_{1}},\dots,\boldsymbol{p_{k}};\boldsymbol{p_{k+1}})$
to $J_{\gamma}(\check{\boldsymbol{p_{0}}};\check{\boldsymbol{p_{1}}},\dots,\check{\boldsymbol{p_{k}}};\check{\boldsymbol{p_{k+1}}})$
where $\gamma$ is the path from $\check{\boldsymbol{p_{0}}}$ to
$\check{\boldsymbol{p_{k+1}}}$ given by ${\rm path}(\boldsymbol{p_{0}})^{-1}\cdot{\rm path}(\boldsymbol{p_{k+1}})$
is well-defined.
\end{prop}

\subsection{\label{subsec:inf-coaction-for-J}The infinitesimal coaction of $J(;;)$}

For $p_{0},\dots,p_{k+1}\in T_{z}^{\times}\mathbb{Q}(z)$, we denote
by $J^{\mathfrak{l}}(p_{0};p_{1},\dots,p_{k};p_{k+1})$ the image
of $J_{\gamma}(p_{0};p_{1},\dots,p_{k};p_{k+1})$ in $\mathfrak{L}^{(T)}$
(note that the image does not depend on the choice of $\gamma$).
In this section our goal is to prove the following proposition.
\begin{prop}
\label{prop:Brown_J_gamma}Let $k\geq l>0$, $p_{0},\ldots,p_{k+1}\in T_{z}^{\times}\mathbb{Q}(z)$
and $\gamma$ a path from $p_{0}$ to $p_{k+1}$ on $\mathbb{C}\setminus\{p_{1},\ldots,p_{k}\}$.
Assume that $\#\{p_{1},\dots,p_{k}\}=\#\{\overline{p_{1}},\dots,\overline{p_{k}}\}$.
Then we have
\begin{align*}
 & D_{l}(J_{\gamma}(p_{0};p_{1},\dots,p_{k};p_{k+1}))\\
 & =\sum_{i=0}^{k-l}J^{\mathfrak{l}}(p_{i};p_{i+1},\dots,p_{i+l};p_{i+l+1})\otimes J_{\gamma}(p_{0};p_{1},\dots,p_{i},p_{i+l+1},\dots,p_{k};p_{k+1}).
\end{align*}
\end{prop}

Thanks to Lemma \ref{lem:compatibility_of_Brown_with_prod} and Proposition
\ref{prop:Ring_hom_J_gamma}, Proposition \ref{prop:Brown_J_gamma}
can be reduced to the case where $\gamma$ is the upper path. Therefore,
in what follows, we shall prove Proposition \ref{prop:Brown_J_gamma}
for the case when $\gamma$ is the upper path.
\begin{lem}
\label{lem:i_l_J_eval}Let $p_{0},\dots,p_{k+1}\in T_{z}^{\times}\mathbb{Q}(z)$.
Suppose that $p_{1}(0),\ldots,p_{k}(0)\in\mathbb{Q}$. If neither
$p_{0}(0)=\cdots=p_{k}(0)$ nor $p_{1}(0)=\cdots=p_{k+1}(0)$ holds,
\[
J^{\mathfrak{l}}(p_{0};p_{1},\ldots,p_{k};p_{k+1})=I^{\mathfrak{l}}(p_{0}[0];p_{1}(0),\ldots,p_{k}(0);p_{k+1}[0]).
\]
\end{lem}

\begin{proof}
If $\overline{p_{0}}(0)=\overline{p_{k+1}}(0)\in\mathbb{Q}$, then
there exists $1\leq j\leq k$ such that $\overline{p_{j}}(0)\neq\overline{p_{0}}(0)$
by the assumption of the lemma, and thus 
\[
J_{\mathrm{up}}(p_{0};p_{1},\ldots,p_{k};p_{k+1})=0=I_{\mathrm{up}}^{\mathfrak{m}}(p_{0}[0];p_{1}(0),\ldots,p_{k}(0);p_{k+1}[0]).
\]
by Lemma \ref{lem:J_vanish}. If $\overline{p_{0}}(0)=\overline{p_{k+1}}(0)=\infty$,
then 
\[
J_{\mathrm{up}}(p_{0};p_{1},\ldots,p_{k};p_{k+1})=I_{\mathrm{up}}^{\mathfrak{m}}(p_{0}[0];p_{1}(0),\ldots,p_{k}(0);p_{k+1}[0])
\]
by definition. Thus, we may assume that $\overline{p_{0}}(0)\neq\overline{p_{k+1}}(0)$.
Put 
\[
M\coloneqq\left\{ \left.(q_{1},\dots,q_{l})\right|l\geq0,\:q_{1},\dots,q_{l}\in\{p_{1},\dots,p_{k}\}\right\} 
\]
and
\begin{align*}
\mathcal{P} & =\{(q_{1},\dots,q_{l})\in M\mid l\geq0,p_{0}(0)=q_{1}(0)=\cdots=q_{l}(0)\},\\
\mathcal{P}' & =\{(q_{1},\dots,q_{l})\in M\mid l\geq0,q_{1}(0)=\cdots=q_{l}(0)=p_{k+1}(0)\},\\
\mathcal{Q} & =\{(q_{1},\dots,q_{l})\in M\mid l\geq0,q_{1}(0)\neq p_{0}(0),q_{l}(0)\neq q_{k+1}(0)\}.
\end{align*}
Now that $(p_{1},\dots,p_{k})$ can be expressed as
\[
(p_{1},\dots,p_{k})=\sum_{j}c_{j}P_{j}\shuffle P_{j}'\shuffle Q_{j}
\]
with $c_{j}\in\mathbb{Z}$, $P_{j}\in\mathcal{P},P_{j}'\in\mathcal{P}',Q_{j}\in\mathcal{Q}$
(N.B. the length of $P_{j}$ nor $P_{j}'$ never happens to be $k$
by the assumption of the lemma). Then by definition
\[
J_{\mathrm{up}}(p_{0};p_{1},\dots,p_{k};p_{k+1})=\sum_{j}c_{j}J_{\mathrm{up}}(p_{0};P_{j};p_{k+1})J_{\mathrm{up}}(p_{0};P_{j}';p_{k+1})J_{\mathrm{up}}(p_{0};Q_{j};p_{k+1}),
\]
and thus
\begin{align*}
J^{\mathfrak{l}}(p_{0};p_{1},\dots,p_{k};p_{k+1}) & =\sum_{\mathrm{len}(Q_{j})=k}c_{j}J^{\mathfrak{l}}(p_{0};Q_{j};p_{k+1})\\
 & =\sum_{\mathrm{len}(Q_{j})=k}c_{j}I^{\mathfrak{l}}(p_{0}[0];Q_{j}(0);p_{k+1}[0]).
\end{align*}
On the other hand, since
\begin{multline*}
I_{\mathrm{up}}^{\mathfrak{m}}(p_{0}[0];p_{1}(0),\ldots,p_{k}(0);p_{k+1}[0])\\
=\sum_{j}c_{j}I_{\mathrm{up}}^{\mathfrak{m}}(p_{0}[0];P_{j}(0);p_{k+1}[0])\cdot I_{\mathrm{up}}^{\mathfrak{m}}(p_{0}[0];P_{j}'(0);p_{k+1}[0])\cdot I_{\mathrm{up}}^{\mathfrak{m}}(p_{0}[0];Q_{j}(0);p_{k+1}[0]),
\end{multline*}
we have
\[
I^{\mathfrak{l}}(p_{0}[0];p_{1}(0),\ldots,p_{k}(0);p_{k+1}[0])=\sum_{\mathrm{len}(Q_{j})=k}c_{j}I^{\mathfrak{l}}(p_{0}[0];Q_{j}(0);p_{k+1}[0]).
\]
Thus the lemma is proved.
\end{proof}
\begin{lem}
\label{lem:motiv_sum}Let $k\geq0$ and $b_{0},\dots,b_{k+1}\in\widehat{T^{\times}}\mathbb{P}^{1}(\mathbb{Q})$.
Assume that $\overline{b_{0}}=\overline{b_{k+1}}=\infty$ and $\overline{b_{1}},\dots,\overline{b_{k}}\in\mathbb{Q}$,
and $\#\{b_{1},\dots,b_{k}\}=\#\{\overline{b_{1}},\dots,\overline{b_{k}}\}$.
Then for $0<l\leq k$, we have
\[
\sum_{i=0}^{k-l}I^{\mathfrak{l}}(b_{i};b_{i+1},\dots,b_{i+l};b_{i+l+1})=I^{\mathfrak{l}}(b_{0};\{x\}^{l};b_{k+1})
\]
where $x\in\mathbb{Q}$ is any element.
\end{lem}

\begin{proof}
Let $\gamma$ be any path from $b_{0}$ to $b_{k+1}$ contained in
a sufficiently small neighborhood of $\infty$ such that $I_{\gamma}^{\mathfrak{m}}(b_{0};x;b_{k+1})\neq0$.
Noting
\[
I_{\gamma}^{\mathfrak{m}}(b_{0};b_{1},\dots,b_{k};b_{k+1})=\frac{I_{\gamma}^{\mathfrak{m}}(b_{0};\{x\};b_{k+1})^{k}}{k!},
\]
its coaction is easily computed as 
\begin{align*}
\Delta I_{\gamma}^{\mathfrak{m}}(b_{0};b_{1},\dots,b_{k};b_{k+1}) & =\frac{\left(\Delta I_{\gamma}^{\mathfrak{m}}(b_{0};\{x\};b_{k+1})\right)^{k}}{k!}\\
 & =\frac{\left(1\otimes I_{\gamma}^{\mathfrak{m}}(b_{0};x;b_{k+1})+I^{\mathfrak{l}}(b_{0};x;b_{k+1})\otimes1\right)^{k}}{k!}\\
 & =\sum_{l+l'=k}\frac{I^{\mathfrak{l}}(b_{0};x;b_{k+1})^{l}}{l!}\otimes\frac{I_{\gamma}^{\mathfrak{m}}(b_{0};x;b_{k+1})^{l'}}{l'!}\\
 & =\sum_{l+l'=k}I^{\mathfrak{l}}(b_{0};\{x\}^{l};b_{k+1})\otimes I_{\gamma}^{\mathfrak{m}}(b_{0};\{x\}^{l'};b_{k+1}),
\end{align*}
and thus we have
\begin{equation}
D_{l}(I_{\gamma}^{\mathfrak{m}}(b_{0};b_{1},\dots,b_{k};b_{k+1}))=I^{\mathfrak{l}}(b_{0};\{x\}^{l};b_{k+1})\otimes I_{\gamma}^{\mathfrak{m}}(b_{0};\{x\}^{k-l};b_{k+1}).\label{eq:motiv_ev1}
\end{equation}
On the other hand, by Proposition \ref{prop:coproduct_extended},
we have
\begin{equation}
D_{l}(I_{\gamma}^{\mathfrak{m}}(b_{0};b_{1},\dots,b_{k};b_{k+1}))=\sum_{i=0}^{k-l}I^{\mathfrak{l}}(b_{i};b_{i+1},\dots,b_{i+l};b_{i+l+1})\otimes I_{\gamma}^{\mathfrak{m}}(b_{0};\{x\}^{k-l};b_{k+1}).\label{eq:motiv_ev2}
\end{equation}
The lemma now readily follows by comparing (\ref{eq:motiv_ev1}) and
(\ref{eq:motiv_ev2}).
\end{proof}
\begin{lem}
\label{lem:motiv_sum2}Let $k\geq l\geq1$, $x\in\mathbb{Q}$, and
$p_{0},\dots,p_{k+1}\in T_{z}^{\times}\mathbb{Q}(z)$. Assume that
$p_{0}(0)\neq x=p_{1}(0)=\cdots=p_{k}(0)\neq p_{k+1}(0)$ and $\#\{p_{1},\dots,p_{k}\}=\#\{\overline{p_{1}},\dots,\overline{p_{k}}\}$.
Then,
\[
\sum_{i=0}^{k-l}J^{\mathfrak{l}}(p_{i};p_{i+1},\ldots,p_{i+l};p_{i+l+1})=I^{\mathfrak{l}}(p_{0}[0];\{x\}^{l};p_{k+1}[0]).
\]
\end{lem}

\begin{proof}
We fix $l$ and prove the claim by induction on $k$. If $p_{1}=\cdots=p_{k}$,
the claim follows by definition (Definitions \ref{def:J_for_the_differnet_endpoints_general}
and \ref{def:J_for_the_same_endpoints}). Suppose that $\sharp\left\{ p_{1},\ldots,p_{k}\right\} >1$.
Note that this implies $\#\{\overline{p_{1}},\dots,\overline{p_{k}}\}>1$
by the assumption of the lemma. Let $m\geq1$ be the maximal integer
such that 
\[
\overline{p_{1}}\equiv\cdots\equiv\overline{p_{k}}\pmod{z^{m}}
\]
and $q\in\mathbb{Q}[z]$ any polynomial such that $q\equiv\overline{p_{1}}\equiv\cdots\equiv\overline{p_{k}}\pmod{z^{m}}.$
Then, by definition, we have
\[
\sum_{i=0}^{k-l}J^{\mathfrak{l}}(p_{i};p_{i+1},\ldots,p_{i+l};p_{i+l+1})=\sum_{i=0}^{k-l}f_{i}
\]
where
\[
f_{i}:=J^{\mathfrak{l}}(p_{i}';p_{i+1}',\ldots,p_{i+l}';p_{i+l+1}'),
\]
$p_{i}'=(p_{i}-q)/z^{m}$. We decompose $\{0\leq i\leq k-l\}$ as
\[
\{0\leq i\leq k-l\}=W\sqcup X\sqcup X'\sqcup Y
\]
where
\begin{align*}
X & =\{i\mid p_{i}'(0)\neq p_{i+1}'(0)=\cdots=p_{i+l}'(0)=p_{i+l+1}'(0)\}\\
X' & =\{i\mid p_{i}'(0)=p_{i+1}'(0)=\cdots=p_{i+l}'(0)\neq p_{i+l+1}'(0)\}\\
Y & =\{i\mid p_{i}'(0)=p_{i+1}'(0)=\cdots=p_{i+l}'(0)=p_{i+l+1}'(0)\}\\
W & =\{0\leq i\leq k-l\}\setminus(X\sqcup X'\sqcup Y).
\end{align*}
Define $\lambda:\{0\leq i\leq k-l\}\to\{X,X',Y,W\}$ by setting $\lambda(i)$
to be the set where $i$ belongs to. By definition, $(\lambda(0),\dots,\lambda(k-l))$
takes the form
\[
(\{W\}^{s_{1}},X,\{Y\}^{s_{2}},X',\dots,\{W\}^{s_{2m-1}},X,\{Y\}^{s_{2m}},X',\{W\}^{s_{2m+1}})
\]
with some $m\geq0$ and $s_{1},\dots,s_{2m+1}\geq0$. On one hand,
for $i\in W$, we have
\[
f_{i}=I^{\mathfrak{l}}(p_{i}'[0];p_{i+1}'(0),\dots,p_{i+l}'(0);p_{i+l+1}'[0])
\]
by Lemma \ref{lem:i_l_J_eval}. On the other hand, for $i<i'$ such
that $(\lambda(i),\dots,\lambda(i'))=(X,\{Y\}^{s},X')$, we have,
by induction hypothesis,
\begin{align}
\sum_{j=i}^{i'}f_{j} & =I^{\mathfrak{l}}(p_{i}'[0];\{y\}^{l};p_{i'+l+1}'[0])\label{eq:sum_f_i}
\end{align}
where $y\coloneqq p_{i+1}'[0]=p_{i+2}'[0]=\cdots=p_{i'+l}'[0]$\footnote{It should be noted that if $(\lambda(0),\dots,\lambda(k-l))$ is of
the form $(X,\{Y\}^{s},X')$, the induction hypothesis cannot be applied,
but such a case never occurs since $\sharp\left\{ \overline{p_{1}'},\ldots,\overline{p_{k}'}\right\} >1$
by definition.}. Then, by the path composition formula, (\ref{eq:sum_f_i}) equals
\begin{align*}
I^{\mathfrak{l}}(p_{i}'[0];\{y\}^{l};p_{i'+l+1}'[0])= & I^{\mathfrak{l}}(p_{i}'[0];\{y\}^{l};y)+I^{\mathfrak{l}}(y;\{y\}^{l};p_{i'+l+1}'[0])\\
= & \sum_{j=i}^{i'}I^{\mathfrak{l}}(p_{j}'[0];p_{j+1}'(0),\dots,p_{j+l}'(0);p_{j+l+1}'[0])
\end{align*}
where the second equality is because $I^{\mathfrak{l}}(p_{j}'[0];p_{j+1}'(0),\dots,p_{j+l}'(0);p_{j+l+1}'[0])=0$
for $i<j<i'$. Therefore, we have
\[
\sum_{i=0}^{k-l}f_{i}=\sum_{i=0}^{k-l}I^{\mathfrak{l}}(p_{i}'[0];p_{i+1}'(0),\dots,p_{i+l}'(0);p_{i+l+1}'[0]).
\]
Now by Lemma \ref{lem:motiv_sum}, it follows that
\begin{align*}
\sum_{i=0}^{k-l}f_{i} & =I^{\mathfrak{l}}(p_{0}'[0];\{x'\}^{l};p_{k+1}'[0])\ \ \ \ (x'\text{ is any element of }\mathbb{Q})\\
 & =J^{\mathfrak{l}}(p_{0};\{q+z^{m}x'\}^{l};p_{k+1})\\
 & =I^{\mathfrak{l}}(p_{0}[0];\{x\}^{l};p_{k+1}[0]).
\end{align*}
\end{proof}
\begin{lem}
\label{lem:motiv_copro_pre1}Let $k\geq l>0$ and $p_{0},\ldots,p_{k+1}\in T_{z}^{\times}\mathbb{Q}(z)$.
Assume that $\#\{p_{1},\dots,p_{k}\}=\#\{\overline{p_{1}},\dots,\overline{p_{k}}\}$.
Then we have
\begin{align}
 & D_{l}(J_{\mathrm{up}}(p_{0};p_{1},\dots,p_{k};p_{k+1}))\label{eq:motiv_ev3}\\
 & =\sum_{i=0}^{k-l}J^{\mathfrak{l}}(p_{i};p_{i+1},\dots,p_{i+l};p_{i+l+1})\otimes J_{\mathrm{up}}(p_{0};p_{1},\dots,p_{i},p_{i+l+1},\dots,p_{k};p_{k+1}).\nonumber 
\end{align}
\end{lem}

\begin{proof}
If is enough to only consider the case where $(p_{0},\dots,p_{k+1})$
is very regular sequence. If there exists $s\in\{1,\dots,k\}$ such
that $p_{s}(0)=\infty$ then $p_{0}(0)\neq\infty$, $p_{k+1}(0)\neq\infty$
by the definition of very regular sequence and $J_{\mathrm{up}}(p_{0};p_{1},\dots,p_{k};p_{k+1})=0$,
and
\[
J_{\mathrm{up}}(p_{i};p_{i+1},\dots,p_{i+l};p_{i+l+1})\otimes J_{\mathrm{up}}(p_{0};p_{1},\dots,p_{i},p_{i+l+1},\dots,p_{k};p_{k+1})=0
\]
for $0\leq i\leq k-l$, and thus (\ref{eq:motiv_ev3}) holds. Thus
it is enough to only consider the case where $p_{i}(0)\neq\infty$
for all $i$. First we have
\begin{align}
 & \sum_{i=0}^{k-l}J^{\mathfrak{l}}(p_{i};p_{i+1},\dots,p_{i+l};p_{i+l+1})\otimes J_{\mathrm{up}}(p_{0};p_{1},\dots,p_{i},p_{i+l+1},\dots,p_{k};p_{k+1})\nonumber \\
 & =\sum_{i=0}^{k-l}J^{\mathfrak{l}}(p_{i};p_{i+1},\dots,p_{i+l};p_{i+l+1})\otimes I_{\mathrm{up}}^{\mathfrak{m}}(p_{0}[0];p_{1}(0),\dots,p_{i}(0),p_{i+l+1}(0),\dots,p_{k}(0);p_{k+1}[0])\label{eq:infinitesimal_coation_J}
\end{align}
since
\[
J_{\mathrm{up}}(p_{0};p_{1},\dots,p_{i},p_{i+l+1},\dots,p_{k};p_{k+1})=I_{\mathrm{up}}^{\mathfrak{m}}(p_{0}[0];p_{1}(0),\dots,p_{i}(0),p_{i+l+1}(0),\dots,p_{k}(0);p_{k+1}[0])
\]
except for the case 
\[
\left(i=0\land\overline{p_{0}}(0)=\overline{p_{l+1}}(0)\right)\lor\left(i=k-l\land\overline{p_{k-l}}(0)=\overline{p_{k+1}}(0)\right),
\]
and for such exceptional cases, we have
\[
J_{\mathrm{up}}(p_{i};p_{i+1},\dots,p_{i+l};p_{i+l+1})=0.
\]

Now as before we decompose $\{0\leq i\leq k-l\}$ as
\[
\{0\leq i\leq k-l\}=W\sqcup X\sqcup X'\sqcup Y
\]
where
\begin{align*}
X & =\{i\mid p_{i}(0)\neq p_{i+1}(0)=\cdots=p_{i+l}(0)=p_{i+l+1}(0)\}\\
X' & =\{i\mid p_{i}(0)=p_{i+1}(0)=\cdots=p_{i+l}(0)\neq p_{i+l+1}(0)\}\\
Y & =\{i\mid p_{i}(0)=p_{i+1}(0)=\cdots=p_{i+l}(0)=p_{i+l+1}(0)\}\\
W & =\{0\leq i\leq k-l\}\setminus(X\sqcup X'\sqcup Y).
\end{align*}
Define $\lambda:\{0\leq i\leq k-l\}\to\{X,X',Y,W\}$ by setting $\lambda(i)$
to be the set where $i$ belongs to. By definition, $(\lambda(0),\dots,\lambda(k-l))$
takes the form
\[
(\{W\}^{s_{1}},X,\{Y\}^{s_{2}},X',\dots,\{W\}^{s_{2m-1}},X,\{Y\}^{s_{2m}},X',\{W\}^{s_{2m+1}})
\]
with some $m\geq0$ and $s_{1},\dots,s_{2m+1}\geq0$. On one hand,
for $i\in W$, we have
\[
J^{\mathfrak{l}}(p_{i};p_{i+1},\dots,p_{i+l};p_{i+l+1})=I^{\mathfrak{l}}(p_{i}[0];p_{i+1}(0),\dots,p_{i+l}(0);p_{i+l+1}[0]).
\]
by Lemma \ref{lem:i_l_J_eval}. On the other hand, for $j<j'$ such
that $(\lambda(j),\dots,\lambda(j'))=(X,\{Y\}^{s},X')$, the sequence
\[
\sigma\coloneqq(p_{1}(0),\dots,p_{i}(0),p_{i+l+1}(0),\dots,p_{k}(0))
\]
does not depend on the choice of $i\in\{j,j+1,\dots,j'\}$ since $p_{j+1}(0)=p_{j+2}(0)=\cdots=p_{j'+l}(0)$.
Thus, the partial sum
\[
\sum_{i=j}^{j'}J^{\mathfrak{l}}(p_{i};p_{i+1},\dots,p_{i+l};p_{i+l+1})\otimes I_{\mathrm{up}}^{\mathfrak{m}}(p_{0}[0];p_{1}(0),\dots,p_{i}(0),p_{i+l+1}(0),\dots,p_{k}(0);p_{k+1}[0])
\]
of equation (\ref{eq:infinitesimal_coation_J}) equals
\[
\left(\sum_{i=j}^{j'}J^{\mathfrak{l}}(p_{i};p_{i+1},\dots,p_{i+l};p_{i+l+1})\right)\otimes I_{\mathrm{up}}^{\mathfrak{m}}(p_{0}[0];\sigma;p_{k+1}[0]).
\]
By Lemma \ref{lem:motiv_sum2},
\begin{align*}
\sum_{i=j}^{j'}J^{\mathfrak{l}}(p_{i};p_{i+1},\dots,p_{i+l};p_{i+l+1}) & =I^{\mathfrak{l}}(p_{j}[0];\{p_{j+1}(0)\}^{l};p_{j'+l+1}[0])\\
 & =\sum_{i=j}^{j'}I^{\mathfrak{l}}(p_{i}[0];p_{i+1}(0),\dots,p_{i+l}(0);p_{i+l+1}[0])
\end{align*}
where the second equality follows from the path composition formula.
Putting together, we arrive at
\begin{align*}
 & \sum_{i=0}^{k-l}J^{\mathfrak{l}}(p_{i};p_{i+1},\dots,p_{i+l};p_{i+l+1})\otimes J_{\mathrm{up}}(p_{0};p_{1},\dots,p_{i},p_{i+l+1},\dots,p_{k};p_{k+1})\\
 & =\sum_{i=0}^{k-l}I^{\mathfrak{l}}(p_{i}[0];p_{i+1}(0),\dots,p_{i+l}(0);p_{i+l+1}[0])\otimes I_{\mathrm{up}}^{\mathfrak{m}}(p_{0}[0];p_{1}(0),\dots,p_{i}(0),p_{i+l+1}(0),\dots,p_{k}(0);p_{k+1}[0])\\
 & =D_{l}(I_{\mathrm{up}}^{\mathfrak{m}}(p_{0}[0];p_{1}(0),\dots,p_{k}(0);p_{k+1}[0]))\\
 & =D_{l}(J_{\mathrm{up}}((p_{0};p_{1},\dots,p_{k};p_{k+1})).
\end{align*}
This completes the proof.
\end{proof}
As explained at the beginning of this section, we can now conclude
that Proposition \ref{prop:Brown_J_gamma} is true by Lemma \ref{lem:motiv_copro_pre1}.

\section{\label{sec:Motivicity-of-confluence}Motivicity of confluence relations}

\subsection{Basic notations}

Throughout Section \ref{sec:Motivicity-of-confluence}, we fix a finite
subsets $\mathcal{P}\subset\mathbb{Q}[z]$ and $S\subset\mathbb{Q}$
such that for all distinct pair $p(z),q(z)\in\mathcal{P}$, all zeros
of $p(z)-q(z)$ lie in $S$. We also fix a base point $z_{0}\in Y\coloneqq\mathbb{C}\setminus S$
and let $\tilde{Y}$ be a universal covering space of $Y$ with the
base point $z_{0}$. For each $a\in S$, we fix $v(a)\in\mathbb{Q}^{\times}$,
and write $\widehat{a}$ for the tangential base point $\overrightarrow{v(a)}_{a}$.
Similarly, for each $p=p(z)\in\mathcal{P}$, we fix $v'(p)\in\mathbb{Q}^{\times}$,
and for $a\in\mathbb{C}$ we write $\widehat{p}(a)$ for the tangential
base point $\overrightarrow{v'(p)}_{p(a)}$. Note that $\tilde{Y}$
is defined as the set of pairs $(a,\gamma)$ where $a\in Y$ and $\gamma$
is a homotopy class of paths from $z_{0}$ to $a$ on $Y$. We denote
by $\tilde{S}$ the set of pairs $(a,\gamma)$ where $a\in S$ and
$\gamma$ is a homotopy class of paths from $z_{0}$ to $\widehat{a}$
on $Y$. For $a\in\mathbb{C}$, we put 
\[
X_{a}\coloneqq\mathbb{C}\setminus\{p(a)\mid p\in\mathcal{P}\}.
\]
Furthermore, for $x,y\in\tilde{S}$ and $a_{1},\dots,a_{k}\in S$,
we denote by
\[
I^{\mathfrak{m}}(x;a_{1},\dots,a_{k};y)
\]
the motivic iterated integral
\[
I_{\gamma}^{\mathfrak{m}}(\widehat{{\rm pr}(x)};a_{1},\dots,a_{k};\widehat{{\rm pr}(y)})\in\mathcal{H}
\]
where ${\rm pr}:\tilde{Y}\sqcup\tilde{S}\to Y\sqcup S$ is the natural
projection and $\gamma\in\pi_{1}^{{\rm top}}(Y;\widehat{{\rm pr}(x)},\widehat{{\rm pr}(y)})$
is the natural path determined by $x$ and $y$.
\begin{defn}
For $p=p(z),q=q(z)\in\mathcal{P}$, we put 
\[
\pi(p,q)\coloneqq\pi_{1}^{{\rm top}}(X_{z_{0}};\widehat{p}(z_{0}),\widehat{q}(z_{0})).
\]
\end{defn}

For each $\gamma\in\pi(p,q)$ and $\widetilde{a}=(a,\gamma')\in\tilde{Y}$,
we define 
\[
\gamma(\widetilde{a})\in\pi_{1}^{{\rm top}}(X_{a};\widehat{p}(a),\widehat{q}(a))
\]
to be the continuous deformation of $\gamma$ along $\gamma'$, which
is well-defined since for $a\in Y$ the map $\mathcal{P}\rightarrow\mathbb{C}$
given by $p\mapsto p(a)$ is injective.

\begin{defn}
We denote by $\mathscr{B}$ the commutative $\mathbb{Q}$-algebra
generated by the formal symbols
\[
\formalit{\gamma}{p_{0}}{p_{1},\dots,p_{k}}{p_{k+1}}
\]
where $k\geq0$, $p_{0},\dots,p_{k+1}\in\mathcal{P}$, and $\gamma\in\pi(p_{0},p_{k+1})$
with the relations (see also \cite[Section 2.1]{GonSym}):
\begin{enumerate}
\item the unit:
\[
\mathbb{I}_{\gamma}(p;q):=\formalit{\gamma}p{\emptyset}q=1;
\]
\item the shuffle product formula:
\[
\formalit{\gamma}p{p_{1},\dots,p_{r+1}}q\formalit{\gamma}p{p_{r+1},\dots,p_{s}}q=\formalit{\gamma}p{(p_{1},\dots,p_{r})\shuffle(p_{r+1},\dots,p_{s})}q;
\]
\item the path composition formula:
\[
\formalit{\gamma_{1}\gamma_{2}}{p_{0}}{p_{1},\dots,p_{k}}{p_{k+1}}=\sum_{i=0}^{k}\formalit{\gamma_{1}}{p_{0}}{p_{1},\dots p_{i}}q\formalit{\gamma_{1}}q{p_{i+1},\dots,p_{k}}{p_{k+1}};
\]
\item the trivial path: For $k>0$ and a trivial path $\gamma\in\pi(q,q)$,
\[
\formalit{\gamma}q{p_{1},\dots,p_{k}}q=0.
\]
\end{enumerate}
\end{defn}

Then we can view $\mathscr{B}$ as a graded ring $\bigoplus_{k=0}^{\infty}\mathcal{\mathscr{B}}_{k}$
by setting the degree of $\formalit{\gamma}{p_{0}}{p_{1},\dots,p_{k}}{p_{k+1}}$
as $k$.

\begin{defn}
We define $\mathcal{\mathscr{B}}^{\mathfrak{a}}:=\mathcal{\mathscr{B}}/I$
where $I$ is the ideal of $\mathcal{\mathscr{B}}$ generated by
\[
\{\formalit{\gamma}q{p_{1},\dots,p_{k}}q\mid k>0,\gamma\in\pi(q,q)\}.
\]
\end{defn}

Notice that the image of $\formalit{\gamma}{p_{0}}{p_{1},\dots,p_{k}}{p_{k+1}}$
in $\mathcal{\mathscr{B}}^{\mathfrak{a}}$ does not depend on the
choice of $\gamma$. Therefore, $\mathcal{\mathscr{B}}^{\mathfrak{a}}$
is naturally identified with the Goncharov's Hopf algebra $I_{\bullet}(\mathcal{P})$
as a commutative ring by the obvious correspondence $\formalit{\gamma}{p_{0}}{p_{1},\dots,p_{k}}{p_{k+1}}\mapsto\mathbb{I}(p_{0};p_{1},\dots,p_{k};p_{k+1})$.
Let $\bigoplus_{k=0}^{\infty}\mathcal{\mathscr{B}}_{k}^{\mathfrak{a}}:=\mathscr{B}^{\mathfrak{a}}$
be the grading of $\mathscr{B}^{\mathfrak{a}}$ induced by that of
$\mathscr{B}$.

\begin{defn}
\label{def:ev}For $a\in\tilde{S}$ we define a ring homomorphism
$\mathrm{ev}_{a}^{\mathfrak{m},(T)}:\mathscr{B}\to\mathcal{H}[T]$
by 
\[
\mathrm{ev}_{a}^{\mathfrak{m},(T)}\left(\formalit{\gamma}{p_{0}}{p_{1},\dots,p_{k}}{p_{k+1}}\right)\coloneqq J_{\gamma'}(p_{0}';p_{1}',\ldots,p_{k}';p_{k+1}')
\]
where $\gamma'(z)\coloneqq\gamma(a+v(\mathrm{pr}(a))z)$ and $p_{j}'(z)\coloneqq p_{j}(\mathrm{pr}(a)+v(\mathrm{pr}(a))z)$
for $0\leq j\leq k+1$, and $\mathrm{ev}_{a}^{\mathfrak{m}}$ as the
composed map $\mathscr{B}\xrightarrow{\mathrm{ev}_{a}^{\mathfrak{m},(T)}}\mathcal{H}[T]\xrightarrow{T=0}\mathcal{H}$.
Here, $a+v(\mathrm{pr}(a))z$ denotes the element of $\widetilde{Y}$
(whose projection on $Y$ is $\mathrm{pr}(a)+v(\mathrm{pr}(a))z$)
defined in a natural way for a sufficiently small $z>0$. 
\end{defn}

The well-definedness of $\mathrm{ev}_{a}^{\mathfrak{m},(T)}$ is clear
by Proposition \ref{prop:Ring_hom_J_gamma}. Note that the composite
maps $\mathscr{B}_{>0}\xrightarrow{\mathrm{ev}_{a}^{\mathfrak{m},(T)}}\mathcal{H}[T]\rightarrow\mathfrak{L}^{(T)}$
(resp. $\mathscr{B}_{>0}\xrightarrow{\mathrm{ev}_{a}^{\mathfrak{m}}}\mathcal{H}\rightarrow\mathfrak{L}$)
factors through $\mathscr{B}_{>0}^{\mathfrak{a}}$ and we thus define
$\mathrm{ev}_{a}^{\mathfrak{l},(T)}:\mathscr{B}_{>0}^{\mathfrak{a}}\rightarrow\mathfrak{L}^{(T)}$
(resp. $\mathrm{ev}_{a}^{\mathfrak{l}}:\mathscr{B}_{>0}^{\mathfrak{a}}\rightarrow\mathfrak{L}$)
as the induced map. One can check that $\mathrm{ev}_{a}^{\mathfrak{l},(T)}$
depends only on $\mathrm{pr}(a)\in S$ and does not depend on the
path to $\widehat{\mathrm{pr}(a)}$ and so we sometimes abuse the
notation $\mathrm{ev}_{a}^{\mathfrak{l},(T)}$ for $a\in S$ if there
is no risk of confusion.

\subsection{Key lemmas}

As corollaries of Proposition \ref{prop:Brown_J_gamma}, we can now
show the following key lemmas which will take crucial roles in the
proof of the motivicity of confluence relations.
\begin{lem}
\label{lem:motiv_copro2}For $p_{0},\dots,,p_{k+1}\in\mathcal{P}$,
$a\in\widetilde{S}$, $0<l\leq k$ and a path $\gamma\in\pi(p_{0},p_{k+1})$,
\[
D_{l}(\mathrm{ev}_{a}^{\mathfrak{m}}\mathbb{I}_{\gamma}(p_{0};p_{1},\dots,p_{k};p_{k+1}))=\sum_{i=0}^{k-l}\mathrm{ev}_{a}^{\mathfrak{l}}\mathbb{I}(p_{i};p_{i+1},\dots,p_{i+l};p_{i+l+1})\otimes\mathrm{ev}_{a}^{\mathfrak{m}}\mathbb{I}_{\gamma}(p_{0};p_{1},\dots,p_{i},p_{i+l+1},\dots,p_{k};p_{k+1}).
\]
\end{lem}

\begin{proof}
This is an immediate consequence of Proposition \ref{prop:Brown_J_gamma}.
\end{proof}
\begin{lem}
\label{lem:motiv_copro3}For $p_{0},\dots,,p_{k+1}\in\mathcal{P}$
and $a\in S$, we have
\begin{align*}
 & \sum_{i=1}^{k}\left({\rm ord}_{z-a}(p_{i}-p_{i+1})-{\rm ord}_{z-a}(p_{i}-p_{i-1})\right){\rm ev}_{a}^{\mathfrak{l}}\mathbb{I}(p_{0};p_{1},\dots,p_{i-1},p_{i+1},\dots,p_{k};p_{k+1})\\
 & =\left({\rm ord}_{z-a}(p_{1}-p_{k+1})-{\rm ord}_{z-a}(p_{1}-p_{0})\right){\rm ev}_{a}^{\mathfrak{l}}\mathbb{I}(p_{1};p_{2},\dots,p_{k};p_{k+1})\\
 & \ +\left({\rm ord}_{z-a}(p_{k}-p_{k+1})-{\rm ord}_{z-a}(p_{k}-p_{0})\right){\rm ev}_{a}^{\mathfrak{l}}\mathbb{I}(p_{0};p_{1},\dots,p_{k-1};p_{k})
\end{align*}
where we define ${\rm ord}_{z-a}(0)\coloneqq0$ as before.
\end{lem}

\begin{proof}
Fix a lift $\widetilde{a}\in\widetilde{S}$ of $a\in S$ and a path
$\gamma\in\pi(p_{0},p_{k+1})$ and put
\[
\mathrm{ev}_{\widetilde{a}}^{\mathfrak{m},(T)}\mathbb{I}_{\gamma}(p_{0};p_{1},\dots,p_{k};p_{k+1})=\sum_{j=0}^{k}s_{j}T^{j}\in\mathcal{H}[T]
\]
where $s_{0},\dots,s_{k}\in\mathcal{H}$ with $\deg s_{j}=k-j$. By
the Leibniz rule $D_{l}(xy)=D_{l}(x)\cdot(1\otimes y)+D_{l}(y)\cdot(1\otimes x)$,
we have
\begin{align}
D_{1}(\mathrm{ev}_{\widetilde{a}}^{\mathfrak{m},(T)}\mathbb{I}_{\gamma}(p_{0};p_{1},\dots,p_{k};p_{k+1})) & =\sum_{j=0}^{k}\left(D_{1}(s_{j})\cdot\left(1\otimes T^{j}\right)+D_{1}(T^{j})\cdot\left(1\otimes s_{j}\right)\right)\nonumber \\
 & =\sum_{j=0}^{k}D_{1}(s_{j})\cdot\left(1\otimes T^{j}\right)+\sum_{j=1}^{k}jT\otimes s_{j}T^{j-1}\label{eq:motiv_ev4}
\end{align}
and
\begin{align}
D_{k-1}(\mathrm{ev}_{\widetilde{a}}^{\mathfrak{m},(T)}\mathbb{I}_{\gamma}(p_{0};p_{1},\dots,p_{k};p_{k+1})) & =\sum_{j=0}^{k}\left(D_{k-1}(s_{j})\cdot\left(1\otimes T^{j}\right)+D_{k-1}(T^{j})\cdot\left(1\otimes s_{j}\right)\right)\nonumber \\
 & =\sum_{j=0}^{1}D_{k-1}(s_{j})\cdot\left(1\otimes T^{j}\right)+\sum_{j=k-1}^{k}D_{k-1}(T^{j})\cdot\left(1\otimes s_{j}\right)\nonumber \\
 & =\left(D_{k-1}(s_{0})+s_{1}\otimes T+\delta_{1,k}T\otimes s_{1}\right).\label{eq:motiv_ev4a}
\end{align}
since $D_{k-1}(s_{1})=s_{1}\otimes1$ and $D_{l}(T^{j})=j\delta_{1,l}T\otimes T^{j-1}$.
On the other hand, by Proposition \ref{prop:Brown_J_gamma}, we have
\begin{align}
D_{1}(\mathrm{ev}_{\widetilde{a}}^{\mathfrak{m},(T)}\mathbb{I}_{\gamma}(p_{0};p_{1},\dots,p_{k};p_{k+1})) & =\sum_{i=1}^{k}\mathrm{ev}_{a}^{\mathfrak{l},(T)}\mathbb{I}(p_{i-1};p_{i};p_{i+1})\otimes\mathrm{ev}_{\widetilde{a}}^{\mathfrak{m},(T)}\mathbb{I}_{\gamma}(p_{0};p_{1},\dots,p_{i-1},p_{i+1},\dots,p_{k};p_{k+1}).\label{eq:motiv_ev5}
\end{align}
By comparing the coefficient of $T\otimes1$ in (\ref{eq:motiv_ev4})
and (\ref{eq:motiv_ev5}), we have
\begin{equation}
s_{1}=\sum_{i=1}^{k}\left({\rm ord}_{z-a}(p_{i}-p_{i+1})-{\rm ord}_{z-a}(p_{i}-p_{i-1})\right)\mathrm{ev}_{\widetilde{a}}^{\mathfrak{m}}\mathbb{I}(p_{0};p_{1},\dots,p_{i-1},p_{i+1},\dots,p_{k};p_{k+1})\label{eq:motiv_ev6}
\end{equation}
since $\mathrm{ev}_{a}^{\mathfrak{l},(T)}\mathbb{I}(p_{i-1};p_{i};p_{i+1})=\left({\rm ord}_{z-a}(p_{i}-p_{i+1})-{\rm ord}_{z-a}(p_{i}-p_{i-1})\right)T+(\text{Const. term})$.

Similarly, by Proposition \ref{prop:Brown_J_gamma}, we have
\begin{align}
D_{k-1}(\mathrm{ev}_{\widetilde{a}}^{\mathfrak{m},(T)}\mathbb{I}_{\gamma}(p_{0};p_{1},\dots,p_{k};p_{k+1})) & =\mathrm{ev}_{a}^{\mathfrak{l},(T)}\mathbb{I}(p_{1};p_{2},\dots,p_{k};p_{k+1})\otimes\mathrm{ev}_{\widetilde{a}}^{\mathfrak{m},(T)}\mathbb{I}_{\gamma}(p_{0};p_{1};p_{k+1})\label{eq:motiv_ev5a}\\
 & \ +\mathrm{ev}_{a}^{\mathfrak{l},(T)}\mathbb{I}(p_{0};p_{1},\dots,p_{k-1};p_{k})\otimes\mathrm{ev}_{\widetilde{a}}^{\mathfrak{m},(T)}\mathbb{I}_{\gamma}(p_{0};p_{k};p_{k+1}).\nonumber 
\end{align}
By comparing the coefficient of $1\otimes T$ in (\ref{eq:motiv_ev4a})
and (\ref{eq:motiv_ev5a}), we have
\begin{align}
(\text{Image of }s_{1}\text{ in }\mathfrak{L}) & =\left({\rm ord}_{z-a}(p_{1}-p_{k+1})-{\rm ord}_{z-a}(p_{1}-p_{0})\right)\mathrm{ev}_{a}^{\mathfrak{l}}\mathbb{I}(p_{1};p_{2},\dots,p_{k};p_{k+1})\label{eq:motiv_ev6a}\\
 & \ +\left({\rm ord}_{z-a}(p_{k}-p_{k+1})-{\rm ord}_{z-a}(p_{k}-p_{0})\right)\mathrm{ev}_{a}^{\mathfrak{l}}\mathbb{I}(p_{0};p_{1},\dots,p_{k-1};p_{k}).\nonumber 
\end{align}
Now, the lemma readily follows by comparing (\ref{eq:motiv_ev6})
and (\ref{eq:motiv_ev6a}).
\end{proof}

\subsection{The definition of the confluence relation}
\begin{defn}
\label{def:diff_operator_formal}For $c\in S$, define a linear map
$\partial_{c}:\mathscr{B}\to\mathscr{B}$ by
\[
\partial_{c}(uv)=\partial_{c}(u)v+u\partial_{c}(v)
\]
and
\[
\partial_{c}(\mathbb{I}_{\gamma}(p_{0};p_{1},\dots,p_{k};p_{k+1}))=\sum_{i=1}^{k}\left({\rm ord}_{z-c}(p_{i}-p_{i+1})-{\rm ord}_{z-c}(p_{i}-p_{i-1})\right)\mathbb{I}_{\gamma}(p_{0};p_{1},\dots,\widehat{p_{i}},\dots,p_{k};p_{k+1}).
\]
\end{defn}

\begin{defn}
For $x,y\in\tilde{S}$, we define $\psi_{x,y}:\mathscr{B}\to\mathcal{H}$
by
\[
\psi_{x,y}(u)=-\evallim y(u)+\sum_{l=0}^{\infty}\sum_{c_{1},\dots,c_{l}\in S}\evallim x(\partial_{c_{1}}\cdots\partial_{c_{l}}u)\cdot I^{\mathfrak{m}}(x;c_{1},\dots,c_{l};y).
\]
\end{defn}

\begin{prop}[Confluence relation for complex-valued iterated integrals]
For $x,y\in\tilde{S}$ and $u\in\mathscr{B}$,
\[
\compmap(\psi_{x,y}(u)))=0.
\]
\end{prop}

This theorem is proved just in the same way as the proof of Theorem
\ref{thm:Confluecne_relations_for_Euler_sums} by using the general
differential formula
\begin{align*}
 & \frac{d}{dz}I_{\gamma}(\widehat{p_{0}};p_{1},\dots,p_{k};\widehat{p_{k+1}})\\
= & \sum_{\substack{1\leq i\leq k\\
p_{i}\neq p_{i+1}
}
}\frac{d\log(p_{i}-p_{i+1})}{dz}I_{\gamma}(\widehat{p_{0}};p_{1},\dots,p_{i-1},p_{i+1},\dots,p_{k};\widehat{p_{k+1}})\\
 & -\sum_{\substack{1\leq i\leq k\\
p_{i}\neq p_{i-1}
}
}\frac{d\log(p_{i}-p_{i-1})}{dz}I_{\gamma}(\widehat{p_{0}};p_{1},\dots,p_{i-1},p_{i+1},\dots,p_{k};\widehat{p_{k+1}})
\end{align*}
 for $p_{0},\dots,p_{k+1}\in\mathcal{P}$. 

The purpose of Section \ref{sec:Motivicity-of-confluence} is to prove
$\psi_{x,y}(u)=0$ for any $x,y$ and $u$. We will prove this by
induction on the degree of $u$. For convenience, we call the statement:
\[
\left[\text{For \ensuremath{x,y\in\tilde{S}}, \ensuremath{l\leq k} and \ensuremath{u\in\mathscr{B}_{l}}, }\psi_{x,y}(u)=0.\right]
\]
as ${\rm Hyp}(k)$.

\subsection{A combinatorial interpretation for $\psi_{x,y}$}

Fix a finite totally-ordered set $P$ and a map $\iota:P\to\mathcal{P}$.
We identify a strictly increasing sequence $a_{1}<\cdots<a_{k}$ of
the elements in $P$ and a subset $\{a_{1},\dots,a_{k}\}$ of $P$.
For $k\geq0$, we put
\[
B_{k}:=\{U\subset P\mid\#U=2+k\}
\]
and $B=\bigcup_{k=0}^{\infty}B_{k}$. For $S_{1},S_{2}\in B$, we
write $S_{1}\subweak S_{2}$ (resp. $S_{1}\subweakneq S_{2}$) if
and only if $S_{1}\subset S_{2}$ (resp. $S_{1}\subsetneq S_{2}$),
$\min(S_{1})=\min(S_{2})$ and $\max(S_{1})=\max(S_{2})$.
\begin{defn}
Let $U_{0},\dots,U_{m}\in B$, and $(p_{1},p_{1}'),\dots,(p_{m},p_{m}')\in P^{2}$.
For $S_{1}\subweak S_{2}$, we call the pair 
\[
((U_{0},\dots,U_{m});(p_{1},p_{1}'),\dots,(p_{m},p_{m}'))
\]
 a \emph{history }from $S_{1}$ to $S_{2}$ if the following conditions
are satisfied:
\begin{itemize}
\item $p_{i}<p_{i}'$ for all $1\leq i\leq m$,
\item $S_{1}=U_{0}\subset\cdots\subset U_{m}=S_{2}$,
\item $\#U_{i}=\#U_{i-1}+1$ for all $1\leq i\leq m$,
\item $U_{i}=U_{i-1}\cup\{p_{i},p_{i}'\}$ for all $1\leq i\leq m$,
\item $q\notin U_{i}$ for all $q$ such that $p_{i}<q<p_{i}'$.
\end{itemize}
\end{defn}

We write a history $((U_{0},\dots,U_{m});(p_{1},p_{1}'),\dots,(p_{m},p_{m}'))$
as
\[
U_{0}\xrightarrow{p_{1}p_{1}'}U_{1}\xrightarrow{p_{2}p_{2}'}\cdots\xrightarrow{p_{m}p_{m}'}U_{m}.
\]

\begin{defn}
For $S_{1}\subweak S_{2}$, we denote by $\history{S_{1}}{S_{2}}$
the set of histories from $S_{1}$ to $S_{2}$. Furthermore, if $\#S_{1}=2$,
we write $\monohistory{S_{2}}$ for $\history{S_{1}}{S_{2}}$.
\end{defn}

\begin{example}
Let $P=\{0,1,2,3\}$. We denote a subset $\{a_{1},\dots,a_{k}\}$
of $P$ simply as $a_{1}\dots a_{k}$. Then $\monohistory{0123}$
consists of the following eight elements:
\[
\begin{split}03\xrightarrow{01}013\xrightarrow{12}0123,\ 03\xrightarrow{01}013\xrightarrow{23}0123,\ 03\xrightarrow{13}013\xrightarrow{12}0123,\ 03\xrightarrow{13}013\xrightarrow{23}0123\\
03\xrightarrow{02}023\xrightarrow{01}0123,\ 03\xrightarrow{02}023\xrightarrow{12}0123,\ 03\xrightarrow{23}023\xrightarrow{01}0123,\ 03\xrightarrow{23}023\xrightarrow{12}0123.
\end{split}
\]
\end{example}

For $U=\{b_{0}<\cdots<b_{k+1}\}\in B_{k}$ and $\gamma\in\pi(\iota(b_{0}),\iota(b_{k+1}))$,
we put
\begin{align*}
\mathbb{I}_{\gamma}(U) & :=\formalit{\gamma}{\iota(b_{0})}{\iota(b_{1}),\dots,\iota(b_{k})}{\iota(b_{k+1})}\in\mathscr{B}_{k},\\
\mathbb{I}(U) & :=\mathbb{I}(\iota(b_{0});\iota(b_{1}),\dots,\iota(b_{k});\iota(b_{k+1}))\in\mathscr{B}_{k}^{\mathfrak{a}}.
\end{align*}

\begin{defn}
For a history
\[
f=\left(U_{0}\xrightarrow{p_{1}p_{1}'}U_{1}\xrightarrow{p_{2}p_{2}'}\cdots\xrightarrow{p_{m}p_{m}'}U_{m}\right)\in\history{U_{0}}{U_{m}},
\]
define the sign of $f$ as
\[
{\rm sgn}(f)\coloneqq(-1)^{t(f)}
\]
where $t(f)$ is the number of $j$ such that $U_{j}=U_{j-1}\cup\{p_{j}'\}$.
\end{defn}

\begin{defn}
For a history
\[
f=\left(U_{0}\xrightarrow{p_{1}p_{1}'}U_{1}\xrightarrow{p_{2}p_{2}'}\cdots\xrightarrow{p_{m}p_{m}'}U_{m}\right)\in\history{U_{0}}{U_{m}},
\]
and $x,y\in\tilde{S}$, define $T_{x,y}(f)\in\mathcal{H}$ by
\[
T_{x,y}(f)\coloneqq{\rm sgn}(f)\sum_{c_{1},\dots,c_{m}\in S}\left(\prod_{i=1}^{m}{\rm ord}_{z-c_{i}}\left(\iota(p_{i})-\iota(p_{i}')\right)\right)I^{\mathfrak{m}}(x;c_{1},\dots,c_{m};y).
\]
\end{defn}

Then by definition we have
\begin{equation}
\psi_{x,y}(\mathbb{I}_{\gamma}(U))=-\evallim y(\mathbb{I}_{\gamma}(U))+\sum_{\substack{V\subweak U}
}\sum_{f\in\history VU}\evallim x(\mathbb{I}_{\gamma}(V))T_{x,y}(f).\label{eq:motiv_combphi}
\end{equation}

\subsection{A combinatorial interpretation of $D\left(\psi_{x,y}\right)$}

Let
\[
f\coloneqq\left(U_{0}\xrightarrow{p_{1}p_{1}'}U_{1}\xrightarrow{p_{2}p_{2}'}\cdots\xrightarrow{p_{m}p_{m}'}U_{m}\right)
\]
be a history. Assume that $m>0$. Then
\begin{align}
D(T_{x,y}(f)) & =\sum_{0\leq i<j\leq m}{\rm sgn}(f)\sum_{c_{1},\dots,c_{m}\in S}\left(\prod_{s=1}^{m}{\rm ord}_{z-c_{s}}\left(\iota(p_{s})-\iota(p_{s}')\right)\right)\label{eq:motiv_d1}\\
 & \ \ \ \ \ \times I^{\mathfrak{l}}(c_{i};c_{i+1},\dots,c_{j};c_{j+1})\otimes I^{\mathfrak{m}}(x;c_{1},\dots,c_{i},c_{j+1},\dots,c_{m};y)\nonumber 
\end{align}
where we put $c_{0}\coloneqq x$ and $c_{m+1}\coloneqq y$. We denote
by $T_{x,y}^{\mathfrak{l}}(f)$ the image of $T_{x,y}(f)$ in $\mathfrak{L}$.
For (\ref{eq:motiv_d1}), we introduce the following notation.
\begin{defn}
For three histories
\[
f_{1}=\left(V\xrightarrow{p_{1}p_{1}'}\cdots\xrightarrow{p_{i}p_{i}'}W_{1}\right)\in\history V{W_{1}},\ f_{2}\in\history{W_{1}}{W_{2}},\ f_{3}=\left(W_{2}\xrightarrow{p_{i+1}p_{i+1}'}\cdots\xrightarrow{p_{m}p_{m}'}U\right)\in\history{W_{2}}U,
\]
we define $T_{x,y}(f_{1},f_{2},f_{3})\in\mathfrak{L}\otimes\mathcal{H}$
by
\[
T_{x,y}(f_{1},f_{2},f_{3})\coloneqq{\rm sgn}(f_{1}){\rm sgn}(f_{3})\sum_{c_{1},\dots,c_{m}\in S}\left(\prod_{s=1}^{m}{\rm ord}_{z-c_{s}}\left(\iota(p_{s})-\iota(p_{s}')\right)\right)T_{c_{i},c_{i+1}}^{\mathfrak{l}}(f_{2})\otimes I^{\mathfrak{m}}(x;c_{1},\dots,c_{m};y)
\]
where we again put $c_{0}\coloneqq x$ and $c_{m+1}\coloneqq y$.
\end{defn}

Then, by definition, we have
\begin{align}
D(\sum_{f\in\history VU}T_{x,y}(f)) & =\sum_{V\subweak W_{1}\subweakneq W_{2}\subweak U}\sum_{\substack{f_{1}\in\history V{W_{1}}\\
f_{2}\in\history{W_{1}}{W_{2}}\\
f_{3}\in\history{W_{2}}U
}
}T_{x,y}(f_{1},f_{2},f_{3})\nonumber \\
 & =\sum_{V\subweak W_{1}\subweakneq W_{2}\subweak U}\sum_{\substack{f_{1}\in\history V{W_{1}}\\
f_{3}\in\history{W_{2}}U
}
}T_{x,y}(f_{1},f_{3})\label{eq:motiv_e1}
\end{align}
where we put
\[
T_{x,y}(f_{1},f_{3})\coloneqq\sum_{f_{2}\in\history{W_{1}}{W_{2}}}T_{x,y}(f_{1},f_{2},f_{3}).
\]
By (\ref{eq:motiv_combphi}) and (\ref{eq:motiv_e1}), we have
\begin{align}
D\left(\psi_{x,y}(\mathbb{I}_{\gamma}(U))\right) & =-D\left(\evallim y(\mathbb{I}_{\gamma}(U))\right)+\sum_{\substack{V\subweak U}
}\sum_{f\in\history VU}\left(1\otimes\evallim x(\mathbb{I}_{\gamma}(V))\right)\cdot D\left(T_{x,y}(f)\right)\nonumber \\
 & \ \ +\sum_{\substack{V\subweak U}
}\sum_{f\in\history VU}\left(1\otimes T_{x,y}(f)\right)\cdot D\left(\evallim x(\mathbb{I}_{\gamma}(V))\right)\nonumber \\
 & =-D\left(\evallim y(\mathbb{I}_{\gamma}(U))\right)+\sum_{V\subweak W_{1}\subweakneq W_{2}\subweak U}\sum_{\substack{f\in\history V{W_{1}}\\
g\in\history{W_{2}}U
}
}\left(1\otimes\evallim x(\mathbb{I}_{\gamma}(V))\right)\cdot T_{x,y}(f,g)\nonumber \\
 & \ \ +\sum_{\substack{V\subweak U}
}\sum_{f\in\history VU}\left(1\otimes T_{x,y}(f)\right)\cdot D\left(\evallim x(\mathbb{I}_{\gamma}(V))\right)\label{eq:motiv_e2}
\end{align}
Now we calculate $T_{x,y}(f,g)$.
\begin{defn}
Assume that $V\subweakneq W$. We write $V\substrong W$ if $V$ and
$W$ can be written as
\[
V=\{a_{0}<\cdots<a_{i}<a_{i+1}<\cdots<a_{k+1}\}
\]
and
\[
W=\{a_{0}<\cdots<a_{i}<b_{1}<\cdots<b_{j}<a_{i+1}<\cdots<a_{k+1}\}
\]
for some $0\leq i\leq k$ and $j>0$ and $a_{0},\ldots,a_{k+1},b_{1},\ldots,b_{j}\in P$.
Furthermore, we define $\omega(V,W)$ to be the substring
\[
\omega(V,W)\coloneqq\{a_{i}<b_{1}<\cdots<b_{j}<a_{i+1}\}
\]
of $W$ in this case.
\end{defn}

\begin{lem}
\label{lem:motiv_Tsum}Let 
\begin{align*}
V & =\{a_{1}<\cdots<a_{n+1}\},\\
U & =\{a_{1}<b_{1,1}<\cdots<b_{1,r_{1}}<a_{2}<b_{2,1}<\cdots<b_{2,r_{2}}<\cdots<a_{n}<b_{n,1}<\cdots<b_{n,r_{n}}<a_{n+1}\}
\end{align*}
and
\[
U_{i}=\{a_{i}<b_{i,1}<\cdots<b_{i,r_{i}}<a_{i+1}\}\quad(1\leq i\leq n)
\]
where $a_{1},\ldots,a_{n+1},b_{1,1},\ldots,b_{1,r_{1}},b_{2,1},\ldots,b_{2,r_{2}},\ldots,b_{n,1},\ldots,b_{n,r_{n}}\in P$.
Then, for $x,y\in\tilde{S}$, 
\[
\sum_{f\in\history VU}T_{x,y}(f)=\prod_{i=1}^{n}\left(\sum_{g\in\monohistory{U_{i}}}T_{x,y}(g)\right).
\]
\end{lem}

\begin{proof}
It follows from the shuffle product formula and the definition of
history.
\end{proof}
\begin{lem}
\label{lem:motiv_Tlsum}Suppose that ${\rm Hyp}(k-1)$ holds. Then
for any $V\subweakneq U$ such that $\#U-\#V\leq k-1$, we have
\[
\sum_{f\in\history VU}T_{x,y}^{\mathfrak{l}}(f)=\begin{cases}
\mathrm{ev}_{y}^{\mathfrak{l}}(\mathbb{I}(\omega(V,U)))-\mathrm{ev}_{x}^{\mathfrak{l}}(\mathbb{I}(\omega(V,U))) & V\substrong U\\
0 & {\rm otherwise}.
\end{cases}
\]
\end{lem}

\begin{proof}
By Lemma \ref{lem:motiv_Tsum}, $\sum_{f\in\history VU}T_{x,y}(f)$
vanishes if there exists $1\leq i<j\leq n$ such that $r_{i},r_{j}>0$.
Thus, we have
\[
\sum_{f\in\history VU}T_{x,y}^{\mathfrak{l}}(f)=\begin{cases}
\sum_{g\in\monohistory{\omega(V,U)}}T_{x,y}^{\mathfrak{l}}(g) & V\substrong U\\
0 & {\rm otherwise}.
\end{cases}
\]
Since $\#U-\#V\leq k-1$,
\[
\psi_{x,y}(\mathbb{I}_{\gamma}(\omega(V,U)))=0
\]
by ${\rm Hyp}(k-1)$, and thus we have
\begin{align*}
0 & =-\mathrm{ev}_{y}^{\mathfrak{l}}(\mathbb{I}(\omega(V,U)))+\sum_{\substack{W\subweak\omega(V,U)}
}\sum_{f\in\monohistory W}\mathrm{ev}_{x}^{\mathfrak{l}}(\mathbb{I}(W))T_{x,y}^{\mathfrak{l}}(f)\\
 & =-\mathrm{ev}_{y}^{\mathfrak{l}}(\mathbb{I}(\omega(V,U)))+\sum_{\substack{\substack{W\subweak\omega(V,U)\\
\#W\text{ is }2\text{ or }\#\omega(V,U)
}
}
}\sum_{f\in\monohistory W}\mathrm{ev}_{x}^{\mathfrak{l}}(\mathbb{I}(W))T_{x,y}^{\mathfrak{l}}(f)\\
 & =-\mathrm{ev}_{y}^{\mathfrak{l}}(\mathbb{I}(\omega(V,U)))+\mathrm{ev}_{x}^{\mathfrak{l}}(\mathbb{I}(\omega(V,U)))+\sum_{f\in\monohistory{\omega(V,U)}}T_{x,y}^{\mathfrak{l}}(f).
\end{align*}
Hence
\[
\sum_{g\in\monohistory{\omega(V,U)}}T_{x,y}^{\mathfrak{l}}(g)=\mathrm{ev}_{y}^{\mathfrak{l}}(\mathbb{I}(\omega(V,U)))-\mathrm{ev}_{y}^{\mathfrak{l}}(\mathbb{I}(\omega(V,U)))
\]
as desired.
\end{proof}
\begin{defn}
Let $V\subweak U$. We call a pair of histories a \emph{divided-history
}from $V$ to $U$ if it is of the form $(f,g)$ where $f\in\history V{W_{1}}$,
$g\in\history{W_{2}}U$ with some $V\subweak W_{1}\substrong W_{2}\subweak U$\emph{.}
\end{defn}

We express a divided-history $(f,g)$ as
\[
V\xrightarrow{p_{1}p_{1}'}\cdots\xrightarrow{p_{r}p_{r}'}W_{1}\Rightarrow W_{2}\xrightarrow{q_{1}q_{1}'}\cdots\xrightarrow{q_{s}q_{s}'}U
\]
where
\[
f=\left(V\xrightarrow{p_{1}p_{1}'}\cdots\xrightarrow{p_{r}p_{r}'}W_{1}\right),\ g=\left(W_{2}\xrightarrow{q_{1}q_{1}'}\cdots\xrightarrow{q_{s}q_{s}'}U\right).
\]
We denote by $\jhistory VU$ the set of divided-histories from $V$
to $U$.
\begin{defn}
For a divided-history
\[
(f,g)\coloneqq\left(V\xrightarrow{p_{1}p_{1}'}\cdots\xrightarrow{p_{i}p_{i}'}W_{1}\Rightarrow W_{2}\xrightarrow{p_{i+1}p_{i+1}'}\cdots\xrightarrow{p_{m}p_{m}'}U\right),
\]
define $T_{x,y}^{R}(f,g)$ and $T_{x,y}^{L}(f,g)$ by
\begin{align*}
T_{x,y}^{R}(f,g) & ={\rm sgn}(f){\rm sgn}(g)\sum_{c_{1},\dots,c_{m}\in S}\left(\prod_{s=1}^{m}{\rm ord}_{z-c_{s}}\left(\iota(p_{s})-\iota(p_{s}')\right)\right)\mathrm{ev}_{c_{i+1}}^{\mathfrak{l}}(\mathbb{I}(\omega(W_{1},W_{2})))\otimes I^{\mathfrak{m}}(x;c_{1},\dots,c_{m};y)\\
T_{x,y}^{L}(f,g) & ={\rm sgn}(f){\rm sgn}(g)\sum_{c_{1},\dots,c_{m}\in S}\left(\prod_{s=1}^{m}{\rm ord}_{z-c_{s}}\left(\iota(p_{s})-\iota(p_{s}')\right)\right)\mathrm{ev}_{c_{i}}^{\mathfrak{l}}(\mathbb{I}(\omega(W_{1},W_{2})))\otimes I^{\mathfrak{m}}(x;c_{1},\dots,c_{m};y)
\end{align*}
where we put $c_{0}\coloneqq x$ and $c_{m+1}\coloneqq y$.
\end{defn}

Now suppose that ${\rm Hyp}(k-1)$ holds and that $U\in\mathscr{B}_{k}$.
Then by Lemma \ref{lem:motiv_Tlsum}, for $V\subweak W_{1}\subweakneq W_{2}\subweak U\in\mathscr{B}_{k}$,
$f\in\history V{W_{1}}$ and $g\in\history{W_{2}}U$, we have
\[
T_{x,y}(f,g)\equiv\begin{cases}
T_{x,y}^{R}(f,g)-T{}_{x,y}^{L}(f,g) & W_{1}\substrong W_{2}\\
0 & W_{1}\not\ntriangleleft W_{2}
\end{cases}
\]
except for the case where $\#W_{1}=2$ and $W_{2}=U$. Thus, we have
\begin{multline}
\sum_{V\subweak W_{1}\subweakneq W_{2}\subweak U}\sum_{\substack{f\in\history V{W_{1}}\\
g\in\history{W_{2}}U
}
}\left(1\otimes\evallim x(\mathbb{I}_{\gamma}(V))\right)\cdot T_{x,y}(f,g)\\
\equiv\sum_{V\subweakneq U}\left(1\otimes\evallim x(\mathbb{I}_{\gamma}(V))\right)\cdot\sum_{(f,g)\in\jhistory VU}\left(T_{x,y}^{R}(f,g)-T_{x,y}^{L}(f,g)\right)\pmod{\mathfrak{L}_{k}\otimes\mathbb{Q}}.\label{eq:motiv_e3}
\end{multline}
Now, we decompose $\jhistory VU$ in two ways:
\[
\jhistory VU=\jexthistory R{(1)}VU\sqcup\jexthistory R{(2)}VU\sqcup\jexthistory R{(3)}VU=\jexthistory L{(1)}VU\sqcup\jexthistory L{(2)}VU\sqcup\jexthistory L{(3)}VU
\]
where $\jexthistory R{(1)}VU$ the subset of $\jhistory VU$ consisting
of the elements of the forms
\[
V\to\cdots\to W\Rightarrow U,
\]
$\jexthistory R{(2)}VU$ the subset of $\jhistory VU$ consisting
of the elements of the forms
\[
\cdots\to\left(Aa_{0}a_{k+1}B\right)\Rightarrow\left(Aa_{0}a_{1}\dots a_{i-1}a_{i+1}\dots a_{k}a_{k+1}B\right)\to\left(Aa_{0}a_{1}\dots a_{k}a_{k+1}B\right)\to\cdots\ \ \ \ (1\leq i\leq k),
\]
$\jexthistory R{(3)}VU$ the subset of $\jhistory VU$ consisting
of the elements of the forms
\[
\cdots\to\left(A_{1}A_{2}abB\right)\Rightarrow\left(A_{1}A_{2}aWbB\right)\to\left(A_{1}xA_{2}aWbB\right)\to\cdots
\]
or
\[
\cdots\to\left(AabB_{1}B_{2}\right)\Rightarrow\left(AaWbB_{1}B_{2}\right)\to\left(AaWbB_{1}xB_{2}\right)\to\cdots,
\]
$\jexthistory L{(1)}VU$ the subset of $\jhistory VU$ consisting
of the elements of the form
\[
V\Rightarrow W\to\cdots\to U,
\]
$\jexthistory L{(2)}VU$ the subset of $\jhistory VU$ consisting
of the elements of the form
\[
\cdots\to\left(Aa_{0}a_{k+1}B\right)\to\left(Aa_{0}a_{i}a_{k+1}B\right)\Rightarrow\left(Aa_{0}a_{1}\dots a_{k}a_{k+1}B\right)\to\cdots\ \ \ \ (k\geq1,i\in\{1,k\}),
\]
$\jexthistory L{(3)}VU$ the subset of $\jhistory VU$ consisting
of the elements of the form
\[
\cdots\to\left(A_{1}A_{2}abB\right)\to\left(A_{1}xA_{2}abB\right)\Rightarrow\left(A_{1}xA_{2}aWbB\right)\to\cdots
\]
or
\[
\cdots\to\left(AabB_{1}B_{2}\right)\to\left(AabB_{1}xB_{2}\right)\Rightarrow\left(AaWbB_{1}xB_{2}\right)\to\cdots.
\]
By Lemma \ref{lem:motiv_copro2} and ${\rm Hyp}(k-1)$, we have
\begin{align}
 & \sum_{V\subweakneq U}\left(1\otimes\evallim x(\mathbb{I}_{\gamma}(V))\right)\cdot\sum_{(f,g)\in\jexthistory R{(1)}VU}T_{x,y}^{R}(f,g)\nonumber \\
 & =\sum_{V\subweakneq U}\left(1\otimes\evallim x(\mathbb{I}_{\gamma}(V))\right)\cdot\left(\sum_{V\subweak W\substrong U}\mathrm{ev}_{y}^{\mathfrak{l}}(\mathbb{I}(\omega(W,U)))\otimes\sum_{f\in\history VW}T_{x,y}(f)\right)\nonumber \\
 & =\sum_{W\substrong U}\mathrm{ev}_{y}^{\mathfrak{l}}(\mathbb{I}(\omega(W,U)))\otimes\left(\sum_{V\subweak W}\evallim x(\mathbb{I}_{\gamma}(V))\sum_{f\in\history VW}T_{x,y}(f)\right)\nonumber \\
 & =\sum_{W\substrong U}\mathrm{ev}_{y}^{\mathfrak{l}}(\mathbb{I}(\omega(W,U)))\otimes\left(\psi_{x,y}(\mathbb{I}_{\gamma}(W))+\evallim y(\mathbb{I}_{\gamma}(W))\right)\nonumber \\
 & =\sum_{W\substrong U}\mathrm{ev}_{y}^{\mathfrak{l}}(\mathbb{I}(\omega(W,U)))\otimes\evallim y(\mathbb{I}_{\gamma}(W))\nonumber \\
 & =D\left(\evallim y(\mathbb{I}_{\gamma}(U))\right)\label{eq:motiv_4a}
\end{align}
and similarly
\begin{align}
 & \sum_{V\subweakneq U}\left(1\otimes\evallim x(\mathbb{I}_{\gamma}(V))\right)\cdot\sum_{(f,g)\in\jexthistory L{(1)}VU}T_{x,y}^{L}(f,g)\nonumber \\
 & =\sum_{V\subweakneq U}\left(1\otimes\evallim x(\mathbb{I}_{\gamma}(V))\right)\cdot\left(\sum_{V\substrong W\subweak U}\mathrm{ev}_{x}^{\mathfrak{l}}(\mathbb{I}(\omega(V,W)))\otimes\sum_{f\in\history WU}T_{x,y}(f)\right)\nonumber \\
 & =\sum_{W\subweak U}\left(1\otimes\sum_{f\in\history WU}T_{x,y}(f)\right)\cdot\left(\sum_{V\substrong W}\mathrm{ev}_{x}^{\mathfrak{l}}(\mathbb{I}(\omega(V,W)))\otimes\evallim x(\mathbb{I}_{\gamma}(V))\right)\nonumber \\
 & =\sum_{W\subweak U}\left(1\otimes\sum_{f\in\history WU}T_{x,y}(f)\right)\cdot D(\evallim x(\mathbb{I}_{\gamma}(W)).\label{eq:motiv_4b}
\end{align}
Next by Lemma \ref{lem:motiv_copro3}, we have
\begin{align*}
 & \sum_{i=1}^{k}T_{x,y}^{R}(\cdots\to\left(Aa_{0}a_{k+1}B\right)\Rightarrow\left(Aa_{0}a_{1}\dots a_{i-1}a_{i+1}\dots a_{k}a_{k+1}B\right)\xrightarrow{a_{i}a_{i+1}}\left(Aa_{0}a_{1}\dots a_{k}a_{k+1}B\right)\to\cdots)\\
 & +\sum_{i=1}^{k}T_{x,y}^{R}(\cdots\to\left(Aa_{0}a_{k+1}B\right)\Rightarrow\left(Aa_{0}a_{1}\dots a_{i-1}a_{i+1}\dots a_{k}a_{k+1}B\right)\xrightarrow{a_{i-1}a_{i}}\left(Aa_{0}a_{1}\dots a_{k}a_{k+1}B\right)\to\cdots)\\
= & \sum_{i\in\{1,k\}}T_{x,y}^{L}(\cdots\to\left(Aa_{0}a_{k+1}B\right)\xrightarrow{a_{i}a_{k+1}}\left(Aa_{0}a_{i}a_{k+1}B\right)\Rightarrow\left(Aa_{0}a_{1}\dots a_{k}a_{k+1}B\right)\to\cdots)\\
 & +\sum_{i\in\{1,k\}}T_{x,y}^{L}(\cdots\to\left(Aa_{0}a_{k+1}B\right)\xrightarrow{a_{0}a_{i}}\left(Aa_{0}a_{i}a_{k+1}B\right)\Rightarrow\left(Aa_{0}a_{1}\dots a_{k}a_{k+1}B\right)\to\cdots)
\end{align*}
and thus
\begin{equation}
\sum_{V\subweakneq U}\left(1\otimes\evallim x(\mathbb{I}_{\gamma}(V))\right)\cdot\sum_{(f,g)\in\jexthistory R{(2)}VU}T_{x,y}^{R}(f,g)=\sum_{V\subweakneq U}\left(1\otimes\evallim x(\mathbb{I}_{\gamma}(V))\right)\cdot\sum_{(f,g)\in\jexthistory L{(2)}VU}T_{x,y}^{L}(f,g).\label{eq:motiv_4c}
\end{equation}
Furthermore, since
\begin{align*}
 & T_{x,y}^{R}\left(\cdots\to\left(A_{1}A_{2}abB\right)\Rightarrow\left(A_{1}A_{2}aWbB\right)\to\left(A_{1}xA_{2}aWbB\right)\to\cdots\right)\\
 & =T_{x,y}^{L}\left(\cdots\to\left(A_{1}A_{2}abB\right)\to\left(A_{1}xA_{2}abB\right)\Rightarrow\left(A_{1}xA_{2}aWbB\right)\to\cdots\right)
\end{align*}
and
\begin{align*}
 & T_{x,y}^{R}\left(\cdots\to\left(AabB_{1}B_{2}\right)\Rightarrow\left(AaWbB_{1}B_{2}\right)\to\left(AaWbB_{1}xB_{2}\right)\to\cdots\right)\\
 & =T_{x,y}^{L}(\cdots\to\left(AabB_{1}B_{2}\right)\to\left(AabB_{1}xB_{2}\right)\Rightarrow\left(AaWbB_{1}xB_{2}\right)\to\cdots),
\end{align*}
we have
\begin{equation}
\sum_{V\subweakneq U}\left(1\otimes\evallim x(\mathbb{I}_{\gamma}(V))\right)\cdot\sum_{(f,g)\in\jexthistory R{(3)}VU}T_{x,y}^{R}(f,g)=\sum_{V\subweakneq U}\left(1\otimes\evallim x(\mathbb{I}_{\gamma}(V))\right)\cdot\sum_{(f,g)\in\jexthistory L{(3)}VU}T_{x,y}^{L}(f,g).\label{eq:motiv_4d}
\end{equation}
By (\ref{eq:motiv_e2}), (\ref{eq:motiv_e3}), (\ref{eq:motiv_4a}),
(\ref{eq:motiv_4b}), (\ref{eq:motiv_4d}) and (\ref{eq:motiv_4c}),
we have
\begin{equation}
D\left(\psi_{x,y}(\mathbb{I}_{\gamma}(U))\right)\equiv0\pmod{\mathcal{L}_{k}\otimes\mathbb{Q}}\label{eq:motiv_key}
\end{equation}
 for $U\in B_{k}$ under the hypothesis $H(k-1)$. Now we are ready
to prove the following Theorem.
\begin{thm}[Confluence relation for motivic iterated integrals]
\label{thm:motivicity}For $x,y\in\tilde{S}$ and $u\in\mathscr{B}$,
we have
\[
\psi_{x,y}(u)=0.
\]
\end{thm}

\begin{proof}
Assume that $U\in B_{k}$. We prove $\psi_{x,y}(\mathbb{I}_{\gamma}(U))=0.$
by induction on $k$. By definition, we have
\[
\compmap(\psi_{x,y}(\mathbb{I}_{\gamma}(U)))=0.
\]
By (\ref{eq:motiv_key}) and the induction hypothesis, we have
\[
D(\psi_{x,y}(\mathbb{I}_{\gamma}(U)))\equiv0\pmod{\mathcal{L}_{k}\otimes\mathbb{Q}}.
\]
Thus by Lemma \ref{lem:Brown}, we have
\[
\psi_{x,y}(\mathbb{I}_{\gamma}(U))=0.
\]
which completes the proof.
\end{proof}
\begin{rem}
Theorem \ref{thm:motivicity} can be viewed as an ultimate generalization
of the confluence relation defined in the authors' previous article
\cite{HS_confluence}. In fact, the confluence relation in \cite{HS_confluence}
is obtained from Theorem \ref{thm:motivicity} under the setting $\mathcal{P}\coloneqq\{0,1,z\}$
and $S\coloneqq\{0,1\}$ together with the tangential vectors $v(0)=1,$
$v(1)=-1$.
\end{rem}

\section{\label{sec:Proof-of-Theorem}Proofs of the main theorems}

In this section, we first prove the motivicity of the confluence relations
for Euler sums discussed in Part 1 i.e. $\iconf\subset\ker(L^{\mathfrak{m}})$
by using the general result in the previous section. We then derive
the main theorems by combining this motivicity result and the results
in Part 1.

For the proof of $\iconf\subset\ker(L^{\mathfrak{m}})$, we apply
Theorem \ref{thm:motivicity} under the setting $\mathcal{P}\coloneqq\{0,-1,z,-z^{2}\}$
and $S\coloneqq\{0,1,-1\}$ where we set the tangential vectors as
$v(0)=1,$ $v(1)=-1$ and $v(-1)\in\mathbb{Q}^{\times}$ arbitrary.
Define a linear map $\mathfrak{i}:\mathcal{B}\to\mathscr{B}$ by $\mathfrak{i}(e_{p_{1}}\cdots e_{p_{k}})\coloneqq\mathbb{I}_{\gamma}(0;p_{1},\ldots,p_{k};z)$
where $\gamma$ is the straight path from $0$ to $z$.
\begin{lem}
\label{lem:Lm_equals_ev}For $w\in\mathcal{B}$, we have $L^{\mathfrak{m}}(\mathrm{reg}_{z\to0}(w))=\evallim 0(\mathfrak{i}(w))$.
\end{lem}

\begin{proof}
By definition of $\evallim 0$ (Definition \ref{def:ev}), for $p_{1},\dots,p_{k}\in\{0,-1,z,-z^{2}\}$
with $p_{1}\neq0$ and $p_{k}\neq z$, we have
\begin{align*}
\evallim 0(\mathfrak{i}(e_{p_{1}}\cdots e_{p_{k}})) & =\evallim 0(\mathbb{I}_{\gamma}(0;p_{1},\ldots,p_{k};z))\\
 & =\restr{J_{\gamma}(0;p_{1},\dots,p_{k};z)}{T=0}.
\end{align*}
By the definition of $J_{\gamma}(;;)$ (see Section \ref{subsec:limit-for-any-path}),
the last quantity is equal to
\[
J_{\mathrm{up}}(0;p_{1},\dots,p_{k};z).
\]
Thus the lemma is reduced to the equality
\[
L^{\mathfrak{m}}(\reg_{z\to0}(e_{p_{1}}\cdots e_{p_{k}}))=\restr{J_{\mathrm{up}}(0;p_{1},\dots,p_{k};z)}{T=0}.
\]
If $-1\in\{p_{1},\dots,p_{k}\}$, then the left hand side is zero
since $\mathrm{reg}_{z\to0}(e_{p_{1}}\cdots e_{p_{k}})=0$, while
the right hand side is also zero by Lemma \ref{lem:J_vanish}. Thus
we assume that $p_{1},\dots,p_{k}\in\{0,z,-z^{2}\}$. Furthermore,
since the both $L^{\mathfrak{m}}\circ\mathrm{reg}_{z\to0}$ and $\evallim 0\circ\mathfrak{i}$
satisfy the shuffle relation, it is enough to only consider the cases
\[
e_{p_{1}}\cdots e_{p_{k}}\in\mathcal{B}'',\ \text{i.e. }p_{1}=z
\]
and
\[
e_{p_{1}}\cdots e_{p_{k}}\in\mathcal{B}''',\ \text{i.e. }p_{1},\dots,p_{k}\in\{0,-z^{2}\}.
\]
For the first case, since the sequence $(0,\frac{p_{1}}{z},\dots,\frac{p_{k}}{z},1)$
is very regular, we have
\begin{align*}
J_{\mathrm{up}}(0;p_{1},\dots,p_{k};z) & =I_{{\rm up}}^{\mathfrak{m}}(0;(\frac{p_{1}}{z})[0],\dots,(\frac{p_{k}}{z})[0];1)\\
 & =L^{\mathfrak{m}}(\barreg_{z\to0}(e_{p_{1}}\cdots e_{p_{k}}))\\
 & =L^{\mathfrak{m}}(\dist(\barreg_{z\to0}(e_{p_{1}}\cdots e_{p_{k}})))\\
 & =L^{\mathfrak{m}}(\reg_{z\to0}(e_{p_{1}}\cdots e_{p_{k}})),
\end{align*}
and thus we are done. For the second case, since the sequence $(0,\frac{p_{1}}{z^{2}},\dots,\frac{p_{k}}{z^{2}};1/z)$
is very regular, we have
\[
J_{\mathrm{up}}(0;p_{1},\dots,p_{k};z)=I_{{\rm up}}^{\mathfrak{m}}(0;(\frac{p_{1}}{z^{2}})[0],\dots,(\frac{p_{k}}{z^{2}})[0];(\frac{1}{z})[0]),
\]
where $(\frac{1}{z})[0])$ is the extended tangential base point at
$\infty$ with the tangential vector $e^{T}\in\mathbb{Q}[[T]]$. Thus,
if we put $b_{j}=-\frac{p_{j}}{z^{2}}\in\{0,1\}$, then
\begin{align*}
\restr{J_{\mathrm{up}}(0;p_{1},\dots,p_{k};z)}{T=0} & =I_{{\rm up}}^{\mathfrak{m}}(0;b_{1},\dots,b_{k};\overrightarrow{-1}_{\infty})\\
 & =L^{\mathfrak{m}}(\shreg(\varrho(e_{b_{1}}\cdots e_{b_{k}})))
\end{align*}
where we have used the M\"{o}bius transformation$t\mapsto t/(t-1)$
for the last equality. Now that $L^{\mathfrak{m}}$ satisfies the
regularized double shuffle relation, the duality relation, and the
distribution relation, we can show
\[
L^{\mathfrak{m}}(\shreg(\varrho(e_{b_{1}}\cdots e_{b_{k}})))=L^{\mathfrak{m}}(\wp(e_{b_{1}}\cdots e_{b_{k}})),
\]
since the proof of Lemma \ref{lem:L1_shreg_varrho_is_L1_wp} only
uses those relations. Since
\[
L^{\mathfrak{m}}(\wp(e_{b_{1}}\cdots e_{b_{k}}))=L^{\mathfrak{m}}(\reg_{z\to0}(e_{p_{1}}\cdots e_{p_{k}}))
\]
by definition, this completes the proof.
\end{proof}
As a special case of Theorem \ref{thm:motivicity}, we obtain the
following theorem.
\begin{prop}
\label{prop:confluence_ker_L}We have
\[
\iconf\subset\ker(L^{\mathfrak{m}}).
\]
\end{prop}

\begin{proof}
We apply Theorem \ref{thm:motivicity} under the setting $\mathcal{P}\coloneqq\{0,-1,z,-z^{2}\}$
and $S\coloneqq\{0,1,-1\}$ where we set the tangential vectors as
$v(0)=1,$ $v(1)=-1$ and $v(-1)\in\mathbb{Q}^{\times}$ arbitrary.
Let $u=e_{p_{1}}\cdots e_{p_{k}}\in\mathcal{B}$ with $p_{1},\ldots,p_{k}\in\mathcal{P}$
such that $p_{1}\neq0$ and $p_{k}\neq0$. It is enough to show that
\[
L^{\mathfrak{m}}(\restr u{z\to1}-\varphi(u))=0.
\]
 Then
\[
L^{\mathfrak{m}}(\restr u{z\to1})=\mathrm{ev}_{1}^{\mathfrak{m}}\mathfrak{i}(u)
\]
and
\begin{align*}
L^{\mathfrak{m}}(\varphi(u)) & =\sum_{l=0}^{\infty}\sum_{c_{1},\dots,c_{l}\in S}L^{\mathfrak{m}}(\mathrm{reg}_{z\to0}(\partial_{c_{1}}\cdots\partial_{c_{l}}u))\cdot L^{\mathfrak{m}}(e_{c_{1}}\cdots e_{c_{l}})\\
 & =\sum_{l=0}^{\infty}\sum_{c_{1},\dots,c_{l}\in S}\evallim 0(\mathfrak{i}(\partial_{c_{1}}\cdots\partial_{c_{l}}u))\cdot L^{\mathfrak{m}}(e_{c_{1}}\cdots e_{c_{l}})\\
 & =\sum_{l=0}^{\infty}\sum_{c_{1},\dots,c_{l}\in S}\evallim 0(\partial_{c_{1}}\cdots\partial_{c_{l}}\mathfrak{i}(u))\cdot I^{\mathfrak{m}}(0;c_{1},\dots,c_{l};1)\\
 & =\psi_{0,1}(\mathfrak{i}(u))+\evallim 1(\mathfrak{i}(u)).
\end{align*}
Here, $\partial_{c}$ on the first and second lines are the one defined
in Definition \ref{def:diff_operator_formal} while that on the third
line is the one defined in Definition \ref{def:diff_operator_words}.
The second equality is by Lemma \ref{lem:Lm_equals_ev}, and the other
equalities are by definition. Hence, by Theorem \ref{thm:motivicity},
\[
L^{\mathfrak{m}}(\restr u{z\to1}-\varphi(u))=-\psi_{0,1}(\mathfrak{i}(u))=0
\]
as desired.
\end{proof}
\begin{thm}
\label{thm:main_conf-2}Let $\widehat{\mathcal{I}}_{\mathrm{CF}}$
be the ideal of $\mathfrak{H}^{(2)}\coloneqq(\mathbb{Q}\left\langle e_{0},e_{1},e_{-1}\right\rangle ,\shuffle)$
generated by $e_{0},e_{1}$ and $\iconf$. Then, we have
\[
\widehat{\mathcal{I}}_{\mathrm{CF}}=\ker(L^{\mathfrak{m}}).
\]
\end{thm}

\begin{proof}
Let $k$ be a non-negative integer and $\mathfrak{H}_{k}^{(2)}$ the
degree $k$ part of $\mathfrak{H}^{(2)}$. Let $F_{k}$ be the Fibonacci
number defined by $\sum_{k=0}^{\infty}F_{k}t^{k}=1/(1-t-t^{2})$.
Then by a result of Deligne \cite{Deli_es}, 
\begin{equation}
\dim_{\mathbb{Q}}\left(\mathfrak{H}_{k}^{(2)}/(\mathfrak{H}_{k}^{(2)}\cap\ker(L^{\mathfrak{m}}))\right)=F_{k},\label{eq:last_1}
\end{equation}
whereas by Theorem \ref{thm:main_main}, 
\begin{equation}
\dim_{\mathbb{Q}}\left(\mathfrak{H}_{k}^{(2)}/(\mathfrak{H}_{k}^{(2)}\cap\widehat{\mathcal{I}}_{\mathrm{CF}})\right)=\dim_{\mathbb{Q}}\left(\mathcal{A}_{k,\infty}^{0}(\{0,1,-1\})/\iconf\right)\leq\#\indset^{\mathrm{D}}(k,\infty)=F_{k}.\label{eq:last_2}
\end{equation}
Combining (\ref{eq:last_1}) and (\ref{eq:last_2}), we have
\begin{equation}
\dim_{\mathbb{Q}}\left(\mathfrak{H}_{k}^{(2)}\cap\widehat{\mathcal{I}}_{\mathrm{CF}}\right)\geq\dim_{\mathbb{Q}}\left(\mathfrak{H}_{k}^{(2)}\cap\ker(L^{\mathfrak{m}})\right).\label{eq:last_3}
\end{equation}
On the other hand, by Proposition \ref{prop:confluence_ker_L}, 
\begin{equation}
\mathfrak{H}_{k}^{(2)}\cap\widehat{\mathcal{I}}_{\mathrm{CF}}\subset\mathfrak{H}_{k}^{(2)}\cap\ker(L^{\mathfrak{m}}).\label{eq:last_4}
\end{equation}
By comparing (\ref{eq:last_3}) and (\ref{eq:last_4}), we find 
\[
\mathfrak{H}_{k}^{(2)}\cap\widehat{\mathcal{I}}_{\mathrm{CF}}=\mathfrak{H}_{k}^{(2)}\cap\ker(L^{\mathfrak{m}}).
\]
Thus, we can conclude that $\widehat{\mathcal{I}}_{\mathrm{CF}}=\ker(L^{\mathfrak{m}})$.
\end{proof}
Put $\zeta^{\mathfrak{m}}(\Bbbk)=(-1)^{d}L^{\mathfrak{m}}(\word(\Bbbk)),\ \modzeta^{\mathfrak{m}}(\Bbbk)=L^{\mathfrak{m}}(\modword(\Bbbk))$
where $\word(\Bbbk),\modword(\Bbbk)$ are as defined in Definition
\ref{def:index_word}. Then, by Proposition \ref{prop:confluence_ker_L},
we obtain the following refinement of Theorem Theorem \ref{thm:main_explicit}.
\begin{thm}
\label{thm:main_general_motivic_zeta}Let $k,d\geq0$. For $\Bbbk\in\indset(k,d)$,
$\modzeta^{\mathfrak{m}}(\Bbbk)$ is a $\mathbb{Z}_{(2)}$-linear
combination of $\left\{ \modzeta^{\mathfrak{m}}(\Bbbk')\mid\Bbbk'\in\indset^{\mathrm{D}}(k,d)\right\} $.
More explicitly, for $\Bbbk\in\indset(k,d)\setminus\indset^{\mathrm{D}}(k,d)$,
we have 
\[
\modzeta^{\mathfrak{m}}(\Bbbk)=\sum_{\Bbbk'\in\indset^{\mathrm{D}}(k,d)}\alpha_{\Bbbk,\Bbbk'}\modzeta^{\mathfrak{m}}(\Bbbk')
\]
where $\alpha_{\Bbbk,\Bbbk'}$'s are as defined in Theorem \ref{thm:main_explicit}.
\end{thm}

\subsection*{Acknowledgements}

This work was supported by JSPS KAKENHI Grant Numbers JP18J00982,
JP18K13392, JP19J00835, and JP20K14293.

\end{document}